\documentclass[11pt]{amsart} \textwidth=14.5cm \oddsidemargin=1cm
\usepackage[utf8]{inputenc}

\usepackage{multirow}
\usepackage{caption}
\usepackage{amsmath,amssymb,amscd,mathrsfs,amscd,wasysym,marvosym,color,xcolor}
\usepackage{setspace}
\usepackage{amsthm}
\usepackage[]{xcolor}
\usepackage[linktocpage=true]{hyperref}
\usepackage{mathrsfs}  
\usepackage{xypic,makecell,hhline}
\usepackage{tikz}
\usepackage{verbatim}
\usepackage{amsfonts}
\usepackage{amsxtra}
\usepackage{amssymb}
\usepackage{eucal}
\usepackage{latexsym,todonotes}
\usepackage{stmaryrd}
\usepackage{graphicx}
\usepackage{fontawesome5}
\usepackage{dynkin-diagrams}
\usepackage{tikz}
\usetikzlibrary{intersections}
\usetikzlibrary{shapes}
\newtheorem{theorem}{Theorem}[section]

\newtheorem{corollary}[theorem]{Corollary}

\newtheorem{definition}[theorem]{Definition}
\newtheorem{example}[theorem]{Example}

\newtheorem{lemma}[theorem]{Lemma}

\newtheorem{proposition}[theorem]{Proposition}

\numberwithin{equation}{section}

\theoremstyle{remark}
\newtheorem{remark}[theorem]{Remark}

\newcommand{\gl}{{\mathfrak{gl}}}
\newcommand{\g}{{\mathfrak{g}}}
\newcommand{\h}{{\mathfrak{h}}}

\newcommand{\n}{{\mathfrak{n}}}

\newcommand{\I}{ I}

\newcommand{\new}{\varrho}

\newcommand{\va}{\varsigma}

\newcommand{\ad}{\text{ad}}

\renewcommand{\t}{\mathfrak t}
\newcommand{\N}{\mathbb{N}}
\newcommand{\Z}{\mathbb{Z}}

\newcommand{\C}{\mathbb{C}}

\newcommand{\Fi}{\mathcal F_\io}

\newcommand{\Hy}{\mathscr{H}}

\newcommand{\End}{\text{End}}

\newcommand{\V}{\mathbb V}
\newcommand{\la}{\langle}
\newcommand{\ra}{\rangle}

\newcommand{\io}{\imath}
\newcommand{\bu}{\bullet}
\newcommand{\U}{\mathbf U}
\newcommand{\Ui}{{\mathbf U}^\imath}
\newcommand{\Uio}{{\mathbf U}^{\imath 0}}

\newcommand{\up}{\Upsilon}
\newcommand{\Veven}{\V_{\overline 0}}
\newcommand{\Vodd}{\V_{\overline 1}}

\newcommand{\Ieven}{\I_{\overline 0}}
\newcommand{\Iodd}{\I_{\overline 1}}

\newcommand{\Ub}{\U_{\bu}}
\newcommand{\qbinom}[2]{\begin{bmatrix} #1\\#2 \end{bmatrix} }

\newcommand{\lskew}{{}_ir}
\newcommand{\rskew}{r_i}

\newcommand{\blue}[1]{{\color{blue}#1}}

\newcommand{\zero}{{\bar{0}}}
\newcommand{\one}{{\bar{1}}}

\newcommand{\ns}{{\mathsf{ns}}}
\newcommand{\iso}{{\mathsf{iso}}}
\newcommand{\niso}{{\mathsf{n}\text{-}\mathsf{iso}}}

\newcommand{\yy}{\mbox{\Yingyang}
}

\renewcommand{\ad}{{\mathrm{ad}}}
\newcommand{\Ad}{{\mathrm{Ad}}}

\newcommand{\af}{\alpha}

\newcommand{\om}{\omega}

\newcommand{\Qy}{e}
\newcommand{\B}{\mathfrak B}

\newcommand{\set}[1]{\left\{#1\right\}}

\everymath=\expandafter{\the\everymath\displaystyle} %causes all math to be \displaystyle
\allowdisplaybreaks

\begin{document}

\title{Quantum supersymmetric pairs of basic types}

\author{Yaolong Shen}
 \author{Weiqiang Wang}
 \address{Department of Mathematics, University of Ottawa, Ottawa, ON, K1N 6N5, Canada}\email{ys8pfr@virginia.edu}
 \address{Department of Mathematics, University of Virginia,
Charlottesville, VA 22903, USA}\email{ww9c@virginia.edu}

\subjclass[2020]{Primary 17B37, 17B10}  
\keywords{Quantum supergroups, quantum symmetric pairs, Schur duality}

\begin{abstract}
    We formulate and classify super Satake diagrams under a mild assumption, building on arbitrary Dynkin diagrams for finite-dimensional basic Lie superalgebras. We develop a theory of quantum supersymmetric pairs associated to the super Satake diagrams. We establish the quantum Iwasawa decomposition and construct quasi $K$-matrix associated with the quantum supersymmetric pairs. We also formulate a Schur duality between an $\imath$quantum supergroup (which is a new $q$-deformation of an  ortho-symplectic Lie superalgebra) and the $q$-Brauer algebra.
\end{abstract}

\maketitle

\setcounter{tocdepth}{1}
\tableofcontents

\section{Introduction}
\renewcommand{\thetheorem}{\Alph{theorem}}
\setcounter{theorem}{0}

%\red{Comments}
%\begin{enumerate}
%    \item 
%\end{enumerate} 

\subsection{Background}
Drinfeld-Jimbo quantum groups are fundamental subjects from several viewpoints in mathematics (such as Lie theory and knot theory) or in physics (such as integrable systems, statistical mechanics, conformal field theory). Similarly, Lie superalgebras are not only important for supersymmetries in mathematical physics, but also have rich representation theories and deep connections to categorification.

Symmetric pairs $(\mathfrak g,\mathfrak g^\theta)$, where $\g$ is a complex simple Lie algebra and $\theta$ an involution on $\g$, are in bijection with real simple Lie algebras and can be classified in terms of the Satake diagrams; cf. \cite{Ara62}. By definition a Satake diagram $(I=\I_\circ \cup \I_\bu, \tau)$ consists of a bicolored Dynkin diagram $I=\I_\circ \cup \I_\bu$ with a diagram automomorphism $\tau$ of order $\le 2$.
The quantum symmetric pairs $(\U,\Ui)$, as developed by Letzter \cite{Let99,Let02}, provide a natural quantization of the symmetric pairs $(\mathfrak g,\mathfrak g^\theta)$. In this context, $\U$ stands for the quantum group associated with $\mathfrak g$ while $\Ui$ denotes a coideal subalgebra of $\U$, often referred to as an $\imath$quantum group. The constructions of quantum symmetric pairs has been extended to the Kac-Moody setting by Kolb \cite{Ko14}. In the last decade, many fundamental constructions for quantum groups, such as Jimbo-Schur duality, canonical basis, (quasi) $R$-matrix, and Yang-Baxter equation, have been generalized to the setting of $\imath$quantum groups; cf. \cite{BW18a, BW18b, BK19, WZ22}. 

Let $\g$ be a basic Lie superalgebra, i.e., a complex simple Lie superalgebra equipped with an even non-degenerate bilinear form (see \cite{K77, CW12}); we also regard $\gl(m|n)$ as a basic Lie superalgebra even though it is not simple. The fundamental systems $\I$ of the root system $\Phi$ associated with a general basic Lie superalgebra $\mathfrak g$ are not conjugated under the Weyl group actions due to the existence of isotropic odd roots. Consequently, the Dynkin diagrams associated with $\mathfrak g$ depend on the choice of Borel subalgebras $\mathfrak b^+$ and the associated positive roots $\Phi^+$, see Table~\ref{Table1}. In \cite{Ya94}, Yamane constructed  Serre presentations of quantum supergroups $\U=\U_q(\mathfrak g)$ associated with arbitrary Dynkin diagrams. A super feature is the existence of many different types of Serre relations involving 2, 3 or 4 Chevalley generators (cf. \cite{Ya94, CHW16}). The quantum supergroups of type $A$ have been studied in depth by Benkart-Kang-Kashiwara \cite{BKK00} and Mitsuhashi \cite{Mi06}. 

\subsection{Goal}
Let $\I$ be a Dynkin diagram of a basic Lie superalgebra $\g$. We formulate a notion of {\em super Satake diagrams} $(\I=\I_\circ \cup \I_\bu,\tau)$, which consist of bicolored Dynkin diagrams $\I=\I_\circ \cup \I_\bu$ and (possibly trivial) diagram involutions $\tau$, subject to the super admissible conditions. The super Satake diagrams are classified by classifying the even and odd rank 1 super Satake diagrams. 

We shall formulate quantum supersymmetric pairs $(\U,\Ui)$ associated to super Satake diagrams.
We develop fundamental properties of $(\U,\Ui)$ including a basis result for $\Ui$, quantum Iwasawa decomposition, and the quasi $K$-matrix (the $\imath$-analog of quasi $R$-matrix). We establish an $\imath$Schur duality between a distinguished $\imath$quantum supergroup (which quantizes the ortho-symplectic Lie superalgebra) and the $q$-Brauer algebra. 

Just as $\imath$quantum groups can be viewed as a generalization of Drinfeld-Jimbo quantum groups (see the survey \cite{Wa23}), we view $\imath$quantum supergroups $\Ui$ as a new family of quantum algebras which form a vast generalization of quantum supergroups. We note that a class of  quantum supersymmetric pairs associated to Satake diagrams with nontrivial diagram involutions has been constructed and studied by the first author \cite{Sh22}; the quasi-split (i.e., $\I_\bu=\emptyset$) case goes back to earlier work of Kolb-Yakimov \cite{KY20}. %These (when $\I_\bu\subset\Ieven$) are special cases of our quantum supersymmetric pairs

\subsection{Main results}

In the super setting, $\I =\Ieven \cup \Iodd$ is $\Z_2$-graded consisting of even and odd simple roots. In this paper from the outset we shall impose the condition that $\I_\bu$ is always even, i.e., $\I_\bu \subset \Ieven$, on super Satake diagrams $(\I=\I_\circ \cup\I_\bu,\tau)$. This ``evenness" assumption greatly simplifies the technical and notational complexities while it still allows most interesting families of quantum supersymmetric pairs to be constructed. In the Appendix \ref{sec:non-isotropic}, the conditions on super admissible pairs are relaxed to allow non-isotropic odd simple roots in $\I_\bu$. 

Toward a reasonable theory of quantum supersymmetric pairs, we first formulate super admissible conditions on $(\I=\I_\circ \cup\I_\bu,\tau)$ in Definition~\ref{def:superad} (as a super generalization of \cite{BBBR95, Ko14}); $(\I=\I_\circ \cup\I_\bu,\tau)$ with super admissible conditions are referred to interchangeably as super admissible pairs or super Satake diagrams. 
We provide a complete list of super Satake diagrams of real (odd) rank 1 in Table~\ref{Tableoddrank1}, which serve as building blocks for super Satake diagrams of higher rank. In contrast to the non-super setting (where the split rank $1$ $\imath$quantum group is a most fundamental building block), the split odd rank $1$ is not permissible in the super setting. We will comment on how the super admissible conditions arise. 

Associated to a super Satake diagram $(\I=\I_\circ \cup\I_\bu,\tau)$, following \cite{Ko14} we define an automorphism $\theta(I,\tau)$ of a basic Lie superalgebra $\g$ in \eqref{eq:theta}, leading to a subalgebra $\mathfrak t$ of $\g$ with a generating set \eqref{eq:t}. It follows by definition that $\mathfrak t +\mathfrak b^+ =\g$. As the Chevelley ``involution" (despite of the terminology) for a basic Lie superalgebra $\g$ is an automorphism of order $4$ in general, $\theta(I,\tau)$ can be order $4$ or order $2$ in our super setting; see Proposition~\ref{prop:thetaorder} for a criterion for $\theta(I,\tau)$ to be an order $2$ automorphism. In case when $\theta(I,\tau)$ is of order $4$, $\mathfrak t$ is not the fixed point subalgebra by $\theta(I,\tau)$ and (the even part of) $\mathfrak t$ may not be reductive; this is reminiscent of features of pseudo symmetric pairs developed by Regelskis-Vlaar \cite{RV20}.

The order 2 and 4 $\C$-linear automorphisms of basic Lie superalgebras $\g$ of type A--D have been classified in \cite[Tables~2, 5]{Ser83} (also cf. \cite{K77}). Not all such automorphisms will be suitable for our constructions of quantum supersymmetric pairs; see Remark~\ref{rem:2and4}. On the other hand, as $\g$ admits non-conjugate Dynkin diagrams (a super phenomenon), a fixed conjugacy class of order 2 or 4 automorphisms on $\g$ can arise from different super Satake diagrams; see Example~\ref{ex:superA}. In this case, the corresponding (classical) supersymmetric pairs can be isomorphic but the quantum supersymmetric pairs are not; see Remark~\ref{rem:yesno}.

Lusztig's braid group symmetries for quantum groups (cf. \cite[\S 37.1.3]{Lus93}) can be extended to algebra isomorphisms of the quantum supergroup $\U$ associated to $\g$, and these braid group operators $T_i$ (for $i\in \Ieven$) on $\U$ satisfy the braid group relations (cf. \cite[\S 6.3]{H10}). The braid group operators associated with $i\in \I_\bu$ and especially $T_{w_\bu}$ associated to the longest element $w_\bu$ in the Weyl group $W_{\I_\bu}$
are used in the definition of the $\imath$quantum supergroup $\Ui$. Following \cite{Ko14} we define an automorphism $\theta_q(I,\tau)$ \eqref{eq:thetaq} of $\U$, quantizing the aforementioned automorphism $\theta(I,\tau)$ of $\g$.

For a given super admissible pair, we define the {\em $\imath$quantum supergroup} $\Ui$ to be the $\C(q)$-subalgebra of $\U$ generated by  certain Cartan elements, $E_j,\ F_j$   $(j\in \I_\bu)$, and 
\begin{align}  \label{eq:parameter}
B_i:=F_i+\va_i T_{w_\bu}(E_i)K_i^{-1}\quad (i\in \I_\circ). 
\end{align} 
Conceptually, $B_i$ is defined in \eqref{eq:Bi} in terms of the automorphism $\theta_q(I,\tau)$ of $\U$, and for simplicity we have set $\kappa_i=0$ therein in the Introduction. Compare \cite{Let02, Ko14}. 

Let $\Uio$ denote the Cartan part of $\Ui$ and let $\U\cong \U^+\U^0\U^-$ denote the triangular decomposition of $\U$. The parameters $\{\va_i\}_{i\in \I_\circ} \in \C(q)^{\times, \I_\circ}$ in \eqref{eq:parameter} are required to satisfy the conditions \eqref{eq:kappai}--\eqref{eq:vai=vataui} throughout of the paper, which arise from ensuring that the identity $\Ui\cap \U^0=\Uio$ holds.

\begin{proposition} 
[Proposition~\ref{prop:coideal}]
    $\Ui$ is a right coideal subalgebra of $\U$. 
\end{proposition}
We shall refer to $(\U, \Ui)$ as a quantum supersymmetric pair. Compare \cite{Let99} and \cite[Proposition 5.2]{Ko14}. 

Roughly speaking, Theorem~\ref{th:B} below says that the algebra $\Ui$ is a flat deformation of $U(\mathfrak t)$, and Part~(2) is known as the quantum Iwasawa decomposition. Denote by $\{F_J \mid J \in \mathcal{J}\}$ a monomial basis for $\U^-$. We denote by $V_\bu^+$ (see \eqref{eq:VUb}) the subalgebra of $U^+$ generated by root vectors in $\U^+$ but not in $\U_\bu^+$, and denote by $\Uio_\tau$ the subalgebra generated by $\set{K_i^{\pm 1}\mid i\in \I_\tau}$ (see \eqref{eq:UtauUio}). It is convenient to set $B_i=F_i$, for $i\in \I_\bu$. We have the following super generalization of \cite{Let99} and \cite[Propositions 6.3, 6.13]{Ko14}. 

\begin{theorem}
[Theorem \ref{thm:basis}, Theorem \ref{thm:Iwa}]
 \label{th:B}
Suppose a super Satake diagram satisfies \eqref{exclude}.

 \begin{enumerate}
 \item The set $\{B_J \mid J \in \mathcal{J}\}$ forms a basis for the left $\U^{+}_{\bu}\Uio$-module $\Ui$.
     \item 
    The multiplication gives a vector space isomorphism 
    $V_\bu^+\otimes \Uio_\tau \otimes \Ui \cong \U$.
 \end{enumerate}
\end{theorem}
The proof of this theorem relies crucially on the projection technique which originate from \cite{Let99, Let02, Ko14}; see also \cite{KY20, Sh22}. The key step in proving it is the establishment of Proposition~\ref{prop:pi00p} for all possible Serre relations as detailed in Table~\ref{TableSerrePolyn}. This process necessitates a thorough case-by-case examination, which motivates Definition~\ref{def:superad} on super admissible conditions and is also the reason behind the mild assumption \eqref{exclude} in Theorem~\ref{th:B}. Proposition~\ref{prop:pi00p} also leads to monomial bases for $\Ui$; see Theorem~\ref{thm:basis}. Theorem~\ref{th:B} in the special case when the super Satake diagram admits a nontrivial diagram involution (generalizing the non-super type AIII) was established in \cite{Sh22} (which goes back to \cite{KY20} under the additional assumption that $\I_\bu=\emptyset$).

A first formulation of the quasi $K$-matrix in quantum symmetric pairs (as an $\imath$-analog of quasi $R$-matrix) was given in \cite{BW18a} via intertwining property involving the bar involution on $\Ui$ under suitable assumption on parameters; see  \cite{BK19} for a general proof. Following the formulation in \cite{WZ22} via a new intertwining property in terms of anti-involutions, we establish  the quasi $K$-matrices of quantum supersymmetric pairs which are valid for general parameters. (This was earlier achieved in a more complicated formulation by Appel--Vlaar \cite{AV22}.)

\begin{theorem}
[Theorem~\ref{thm:ibar}]
\label{th:C}
Retain the assumption \eqref{exclude}.
There exists a unique element $\up =\sum_{\mu \in X^+} \up_\mu$, where $\up_0=1$ and $\up_\mu \in \U^+_\mu$, such that  
$B_{\tau i} \up=\up \sigma \tau (B_{i})$,
for $i\in \I_\circ$ and $x \up =\up x$,  for $x \in \U^{\imath 0} \U_\bu.$
\end{theorem}
Besides adapting several main ingredients from \cite{BW18a, BK19, WZ22}, we formulate an essential super component for establishing Theorem~\ref{thm:ibar} as Theorem~\ref{thm:2b}. Similar to the quantum Iwasawa decomposition, the validation of Theorem~\ref{thm:2b} rely on various Serre relations, thus requiring an extensive case-by-case examination. For quantum supersymmetric pairs with nontrivial diagram involutions (generalizing the usual type AIII), quasi $K$-matrix in a characterization different from Theorem~\ref{th:C} was established in \cite{KY20, Sh22} (generalizing \cite{BW18a}).

The quasi $K$-matrix is a key ingredient for the construction of the universal K-matrix (cf. \cite{BW18a, BW18b, BK19}), which was used in constructing canonical bases and leads to solutions for reflection equations (see \cite{BK19, AV22}). It will be interesting to develop this further. There are many other constructions on quantum symmetric pairs (see the survey \cite{Wa23}) whose super analogue can be very interesting to develop. 

Finally, we illustrate by an example that $\imath$quantum supergroups are interesting new families of quantum algebras. More specifically, we consider the quantum supersymmetric pair associated with the following super Satake diagram: 
{\rm \[
\hspace{.75in}\xy
(-20,0)*{\fullmoon};(-10,0)*{\cdots}**\dir{-};(0,0)*{\fullmoon}**\dir{-};(10,0)*{\otimes}**\dir{-};(20,0)*{\newmoon}**\dir{-};(30,0)*{\fullmoon}**\dir{-};(40,0)*{\newmoon}**\dir{-};(50,0)*{\cdots}**\dir{-};(60,0)*{\fullmoon}**\dir{-};(70,0)*{\newmoon}**\dir{-};
(-20,-4)*{\scriptstyle 1};(0,-4)*{\scriptstyle m-1};(10,-4)*{\scriptstyle m};(20,-4)*{\scriptstyle m+1};(30,-4)*{\scriptstyle m+2};(40,-4)*{\scriptstyle m+3};(70,-4)*{\scriptstyle m+2n-1};
\endxy
\]}
where $\Iodd=\{m\}$,  $I_\bu=\{m+2a-1\mid 1\leq a\leq n\}$ and $I_\circ=I \backslash I_\bu$. It is worth noting that the $\imath$quantum supergroup $\Ui$ in this case is a new $q$-deformation of the enveloping superalgebra of $\mathfrak{osp}(m|2n)$, and we actually have various different $q$-deformations of the ortho-symplectic Lie superalgebras associated with different super type A Satake diagrams in Example~\ref{ex:superA}. Let $\V$ denote the natural representation of $\U$. 

\begin{theorem} [Theorem~\ref{thm:duality}]
\label{th:D}
   Let $m,n \ge 1$. The actions of the $q$-Brauer algebra $\B_d(q,q^{m-2n})$ and $\Ui$ on $\V^{\otimes d}$ commute with each other. Moreover, they satisfy a double centralizer property when $\B_d(q,q^{m-2n})$ is semisimple.
\end{theorem}

This version of $q$-Brauer algebra was introduced in \cite{We12}. An explicit criterion on parameters for the semisimplicity of the algebra $\B_d(q,q^{m-2n})$ is given by Rui-Si-Song in \cite{RSS24}. 
The $\imath$Schur duality between an $\imath$quantum supergroup and the $q$-Brauer algebra in Theorem~\ref{th:D} unifies the $\imath$Schur dualities of type AI and AII given in \cite{CS22}; an earlier variant of such $\imath$Schur duality of type AI was due to Molev \cite{Mol03} where a different (larger) $q$-Brauer algebra was used. 
%An $\imath$Schur duality involving $\imath$quantum supergroup of type AIII was constructed by the first author in \cite{Sh22}. 

%\subsection{Further directions}

\subsection{Organization}
The paper is organized as follows. In Section~\ref{sec:symmetric pair}, we review the root datum of basic Lie superalgebras and formulate super admissible pairs or super Satake diagrams. We construct an automorphism $\theta(\I,\tau)$ of order $2$ or $4$ associated with a super admissible pair. In Section~\ref{sec:quantum involution}, we review the definition of quantum supergroups $\U$ and construct a quantum analogue of the automorphism $\theta_q(\I,\tau)$ of $\U$. 

In Section~\ref{sec:QSP}, we define the quantum supersymmetric pairs $(\U,\Ui)$ associated with super admissible pairs and establish a quantum Iwasawa decomposition for $(\U,\Ui)$. In Section~ \ref{sec:braid}, we show that the braid group action $T_i$, for $i\in \I_\bu$, on $\U$ preserves $\Ui$. 
In Section~\ref{sec:quasiK}, we establish the quasi $K$-matrices of  quantum supersymmetric pairs. An $\imath$Schur duality between a distinguished class of $\imath$quantum groups and the $q$-Brauer algebra is developed in Section~\ref{sec:duality}. In Appendix~\ref{sec:non-isotropic}, we further extend our constructions to Satake diagrams which allow non-isotropic odd roots in $\I_\bu$.

\vspace{2mm}
\noindent {\em Notes added.} 
A preprint \cite{AMS24} on similar topics appeared in \href{https://arxiv.org/abs/2407.19477}{arXiv:2407.19477}, and this prompted us to give a final touch to our paper (whose main results were completed a few months ago). The $\imath$quantum supergroups were independently defined and shown {\em loc. cit.} to be coideal subalgebras of quantum supergroups (of type ABCD only), though the main Theorems~\ref{th:B}-\ref{th:C}-\ref{th:D} of our paper did not appear.

\section{Super Satake diagrams}
\label{sec:symmetric pair}
In this section, we review various properties of basic Lie superalgebras $\g$ and the associated Dynkin diagrams. We formulate the notion of super Satake diagrams (or super admissible pairs), whose super admissible conditions are motivated by considerations in later sections. Then we construct a corresponding automorphism of $\g$ of order 2 or 4, following \cite{Ko14}. 

\renewcommand{\thetheorem}{\thesection.\arabic{theorem}}
\setcounter{theorem}{0}

\subsection{Root Datum} \label{SS:Root Data}
Let $\g=\g_\zero \oplus \g_\one$ be a complex basic Lie superalgebra of type A-G with a Cartan subalgebra $\h$ and a root system $\Phi$ (cf. \cite{CW12}). The Dynkin diagrams of $\g$ are listed in Table~\ref{Table1}. Given such a Dynkin diagram $\I$, we denote by $\Pi=\{\alpha_i\mid i\in \I\}$ the set of simple roots and $\Phi^+ =\Phi^+_\zero \cup \Phi^+_\one$ the associated positive system; as usual $\zero$ and $\one$ here and below indicate the even and odd roots respectively. Decompose $\Pi=\Pi_\zero \sqcup \Pi_\one$ where $\Pi_s=\Pi\cap\Phi_s$ with $s \in \{\zero, \one\}$, and accordingly we have $\I=\Ieven\sqcup\Iodd$. 
For $\beta \in \Phi$, we write $p(\beta)=s$ if $\beta\in \Phi_s$ with $s\in \{\zero,\one\}$. We often write $p(i):=p(\alpha_i)$, for $i\in \I$. By linearity, we have a parity function $p(\cdot)$ on the root lattice $X =\Z\Pi$. The Lie superalgebra $\g$ is generated by Chevalley generators $\set{e_i,f_i\mid i\in\I}$ and $\h$, and we have a  triangular decomposition $\g=\n^+\oplus \h\oplus\n^-$, where $\n^{\pm} =\oplus_{\beta\in \Phi^{\pm}}\g_\beta$. Set $\mathfrak b^+ =\h \oplus \n^+.$ 

The basic Lie superalgebras are examples of symmetrizable contragredient Lie superalgebras associated to Cartan matrices \cite{K77}, which are endowed with a non-degenerate even supersymmetric bilinear form. Denote by $A=(a_{ij})_{i,j\in \I}$ the Cartan matrix for $\g$.
There exist non-zero integers $d_i$, for $i\in \I$, such that 
\begin{align}
    \label{di}
 d_ia_{ij}=d_ja_{ji}, \qquad \gcd(d_i | i\in \I)=1.
\end{align}
The set of simple coroots $\Pi^\vee=\{h_i\mid i\in \I\} \subset \h$ is given by
\[
    \alpha_j(h_i)=a_{ij},\quad \forall i,j\in \I.
\]
Let $Y=\Z \Pi^\vee$ denote the coroot lattice. We define the parity function $p$ on $Y$ by $p(h_i)=p(i)$, for all $i \in \I$, and extend by linearity. 

Define a symmetric bilinear form $(\cdot,\cdot):X\times X\longrightarrow \Z$ by letting
$$ (\af_i,\af_j) =d_i a_{ij},\;\;\;i,j\in \I.$$
A root $\alpha$ is called isotropic if $(\alpha, \alpha)=0$, and this can only happen when $\alpha$ is odd. We further decompose
$$\Phi_\one=\Phi_\iso\sqcup\Phi_\niso$$
where $\Phi_\iso$ (resp. $\Phi_\niso$) is the set of isotropic (resp. non-isotropic) odd roots. Decompose
$\Pi_\one=\Pi_\iso\sqcup\Pi_\niso$ (resp. $\I_\one=\I_{\iso}\sqcup\I_{\niso}$) accordingly.

\begin{lemma}
The following are equivalent, for $i\in \I$:
\[
(1) \; a_{ii}=0;\qquad (2) \; i\in \I_{\iso};\qquad (3) \; (\af_i,\af_i)=0.
\]
\end{lemma}

In Table~\ref{Table1} below, we present the general Dynkin diagrams for basic Lie superalgebras of type $A$-$G$. The nodes of the diagrams are labeled with the corresponding labels for the simple roots. The symbols $\fullmoon$, $\otimes$, and $\LEFTcircle$ are assigned to denote even, isotropic odd, and non-isotropic odd simple roots, respectively. Additionally, we utilize the notation $\odot$ to represent a simple root that can be either even or isotropic odd, and $\yy$ to represent a simple root that can be either even or non-isotropic odd.

\begin{center}
\begin{table}[ht]
 \caption{Dynkin diagrams for basic Lie superalgebras}{\label{Table1}}
\scalebox{0.8}{
\begin{tabular}{|c|c|}\hline
$A(m,n)$&
$$ \xy
(-30,0)*{\odot};(-20,0)*{\odot}**\dir{-};(-15,0)*{\cdots};(-10,0)*{\odot};(0,0)*{\odot}**\dir{-};
(0,0)*{\odot};(10,0)*{\odot}**\dir{-};
(15,0)*{\cdots};(20,0)*{\odot};(30,0)*{\odot}**\dir{-};
(-30,-4)*{\scriptstyle 1};(-20,-4)*{\scriptstyle 2};(-10,-4)*{\scriptstyle n};(0,-4)*{\scriptstyle n+1};(10,-4)*{\scriptstyle n+2};
(20,-4)*{\scriptstyle m+n};(30,-4)*{\scriptstyle m+n+1};
(0,-8)*{};(0,8)*{};
\endxy$$
\\\hline
$B(m,n+1)$&
$$ \xy
(-30,0)*{\odot};(-20,0)*{\odot}**\dir{-};(-15,0)*{\cdots};(-10,0)*{\odot};(0,0)*{\odot}**\dir{-};
(0,0)*{\odot};(10,0)*{\odot}**\dir{-};
(15,0)*{\cdots};(20,0)*{\odot};(30,0)*{\yy}**\dir{=};(25,0)*{>};
(-30,-4)*{\scriptstyle 1};(-20,-4)*{\scriptstyle 2};(-10,-4)*{\scriptstyle n};(0,-4)*{\scriptstyle n+1};(10,-4)*{\scriptstyle n+2};
(20,-4)*{\scriptstyle m+n};(31,-4)*{\scriptstyle m+n+1};
(0,-8)*{};(0,8)*{};
\endxy$$
\\\hline
$C(n+1)$&
$$ \xy
(-20,0)*{\odot};(-10,0)*{\odot}**\dir{-};(-5,0)*{\cdots};(0,0)*{\odot};(10,0)*{\odot}**\dir{-};
(10,0)*{\odot};(20,0)*{\fullmoon}**\dir{=};(15,0)*{<};
(-20,-4)*{\scriptstyle 1};(-10,-4)*{\scriptstyle 2};
(10,-4)*{\scriptstyle n};(20,-4)*{\scriptstyle n+1};
(0,-8)*{};(0,8)*{};
\endxy$$
\\
$D(m,n+1)$&
$$\xy
(-30,0)*{\odot};(-20,0)*{\odot}**\dir{-};(-15,0)*{\cdots};
(-10,0)*{\odot};(0,0)*{\odot}**\dir{-};(0,0)*{\odot};(10,0)*{\odot}**\dir{-};
(15,0)*{\cdots};(20,0)*{\odot};(27,5)*{\fullmoon}**\dir{-};(20,0)*{\odot};(27,-5)*{\fullmoon}**\dir{-};
(-30,-4)*{\scriptstyle 1};(-20,-4)*{\scriptstyle 2};(-10,-4)*{\scriptstyle n};(0,-4)*{\scriptstyle n+1};(10,-4)*{\scriptstyle n+2};
(33,5)*{\scriptstyle m+n};(34,-5)*{\scriptstyle m+n+1};
(0,-8)*{};(0,8)*{};
\endxy$$
\\
&
$$\xy
(-30,0)*{\odot};(-20,0)*{\odot}**\dir{-};(-15,0)*{\cdots};
(-10,0)*{\odot};(0,0)*{\odot}**\dir{-};(0,0)*{\odot};(10,0)*{\odot}**\dir{-};
(15,0)*{\cdots};(20,0)*{\odot};(27,5)*{\otimes}**\dir{-};(20,0)*{\odot};(27,-5)*{\otimes}**\dir{-};
(27,5)*{\otimes};(27,-5)*{\otimes}**\dir{=};
(-30,-4)*{\scriptstyle 1};(-20,-4)*{\scriptstyle 2};(-10,-4)*{\scriptstyle n};(0,-4)*{\scriptstyle n+1};(10,-4)*{\scriptstyle n+2};
(33,5)*{\scriptstyle m+n};(34,-5)*{\scriptstyle m+n+1};
(0,-8)*{};(0,8)*{};
\endxy$$
\\\hline
$F(3|1)$&
$$\xy (0,8)*{};%(-20,0)*{(\star)};
(-15,0)*{\fullmoon};(-5,0)*{\fullmoon}**\dir{-};
(-5,0)*{\fullmoon};(5,0)*{\fullmoon}**\dir{=};(0,0)*{>};(5,0)*{\fullmoon};(15,0)*{\otimes}**\dir{-};
(-15,-4)*{\scriptstyle 1};(-5,-4)*{\scriptstyle 2};(5,-4)*{\scriptstyle 3};(15,-4)*{\scriptstyle 4};
(0,-8)*{};
\endxy$$
\\&
$$\xy (0,8)*{};
{\ar@3{-}(-15,0)*{\fullmoon};(-5,0)*{\otimes}};(-10,0)*{>};
(-5,0)*{\otimes};(5,0)*{\fullmoon}**\dir{=};(0,0)*{<};(5,0)*{\fullmoon};(15,0)*{\fullmoon}**\dir{-};
(-15,-4)*{\scriptstyle 1};(-5,-4)*{\scriptstyle 2};(5,-4)*{\scriptstyle 3};(15,-4)*{\scriptstyle 4};
(0,-8)*{};
\endxy$$ \hspace{.25in}
$$\xy (0,8)*{};
{\ar@3{-}(-15,0)*{\fullmoon};(-5,0)*{\otimes}};(-10,0)*{>};
(-5,0)*{\otimes};(5,0)*{\fullmoon}**\dir{-};(5,0)*{\fullmoon};(15,0)*{\fullmoon}**\dir{=};(10,0)*{<};
(-15,-4)*{\scriptstyle 1};(-5,-4)*{\scriptstyle 2};(5,-4)*{\scriptstyle 3};(15,-4)*{\scriptstyle 4};
(0,-8)*{};
\endxy$$
\\&
$$\xy (0,8)*{};
{\ar@3{-}(-10,0)*{\otimes};(-5,7)*{\fullmoon}};(-5,7)*{\fullmoon};(0,0)*{\otimes}**\dir{-};
(-10,0)*{\otimes};(0,0)*{\otimes}**\dir{=};(0,0)*{\otimes};(10,0)*{\fullmoon}**\dir{=};(5,0)*{<};
(-10,-4)*{\scriptstyle 1};(-5,9)*{\scriptstyle 2};(0,-4)*{\scriptstyle 3};(10,-4)*{\scriptstyle 4};
(0,-8)*{};
\endxy$$ \hspace{.25in}
$$\xy (0,8)*{};
(-10,0)*{\otimes};(-5,7)*{\fullmoon}**\dir{-};(-5,7)*{\fullmoon};(0,0)*{\otimes}**\dir{-};
(-10,0)*{\otimes};(0,0)*{\otimes}**\dir{=};(0,0)*{\otimes};(10,0)*{\fullmoon}**\dir{=};(5,0)*{<};
(-10,-4)*{\scriptstyle 1};(-5,9)*{\scriptstyle 2};(0,-4)*{\scriptstyle 3};(10,-4)*{\scriptstyle 4};
(0,-8)*{};
\endxy$$
\\\hline
$G(3)$&
$$ \xy(0,5)*{};%(-15,0)*{(\star)};
{\ar@3{-}(0,0)*{\fullmoon};
(10,0)*{\fullmoon}};(5,0)*{<};(-10,0)*{\otimes};(0,0)*{\fullmoon}**\dir{-};
(-10,-4)*{\scriptstyle 1};(0,-4)*{\scriptstyle 2};(10,-4)*{\scriptstyle 3};
(0,-8)*{};(0,8)*{};
\endxy$$ \hspace{.25in}
$$ \xy(0,5)*{};
{\ar@3{-}(0,0)*{\otimes};(10,0)*{\fullmoon}};(5,0)*{<};(-10,0)*{\otimes};(0,0)*{\otimes}**\dir{-};
(-10,-4)*{\scriptstyle 1};(0,-4)*{\scriptstyle 2};(10,-4)*{\scriptstyle 3};
(0,-8)*{};(0,8)*{};
\endxy$$
\\&
$$ \xy(0,5)*{};
{\ar@3{-}(0,0)*{\otimes};(10,0)*{\fullmoon}};(5,0)*{<};{\ar@2{-}(-10,0)*{\LEFTcircle};(0,0)*{\otimes}};(0,0)*{\otimes}**\dir{-};
(-10,-4)*{\scriptstyle 1};(0,-4)*{\scriptstyle 2};(10,-4)*{\scriptstyle 3};
(0,-8)*{};(0,8)*{};
\endxy$$ \hspace{.25in}
$$\xy (0,5)*{};
{\ar@2{-}(-5,0)*{\otimes};(0,7)*{\fullmoon}};(0,7)*{\fullmoon};(5,0)*{\otimes}**\dir{-};
{\ar@3{-}(-5,0)*{\otimes};(5,0)*{\otimes}};
(-5,-4)*{\scriptstyle 1};(0,9)*{\scriptstyle 2};(5,-4)*{\scriptstyle 3};
(0,-8)*{};(0,8)*{};
\endxy$$
\\\hline
$D(2|1;\af)$&
$$\xy
(0,0)*{\otimes};(7,5)*{\fullmoon}**\dir{-};(0,0)*{\otimes};(7,-5)*{\fullmoon}**\dir{-};(2,4)*{\scriptstyle -1};(2,-4)*{\scriptstyle 1+\af};
(-3,0)*{\scriptstyle 1};(10,5)*{\scriptstyle 2};(10,-5)*{\scriptstyle 3};
(0,8)*{};(0,-8)*{};
\endxy$$ \hspace{.25in}
$$\xy
(0,0)*{\otimes};(7,5)*{\fullmoon}**\dir{-};(0,0)*{\otimes};(7,-5)*{\fullmoon}**\dir{-};(2,4)*{\scriptstyle -\af};(2,-4)*{\scriptstyle 1+\af};
(-3,0)*{\scriptstyle 1};(10,5)*{\scriptstyle 2};(10,-5)*{\scriptstyle 3};
(0,8)*{};(0,-8)*{};
\endxy$$
\\$(\af\in\Z_{>0})$&
$$\xy
(-5,0)*{\otimes};(0,7)*{\otimes}**\dir{-};
(-5,0)*{\otimes};(5,0)*{\otimes}**\dir{-};
(0,7)*{\otimes};(5,0)*{\otimes}**\dir{-};
(-6,-3)*{\scriptstyle 1};(0,10)*{\scriptstyle 2};(6,-3)*{\scriptstyle 3};
(4,4)*{\scriptstyle \af};(0,-2)*{\scriptstyle -1-\af};
(0,8)*{};(0,-8)*{};
\endxy$$
\\\hline
\end{tabular}}
\end{table}
\end{center}
The Chevalley involution $\om$ on $\g$ (of order 4 in the super setting) is defined by 
\begin{equation}
\label{eq:Chevalley}
    \om(e_i)=-f_i,\ \om(f_i)=-(-1)^{p(i)}e_i,\ \om(h_i)=-h_i,\ \forall i\in\I.
\end{equation}

For each $i\in \Ieven$ the even reflection $r_i$ is defined by
\[
r_i(h)=h-\alpha_i(h)h_i,\ \text{ for all }h\in \h.
\] 
%The Weyl group $W$ is the Coxeter group generated all the even reflections. Moreover, the symmetric bilinear form $(\cdot,\cdot)$ is $W$-invariant.
Denote by $G_\zero$ the adjoint group of $\g_\zero$. 
Let $\text{Aut}(\g)$ denote the group of Lie superalgebra automorphisms of $\g$. For any $i\in \Ieven$, define an element $m_i \in G_\zero$ by 
\[
m_i=\exp(e_i) \exp(-f_i) \exp(e_i).
\]
The adjoint action $\g_\zero$ on $\g$ induces the adjoint action $\Ad\colon G_\zero \rightarrow \text{Aut} (\g)$. One has $\Ad(m_i)(\g_\alpha)=\g_{r_i(\alpha)}$, for all $\alpha\in \h^*$, and $\Ad(m_i)|_\h=r_i$, for all $i\in \Ieven$. Moreover, the elements $m_i$ satisfy the relations
\begin{equation*}\underbrace{m_im_jm_i\cdots}_{d_{ij}}=\underbrace{m_jm_im_j\cdots}_{d_{ij}}
\end{equation*}
where $d_{ij}=2,3,4,6$ if $|a_{ij}a_{ji}|=0,1,2,3$, respectively.

\subsection{Super admissible pairs}

For a subset $\I_\bu\subset \Ieven$, denote by $\g_\bu$ the semisimple Lie subalgebra of $\g$ associated with $\I_\bu$ and by $\omega_\bu$ the Chevalley involution of $\g_\bu$ which sends $e_i \rightarrow -f_i, h_i \rightarrow -h_i$, for $i\in \I_\bu$. Denote by $W_\bu$ the Weyl group for $\g_\bu$ with the longest element $w_\bu$. Let $\tau_\bu$ denote the diagram involution of $\g_\bu$ corresponding to the element $w_\bu$. Let $\Phi_\bu \subset \Phi$ be the corresponding root system and denote $\Phi_\bu^+ =\Phi_\bu \cap \Phi^+$. Fix any reduced expression $w_\bu=r_{i_1}\cdots r_{i_k}$, we can define $m_\bu=m_{i_1}\cdots m_{i_k}$ and the element $m_\bu$ does not depend on the choice of the reduced expression.

We view $\g_\bu$ as a Levi subalgebra of $\g$, and therefore we have the positive coroots for $\g_\bu$ as part of coroots for $\g$. Let $\rho_\bu^\vee$ denote half the sum of the positive coroots in $\Phi_\bu$. Let $\text{Aut} (I)$ denote the group of all parity-preserving permutations $\tau$ of $\I$ such that $d_ia_{ij}=d_{\tau i}a_{\tau i,\tau j}$ and 
\[
\text{Aut}(I,\I_\bu)=\set{\tau\in \text{Aut}(I)\mid \tau(\I_\bu)=\I_\bu}.
\]
This induces an automorphism $\tau$ of Lie superalgebra $\g$ such that
\begin{align}
\label{eq:tau}
\begin{split}
\tau \colon\g & \longrightarrow\g, 
\\
e_i &\mapsto \begin{cases}
    e_{i}, & \tau i=i \\
    (-1)^{p(i)}e_{\tau i},& \tau i\neq i,
\end{cases}
%e_i\mapsto e_{\tau i},
\quad 
f_i\mapsto f_{\tau i},\quad 
%h_i\mapsto h_{\tau i} %
h_i\mapsto \begin{cases}
    h_{i}, & \tau i=i \\
    (-1)^{p(i)}h_{\tau i},& \tau i\neq i.
\end{cases}
\end{split}
\end{align}
(In case $\tau i \neq i \in \Iodd$, the sign change between $d_i$ and $d_{\tau i}$ is consistent with the sign on the formula above.) 

The following is well known.

\begin{proposition} {\rm (cf. \cite[Proposition 2.2]{Ko14})}
\label{prop:Adm}
\begin{enumerate}
    \item 
The automorphism $\Ad(m_\bu)$ of $\g$ leaves $\g_\bu$ invariant and satisfies the relation
\[\Ad(m_\bu)|_{\g_\bu}=\tau_\bu \omega_\bu.\]

\item 
In $\text{Aut}(\g)$ we have 
\[\Ad(m_\bu^2)=\Ad\big( \exp(\sqrt{-1}\pi 2\rho_\bu^\vee)\big).
\]
\item 
The automorphism $\Ad(m_\bu)$ of $\g$ commutes with all elements in $\text{Aut}(I,\I_\bu)$ and with the Chevalley involution $\om$.
\end{enumerate}
\end{proposition}

We now introduce the notion of super admissible pairs which generalizes admissible pairs as given in \cite{BBBR95, Ko14} (which correspond to Satake diagrams in \cite{Ara62}).

\begin{definition}
\label{def:superad}
A pair $(\I=\I_\circ \cup\I_\bu,\tau)$ consisting of subsets $\I_\bu\subset \Ieven$, $I_\circ=\I\backslash \I_\bu$, and an element $\tau\in \text{Aut}(I,\I_\bu)$ is called {\em super admissible} if the following conditions are satisfied:
\begin{enumerate}
    \item $\tau^2=id$,
    \item 
    The action of $\tau$ on $\I_\bu$ coincides with the action of $-w_\bu$,
    \item 
    If $j\in (\Ieven\cup\I_{\niso})\cap \I_\circ$ and $\tau(j)=j$, then $\alpha_j(\rho_\bu^\vee)\in \Z$, 
    \item
    If $j\in \I_{\iso}\cap \I_\circ$ and $\tau(j)=j$, then
$\alpha_j(\rho_\bu^\vee)\neq 0$.
\end{enumerate}
\end{definition}

\begin{remark}
Conditions (1), (2) and (3) (where $(\Ieven\cup\I_{\niso})\cap \I_\circ$ is replaced by $\I_\circ$) are exactly the conditions for admissible pairs of (non-super) Kac-Moody type  \cite[\S 2.4]{Ko14}. Condition~ (4) can be rephrased as that any isotropic simple root fixed by $\tau$ must be connected to $\I_\bu$.
\end{remark}

We shall convert the super admissible pairs into diagrams called super Satake diagrams, and we will use the terms ``super admissible pairs" and ``super Satake diagrams" interchangeably. 
To that end, we will use the node $\newmoon$ in a diagram to represent simple even roots in $\I_\bu$. In addition, we will use the nodes $\fullmoon$, $\LEFTcircle$ and $\otimes$ to denote even, non-isotropic odd and isotropic odd simple roots in $\I_\circ$, respectively.

\begin{remark}  \label{rem:subdiagram}
A subdiagram $\Gamma$ of a super Satake diagram $\Gamma'$ is understood in a strict sense as follows: if $i\in \I_\circ$ is a white node of $\Gamma$, then $\Gamma$ contains each connected component of black nodes of $\Gamma'$ which is connected to $i$. 
\end{remark}

\begin{example}
\label{example:admissible}
The following diagrams are super admissible: 
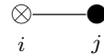
\begin{enumerate}
\item
{\rm 
$$\hspace{.75in}\xy
(-5,0)*{\otimes};(5,0)*{\newmoon}**\dir{-};
(-5,-4)*{\scriptstyle i};(5,-4)*{\scriptstyle j};
\endxy$$}
where $\I_\bu=\set{j}$ and $\tau=id$. In this case we see that $\rho_\bu^\vee=\frac{1}{2}h_j$, thus $\alpha_i(\rho_\bu^\vee)=\pm \frac{1}{2}\neq 0$. 

\item
{\rm
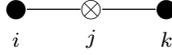
$$\hspace{.75in}\xy
(-10,0)*{\newmoon};(0,0)*{\otimes}**\dir{-};(10,0)*{\newmoon}**\dir{-};(-10,-4)*{\scriptstyle i};
(0,-4)*{\scriptstyle j};(10,-4)*{\scriptstyle k};
\endxy$$}
where $\I_\bu=\set{i,k}$ and $\tau=id$. In this case we see that $\rho_\bu^\vee=\frac{h_i+h_k}{2}$, thus $\alpha_j(\rho_\bu^\vee)=\pm 1\neq 0$. 

\item
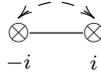
$$\xy
(-10,0)*{\otimes};
(0,0)*{\otimes}**\dir{-};(-10,-4)*{\scriptstyle -i};(0,-4)*{\scriptstyle i};
(-10,2)\ar@/_/@{<-->} (0,2);
\endxy$$
where  $I_\bu=\varnothing$, $\tau (-i)=i$ and $\tau i=-i$. In this case the Cartan matrix is $\begin{pmatrix}
    0 & -1\\ 1 &0
\end{pmatrix}$
while $d_{-i}=1=-d_i$.

\item 
{\rm
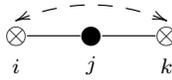
$$\hspace{.75in}\xy
(-10,0)*{\otimes};(0,0)*{\newmoon}**\dir{-};(10,0)*{\otimes}**\dir{-};(-10,-4)*{\scriptstyle i};
(0,-4)*{\scriptstyle j};(10,-4)*{\scriptstyle k};(-10,2)\ar@/_/@{<-->} (10,2);
\endxy$$}
where $\I_\bu=\set{j}$ and $\tau i=k$. In this case the Cartan matrix is 
\[\begin{pmatrix}
    0 & 1 & 0\\
    -1 & 2 &-1\\
    0 & -1 & 0
\end{pmatrix}\]
while $d_i=1,\ d_j=d_k=-1$. 

\item 
{\rm
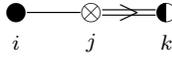
$$\hspace{.75in}\xy
(10,0)*{\newmoon};(20,0)*{\otimes}**\dir{-};(30,0)*{\LEFTcircle}**\dir{=};(25,0)*{>};(10,-4)*{\scriptstyle i};(20,-4)*{\scriptstyle j};(30,-4)*{\scriptstyle k};
\endxy$$}
where $I_\bu=\{i\}$ and $\tau=id$.
\end{enumerate}
\end{example}

\begin{example}
\label{example:nonadmissible}
The following diagrams are not super admissible:
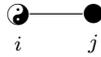
\begin{enumerate}
    \item 
     {\rm $$\xy
(-5,0)*{\fullmoon};(5,0)*{\newmoon}**\dir{-};
(-5,-4)*{\scriptstyle i};(5,-4)*{\scriptstyle j};
\endxy$$ 
where $\I_\bu=\set{j},\ \tau=id$ but $\alpha_i(\rho_\bu^\vee)=-1/2\notin \Z$. This is a purely even diagram.}

\item 
 {\rm $$\xy
(0,0)*{\otimes};(0,-4)*{\scriptstyle i}
\endxy$$
where $\tau=id$ but $\alpha_i(\rho_\bu^\vee)=0$ as $\I_\bu=\varnothing$.}

\item 
 {\rm 
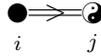
$$\xy
(-5,0)*{\newmoon};(5,0)*{\yy}**\dir{=};
(-5,-4)*{\scriptstyle i};(5,-4)*{\scriptstyle j};(0,0)*{>};
\endxy$$ 
where $\I_\bu=\set{i},\ \tau=id$ but $\alpha_j(\rho_\bu^\vee)=\pm 1/2\notin \Z$.}

\item 
{\rm
$$ 
\xy
(10,0)*{\fullmoon};(20,0)*{\otimes}**\dir{-};(30,0)*{\yy}**\dir{=};(25,0)*{>};(10,-4)*{\scriptstyle i};(20,-4)*{\scriptstyle j};(30,-4)*{\scriptstyle k};
\endxy$$}
where $I_\bu=\varnothing$ and $\tau=id$ but $\alpha_j(\rho_\bu^\vee)=0$. 
\end{enumerate}
\end{example}

% {\rm $$\xy
%    (-10,0)*{\otimes};(0,0)*{\fullmoon}**\dir{-};   (10,0)*{\fullmoon}**\dir{-};(20,0)*{\otimes}**\dir{-};   (-10,2)\ar@/_/@{<-->}(20,2);
%\endxy$$}
%and
%$$   {\rm  \xy
%    (-10,0)*{\otimes};(0,0)*{\otimes}**\dir{-};    (10,0)*{\otimes}**\dir{-};(20,0)*{\otimes}**\dir{-};    (-10,2)\ar@/_/@{<-->}(20,2);
%\endxy}
%$$

\subsection{Super Satake diagrams of real odd rank one}

We recall in Table \ref{table:evenrank1} below the well-known classification of (non-super) Satake diagrams of real (even) rank one \cite{Ara62} (cf. \cite[Table 1]{BW18b}).
%%%%%%%%%%%%%%% 
\begin{table}[h]
\caption{Satake diagrams of real even rank one}
\label{table:evenrank1}
\begin{tabular}{| c | c || c | c |}
\hline
\begin{tikzpicture}[baseline=0]
\node at (0, -0.15) {AI$_1$};
\end{tikzpicture} 
& 
$	\begin{tikzpicture}[baseline=0]
		\node  at (0,0) {$\circ$};
		\node  at (0,-.3) {$\scriptstyle 1$};
	\end{tikzpicture}
$ 
&
\begin{tikzpicture}[baseline=0]
\node at (0, -0.15) {AII$_3$};
\end{tikzpicture}
&
$	\begin{tikzpicture}[baseline=0]
		\node at (0,0) {$\bullet$};
		\draw (0.1, 0) to (0.4,0);
		\node  at (0.5,0) {$\circ$};
		\draw (0.6, 0) to (0.9,0);
		\node at (1,0) {$\bullet$};
		\node at (0,-.3) {$\scriptstyle 1$};
		\node  at (0.5,-.3) {$\scriptstyle 2$};
		\node at (1,-.3) {$\scriptstyle 3$};
			\end{tikzpicture}	
$

%%%%
\\
\hline
\begin{tikzpicture}[baseline=0]
\node at (0, -0.2) {AIII$_{11}$};
\end{tikzpicture}
 &
\begin{tikzpicture}[baseline = 6] %, scale =1.3]
		\node at (-0.5,0) {$\circ$};
%		\draw[dashed] (-0.4,0) to (0.4,0);
		\node at (0.5,0) {$\circ$};
%involutions
		\draw[<->] (-0.5, 0.2) to[out=45, in=180] (-0.15, 0.35) to (0.15, 0.35) to[out=0, in=135] (0.5, 0.2);
		\node at (-0.5,-0.3) {$\scriptstyle 1$};
%		\node at (0,-0.2) {\small 2};
		\node at (0.5,-0.3) {$\scriptstyle 2$};
	\end{tikzpicture}
	&
\begin{tikzpicture}[baseline=0]
\node at (0, -0.2) {AIV, $n\ge2$};
\end{tikzpicture} &
\begin{tikzpicture}	[baseline=6]
		\node at (-0.5,0) {$\circ$};
		\draw[-] (-0.4,0) to (-0.1, 0);
		\node  at (0,0) {$\bullet$};
		\node at (2,0) {$\bullet$};
		\node at (2.5,0) {$\circ$};
		\draw[-] (0.1, 0) to (0.5,0);
		\draw[dashed] (0.5,0) to (1.4,0);
		\draw[-] (1.6,0)  to (1.9,0);
		\draw[-] (2.1,0) to (2.4,0);
%involution
		\draw[<->] (-0.5, 0.2) to[out=45, in=180] (0, 0.35) to (2, 0.35) to[out=0, in=135] (2.5, 0.2);
		\node at (-0.5,-.3) {$\scriptstyle 1$};
		\node  at (0,-.3) {$\scriptstyle 2$};
		\node at (2.5,-.3) {$\scriptstyle n$};
	\end{tikzpicture}
%%%%
\\
\hline
\begin{tikzpicture}[baseline=0]
\node at (0, -0.2) {BII, $n\ge2$};
\end{tikzpicture} & 
    	\begin{tikzpicture}[baseline=0, scale=1.5]
		\node at (1.05,0) {$\circ$};
		\node at (1.5,0) {$\bullet$};
		\draw[-] (1.1,0)  to (1.4,0);
		\draw[-] (1.4,0) to (1.9, 0);
		\draw[dashed] (1.9,0) to (2.7,0);
		\draw[-] (2.7,0) to (2.9, 0);
		\node at (3,0) {$\bullet$};
		\draw[-implies, double equal sign distance]  (3.1, 0) to (3.7, 0);
		\node at (3.8,0) {$\bullet$};
		\node at (1,-.2) {$\scriptstyle 1$};
		\node at (1.5,-.2) {$\scriptstyle 2$};
%		\node at (3,0) {$\bullet$};
		\node at (3.8,-.2) {$\scriptstyle n$};
	\end{tikzpicture}	
&
\begin{tikzpicture}[baseline=0]
\node at (0, -0.15) {CII, $n\ge3$};
\end{tikzpicture}  
& 
		\begin{tikzpicture}[baseline=6]
		\draw (0.6, 0.15) to (0.9, 0.15);
		\node  at (0.5,0.15) {$\bullet$};
		\node at (1,0.15) {$\circ$};
		\node at (1.5,0.15) {$\bullet$};
		\draw[-] (1.1,0.15)  to (1.4,0.15);
		\draw[-] (1.4,0.15) to (1.9, 0.15);
%		\node at (2,0.15) {$\bullet$};
		\draw (1.9, 0.15) to (2.1, 0.15);
		\draw[dashed] (2.1,0.15) to (2.7,0.15);
		\draw[-] (2.7,0.15) to (2.9, 0.15);
		\node at (3,0.15) {$\bullet$};
		\draw[implies-, double equal sign distance]  (3.1, 0.15) to (3.7, 0.15);
		\node at (3.8,0.15) {$\bullet$};
		\node  at (0.5,-0.15) {$\scriptstyle 1$};
		\node at (1,-0.15) {$\scriptstyle 2$};
%		\node at (1.5,0.15) {$\bullet$};
%		\node at (2,0.15) {$\bullet$};
%		\node at (3,0.15) {$\bullet$};
		\node at (3.8,-0.15) {$\scriptstyle n$};
	\end{tikzpicture}		
\\
\hline
 %%%
\begin{tikzpicture}[baseline=0]
\node at (0, -0.05) {DII, $n\ge4$};
\end{tikzpicture}&
	\begin{tikzpicture}[baseline=0]
		\node at (1,0) {$\circ$};
		\node at (1.5,0) {$\bullet$};
		\draw[-] (1.1,0)  to (1.4,0);
		\draw[-] (1.4,0) to (1.9, 0);
		\draw[dashed] (1.9,0) to (2.7,0);
		\draw[-] (2.7,0) to (2.9, 0);
		\node at (3,0) {$\bullet$};
		\node at (3.5, 0.4) {$\bullet$};
		\node at (3.5, -0.4) {$\bullet$};
		\draw (3.1, 0.1) to (3.4, 0.35);
		\draw (3.1, -0.1) to (3.4, -0.35);
		\node at (1,-.3) {$\scriptstyle 1$};
		\node at (1.5,-.3) {$\scriptstyle 2$};
%		\node at (3,0) {$\bullet$};
		\node at (3.5, 0.7) {$\scriptstyle n-1$};
		\node at (3.5, -0.6) {$\scriptstyle n$};
	\end{tikzpicture}		
&
\begin{tikzpicture}[baseline=0]
\node at (0, -0.2) {FII};
\end{tikzpicture}&
\begin{tikzpicture}[baseline=0][scale=1.5]
	\node at (0,0) {$\bullet$};
	\draw (0.1, 0) to (0.4,0);
	\node at (0.5,0) {$\bullet$};
	\draw[-implies, double equal sign distance]  (0.6, 0) to (1.2,0);
	\node at (1.3,0) {$\bullet$};
	\draw (1.4, 0) to (1.7,0);
	\node at (1.8,0) {$\circ$};
	\node at (0,-.3) {$\scriptstyle 1$};
	\node at (0.5,-.3) {$\scriptstyle 2$};
	\node at (1.3,-.3) {$\scriptstyle 3$};
	\node at (1.8,-.3) {$\scriptstyle 4$};
\end{tikzpicture}
%%%%
\\
\hline
\end{tabular}
\newline
\smallskip
\end{table}

\begin{definition}
We define the {\em real odd (resp. even)  rank} of $(\I=\I_\circ \cup \I_\bu,\tau)$ by the number of $\tau$-orbits in $\I_\circ\cap \I_\iso$ (resp. $\I_\circ \cap(\Ieven\cup \I_\niso)$). Moreover, the real rank of $(\I=\I_\circ \cup \I_\bu,\tau)$ is defined to be the sum of its real odd and real even rank.
\end{definition}

Through a case-by-case checking with the help of Table \ref{Table1}, we can prove the following. 

\begin{proposition}
  \label{prop:rank1}
A complete list of real odd rank one super Satake diagrams is given in Table~\ref{Tableoddrank1}.
\end{proposition}
    
\begin{table}[ht]
\renewcommand\arraystretch{2}
     \caption{Super Satake diagrams of real odd rank 1}{\label{Tableoddrank1}}
\resizebox{\linewidth}{!}{
\begin{tabular}{|c|c||c|c|}
\hline
    sAI  &$$\xy
(-30,0)*{\otimes};(-20,0)*{\newmoon}**\dir{-};(-30,-4)*{\scriptstyle 1};(-20,-4)*{\scriptstyle 2};%(0,-4)*{\scriptstyle n}
\endxy$$ & sFI &
$$\xy (0,8)*{};%(-20,0)*{(\star)};
(-15,0)*{\newmoon};(-5,0)*{\newmoon}**\dir{-};
(-5,0)*{\newmoon};(5,0)*{\newmoon}**\dir{=};(0,0)*{>};(5,0)*{\newmoon};(15,0)*{\otimes}**\dir{-};(-15,-4)*{\scriptstyle 1};(-5,-4)*{\scriptstyle 2};(5,-4)*{\scriptstyle 3};(15,-4)*{\scriptstyle 4};
\endxy$$
\\ 
 sAII & $$ \xy 
(-40,0)*{\newmoon};(-30,0)*{\otimes}**\dir{-};(-20,0)*{\newmoon}**\dir{-};(-40,-4)*{\scriptstyle -1};(-30,-4)*{\scriptstyle 0};(-20,-4)*{\scriptstyle 1};
\endxy$$ & &%sFII & $$\xy (0,8)*{};
%{\ar@3{-}(-15,0)*{\newmoon};(-5,0)*{\otimes}};(-10,0)*{>};
%(-5,0)*{\otimes};(5,0)*{\newmoon}**\dir{=};(0,0)*{<};(5,0)*{\newmoon};(15,0)*{\newmoon}**\dir{-};(-15,-4)*{\scriptstyle 1};(-5,-4)*{\scriptstyle 2};(5,-4)*{\scriptstyle 3};(15,-4)*{\scriptstyle 4};
%\endxy$$ \\
\\
sAIII${}_{11}$ & $$\xy
(-10,0)*{\otimes};
(0,0)*{\otimes};(-10,-4)*{\scriptstyle 1};(0,-4)*{\scriptstyle 2};
(-10,2)\ar@/_/@{<-->} (0,2);
\endxy$$ & sFII & $$\xy (0,8)*{};
{\ar@3{-}(-15,0)*{\newmoon};(-5,0)*{\otimes}};(-10,0)*{>};
(-5,0)*{\otimes};(5,0)*{\newmoon}**\dir{-};(5,0)*{\newmoon};(15,0)*{\newmoon}**\dir{=};(10,0)*{<};(-15,-4)*{\scriptstyle 1};(-5,-4)*{\scriptstyle 2};(5,-4)*{\scriptstyle 3};(15,-4)*{\scriptstyle 4};
\endxy$$\\
sAIV, $n\geq 2$ &
$$\xy
(-30,-6)*{\ };
(-30,0)*{\otimes};(-20,0)*{\newmoon}**\dir{-};(-10,0)*{\cdots}**\dir{-};
(0,0)*{\newmoon}**\dir{-};
(10,0)*{\otimes}**\dir{-};(-30,-3.5)*{\scriptstyle 1};(-20,-3.5)*{\scriptstyle 2};(0,-3.5)*{\scriptstyle n-1};
(10,-3.5)*{\scriptstyle n};
(-30,2)\ar@/_/@{<-->} (10,2);
\endxy$$ &&\\\hline
sBI, \ $n\geq 1$& $$\xy 
(-10,0)*{\otimes};
(0,0)*{\newmoon}**\dir{-};
(10,0)*{\cdots}**\dir{-};
(20,0)*{\newmoon}**\dir{-};(30,0)*{\newmoon}**\dir{=};
(25,0)*{>};
(-10,-4)*{\scriptstyle 0};(0,-4)*{\scriptstyle 1};(20,-4)*{\scriptstyle n-1};(30,-4)*{\scriptstyle n};
\endxy$$& sGI &$$ \xy(0,5)*{};%(-15,0)*{(\star)};
{\ar@3{-}(0,0)*{\newmoon};(10,0)*{\newmoon}};(5,0)*{<};(-10,0)*{\otimes};(0,0)*{\newmoon}**\dir{-};(-10,-4)*{\scriptstyle 1};(0,-4)*{\scriptstyle 2};(10,-4)*{\scriptstyle 3};
\endxy$$\\ 
sBII, \ $n\geq 1$ & $$\xy 
(-30,-6)*{\ };
(-20,0)*{\newmoon};
(-10,0)*{\otimes}**\dir{-};
(0,0)*{\newmoon}**\dir{-};
(10,0)*{\cdots}**\dir{-};
(20,0)*{\newmoon}**\dir{-};(30,0)*{\newmoon}**\dir{=};
(25,0)*{>};
(-10,-4)*{\scriptstyle 0};(0,-4)*{\scriptstyle 1};(20,-4)*{\scriptstyle n-1};(30,-4)*{\scriptstyle n};(-20,-4)*{\scriptstyle -1}
\endxy$$& &\\
\hline
sCI$,\ n\geq 1$   &$$\xy 
(-10,0)*{\otimes};
(0,0)*{\newmoon}**\dir{-};
(10,0)*{\cdots}**\dir{-};
(20,0)*{\newmoon}**\dir{-};(30,0)*{\newmoon}**\dir{=};
(25,0)*{<};
(-10,-4)*{\scriptstyle 0};(0,-4)*{\scriptstyle 1};(20,-4)*{\scriptstyle n-1};(30,-4)*{\scriptstyle n};
\endxy$$  &sD$\af$I, $\af\in \N$&
$$\xy
(0,0)*{\otimes};(7,5)*{\newmoon}**\dir{-};(0,0)*{\otimes};(7,-5)*{\newmoon}**\dir{-};(2,4)*{\scriptstyle -1};(2,-4)*{\scriptstyle 1+\af};
(-3,0)*{\scriptstyle 1};(10,5)*{\scriptstyle 2};(10,-5)*{\scriptstyle 3};
(0,8)*{};(0,-8)*{};
\endxy$$   \\ 
sCII, \ $n\geq 1$ & $$\xy 
(-20,0)*{\newmoon};
(-10,0)*{\otimes}**\dir{-};
(0,0)*{\newmoon}**\dir{-};
(10,0)*{\cdots}**\dir{-};
(20,0)*{\newmoon}**\dir{-};(30,0)*{\newmoon}**\dir{=};
(25,0)*{<};
(-10,-4)*{\scriptstyle 0};(0,-4)*{\scriptstyle 1};(20,-4)*{\scriptstyle n-1};(30,-4)*{\scriptstyle n};(-20,-4)*{\scriptstyle -1};
\endxy$$& sD$\af$II, $\af\in \N$& $$\xy
(0,0)*{\otimes};(7,5)*{\newmoon}**\dir{-};(0,0)*{\otimes};(7,-5)*{\newmoon}**\dir{-};(2,4)*{\scriptstyle -\af};(2,-4)*{\scriptstyle 1+\af};
(-3,0)*{\scriptstyle 1};(10,5)*{\scriptstyle 2};(10,-5)*{\scriptstyle 3};
(0,8)*{};(0,-8)*{};
\endxy$$\\ \hline
sDI &
$$\xy 
(10,0)*{\newmoon};
(20,0)*{\otimes}**\dir{-};
(27,4)*{\newmoon}**\dir{-};
(20,0)*{\otimes};(27,-4)*{\newmoon}**\dir{-};(10,-4)*{\scriptstyle 1};(20,-4)*{\scriptstyle 2};
(32,4)*{\scriptstyle 3};(32,-4)*{\scriptstyle 4};
\endxy$$ 
\\
sDII$,n \geq 3$&
$$\xy (-10,0)*{\otimes};
(0,0)*{\newmoon}**\dir{-};
(10,0)*{\cdots}**\dir{-};
(20,0)*{\newmoon}**\dir{-};
(27,4)*{\newmoon}**\dir{-};
(20,0)*{\newmoon};(27,-4)*{\newmoon}**\dir{-};(-10,-4)*{\scriptstyle 0};(0,-4)*{\scriptstyle 1};(20,-4)*{\scriptstyle n-2};
(32,4)*{\scriptstyle n-1};(32,-4)*{\scriptstyle n};
\endxy$$
\\
sDIII,\ $n\geq 3$&
$$\xy 
(-20,0)*{\newmoon};
(-10,0)*{\otimes}**\dir{-};
(0,0)*{\newmoon}**\dir{-};
(10,0)*{\cdots}**\dir{-};
(20,0)*{\newmoon}**\dir{-};
(27,4)*{\newmoon}**\dir{-};
(20,0)*{\newmoon};(27,-4)*{\newmoon}**\dir{-};(-10,-4)*{\scriptstyle 0};(0,-4)*{\scriptstyle 1};(20,-4)*{\scriptstyle n-2};
(32,4)*{\scriptstyle n-1};(32,-4)*{\scriptstyle n};(-20,-4)*{\scriptstyle -1};
\endxy$$
\\
sDIV&
$$\xy
(27,8)*{\ };
(27,-8)*{\ };
(20,0)*{\newmoon};(27,5)*{\otimes}**\dir{-};(20,0)*{\newmoon};(27,-5)*{\otimes}**\dir{-};
(27,5)*{\otimes};(27,-5)*{\otimes}**\dir{=};
(20,-4)*{\scriptstyle 1};
(34,5)*{\scriptstyle 2};
(34,-5)*{\scriptstyle 3};
(29,5) \ar@/_/@{<-->} (29,-5);
\endxy$$
\\ \hhline{|-|-|~|~}
\end{tabular}}
\end{table}
 By Remark~\ref{rem:subdiagram}, any (connected) subdiagram of a real rank 1 (super) Satake diagram $\Gamma$ that contains the white node must be the whole diagram $\Gamma$. Accordingly, the real odd/even rank 1 super Satake diagrams are building blocks of a general super Satake diagram (whose underlying Dynkin diagram is one of the basic types in Table~\ref{Table1}). On the other hand, in addition to having the underlying Dynkin diagrams in Table~\ref{Table1}, a general super Satake diagram has all its real odd rank 1 subdiagrams as in Table~\ref{Tableoddrank1} and real even rank 1 subdiagrams as in Table~\ref{table:evenrank1}.

\subsection{The automomorphism $\theta(\I,\tau)$}

%Other than $\Ad(G_\zero)$, we l
Let $T^\vee= \text{Hom}(\Z\Pi,\C^*)$ be the set of group homomorphisms from $\Z\Pi$ to $\C^*$. For $s\in T^\vee$, one defines $\Ad(s)\in \text{Aut}(\g)$ by 
\[
\Ad(s)|_\h=id,\quad \Ad(s)(v)=s(\alpha)v, \; \forall \alpha\in \Phi,v\in \g_\alpha.
\]

We now associate an automorphism of $\g$ to any super admissible pair following \cite{Ko14, BK19}. 
%Fix a total order $>$ on the set $\I$. 
Associated to a super admissible pair $(\I=\I_\circ \cup \I_\bu,\tau)$, we define an element $s(\I,\tau)\in T^\vee$ such that
\begin{equation}
\label{eq:sadmissible}
\begin{aligned}
 s(\I,\tau)(\alpha_i) &=1, & \text{ if } \tau i =i \in \I_\circ \text{ or }i\in \I_\bu,\\
    \frac{s(\I,\tau)(\alpha_i)}{s(\I,\tau)(\alpha_{\tau i})} &=(-1)^{\alpha_i{(2\rho_\bu^\vee)}}, &\text{ if } \tau i \neq i\in \I_\circ.
\end{aligned}
\end{equation}
Such an element always exists. As noted in \cite[Remark 5.2]{BK19}, the properties \eqref{eq:sadmissible} are the only ones that we need for $s$, and hence it is possible to choose $s(i)\in \{1,-1\}$, for all $i \in \I$.
Define
    \begin{equation}
    \label{eq:theta}
        \theta(\I,\tau) :=\Ad(s(\I,\tau))\circ \tau\circ \om\circ \Ad(m_\bu).
    \end{equation}
    
The following is a super generalization of \cite[Theorem 2.5]{Ko14}. 
\begin{proposition}
\label{prop:thetaorder}
    Let $(\I=\I_\circ \cup \I_\bu,\tau)$ be a super admissible pair. Then
    \begin{enumerate}
        \item 
        $\theta(\I,\tau)$ is an automorphism of Lie superalgebra $\g$ of order $2$ or $4$,
  \item 
  $\theta(I,\tau)$ is of order $2$ if and only if $\alpha_i(2\rho_\bu^\vee)\equiv 1 \pmod 2$ for all $i\in I_\circ\cap \Iodd$ such that $\tau i=i$.
    \end{enumerate}
\end{proposition}

\begin{proof}
By Proposition~\ref{prop:Adm}(2), we see that 
\begin{align} \label{Adm2fixehf}
\Ad(m_\bu)^2 \text{ fixes } e_i,f_i,h_i, \quad \text{ for } i\in\I_\bu, 
\end{align}
and moreover, we have 
\begin{align}  \label{Admsquare}
\begin{split}
    \Ad(m_\bu)^2(h_i)=  h_i, \quad
    \Ad(m_\bu)^2(e_i) &= 
    (-1)^{\alpha_i(2\rho_\bu^\vee)}e_i,  \\
    \Ad(m_\bu)^2(f_i) &= 
    (-1)^{\alpha_i(2\rho_\bu^\vee)}f_i,   \qquad \text{ for } i\in I_\circ.
    \end{split}
\end{align}

One checks directly that $\Ad(s(\I,\tau))$ commutes with $\Ad(m_\bu)$. Then, by Proposition~\ref{prop:Adm}(3) $\Ad(m_\bu)$ commutes with both $\tau$ and $\om$. Thus by \eqref{eq:theta} we have 
\begin{align}
\label{eq:thetasquare}
\theta(\I,\tau)^2= (\Ad(s(\I,\tau)) \circ \tau \circ \om)^2  \circ \Ad(m_\bu)^2. 
\end{align}
 On the other hand, by a direct computation using \eqref{eq:Chevalley}, \eqref{eq:tau} and \eqref{eq:sadmissible}, we see that 
\begin{align} \label{fixehf}
(\Ad(s(\I,\tau)) \circ \tau \circ \om)^2 \text{ fixes } e_i,f_i,h_i, \quad \text{ for }i\in\I_\bu, 
\end{align}
and moreover, we have
\begin{align}
\label{eq:square}
\begin{split}
(\Ad(s(\I,\tau)) \circ \tau \circ \om)^2(e_i) &=\begin{dcases}
(-1)^{\alpha_i(2\rho_\bu^\vee)}e_i, & \text{ if } \tau i\neq i,\\
(-1)^{p(i)}e_i, & \text{ if } \tau i=i,
\end{dcases}
\\
(\Ad(s(\I,\tau)) \circ \tau \circ \om)^2(h_i) &=h_i,
\\
(\Ad(s(\I,\tau)) \circ \tau \circ \om)^2(f_i) &=\begin{dcases}
(-1)^{\alpha_i(2\rho_\bu^\vee)}f_i, & \text{ if } \tau i\neq i,\\
(-1)^{p(i)}f_i, & \text{ if } \tau i=i, \quad \text{ for }i\in \I_\circ.
\end{dcases}
\end{split}
\end{align}
%for $i\in \I_\circ$.

It follows by \eqref{Adm2fixehf}, \eqref{eq:thetasquare} and \eqref{fixehf} that $\theta(\I,\tau)^2$ fixes $e_i, f_i, h_i$, for $i\in \I_\bu$. 

It follows by \eqref{Admsquare}, \eqref{eq:thetasquare}, \eqref{eq:square} and Definition~\ref{def:superad}(3) that $\theta(\I,\tau)^2$ fixes $e_i, f_i, h_i$, for $\tau i \neq i\in \I_\circ$ or $\tau i =i \in \I_\circ \cap \Ieven$. (Recall here $p(i)=0$ if and only if $i\in \Ieven$.)

For $\tau i =i \in \I_\circ \cap \Iodd$, it also follows by \eqref{Admsquare}, \eqref{eq:thetasquare} and \eqref{eq:square} again that 
\begin{align*}
\theta(\I,\tau)^2 (h_i) =h_i, \quad
\theta(\I,\tau)^2 (e_i) = -(-1)^{\alpha_i(2\rho_\bu^\vee)} e_i, \quad
\theta(\I,\tau)^2 (f_i) = -(-1)^{\alpha_i(2\rho_\bu^\vee)} f_i. 
\end{align*}

%When $i\in I_\circ$, we consider the following possibilities. If $\tau i\neq i$, then by \eqref{eq:sadmissible} and \eqref{eq:square} we still have $\theta(\I,\tau)^2(e_i)=e_i$, $\theta(\I,\tau)^2(f_i)=f_i$ and $\theta(\I,\tau)^2(h_i)=h_i$. If $\tau i=i$ and $i\in \Ieven\cap \I_\circ$, we see that $\alpha_i(2\rho_\bu^\vee)\equiv 0 \mod 2$ by Condition (3) of Definition~\ref{def:superad}. Hence $\theta(\I,\tau)^2(e_i)=e_i$, $\theta(\I,\tau)^2(f_i)=f_i$ and $\theta(\I,\tau)^2(h_i)=h_i$ as well. If $\tau i=i$ and $i\in \Iodd\cap \I_\circ$, then by \eqref{eq:thetasquare} and \eqref{eq:square}, we see that $\theta(\I,\tau)^2$ preserves $e_i$ and $f_i$ if and only if $(-1)^{\alpha_i(2\rho_\bu^\vee)}=-1$.

Therefore, we conclude that $\theta(I,\tau)$ is always of order $2$ or $4$; moreover, $\theta(I,\tau)$ has order $2$ if and only if $\alpha_i(2\rho_\bu^\vee)\equiv 1 \pmod 2$
%$(-1)^{\alpha_i(2\rho_\bu^\vee)}=-1$
for all $\tau i =i\in I_\circ\cap \Iodd$. 
\end{proof}

In our view, the (quantum) supersymmetric pairs associated with $\theta(\I,\tau)$ of order $2$ are more fundamental and interesting. Such automorphisms of order $2$ for a basic Lie superalgebra $\g$ can be reduced to those in the local real odd rank one cases which are completely determined below. 

\begin{corollary}
   Among the real odd rank one super Satake diagrams in Table~\ref{Tableoddrank1}, the automorphisms $\theta(\I,\tau)$ have order $2$ in exactly the following cases:
   {\rm sAI, sAIII$_{11}$, sAIV, sBII, sCI, sDI, sDIII, sDIV.}
\end{corollary}

\begin{proof}
    Follows by case-by-case computations of  $\alpha_i(2\rho_\bu^\vee)$. We skip the details. 
\end{proof}
Note that $\theta(\I,\tau)$ preserves $\h$ and  $\theta(\I,\tau)= -\tau w_\bu$ when acting on $\h$.

\begin{example}
\label{example:order}
    In Example~\ref{example:admissible}(1),(2),(3),  $\theta(\I,\tau)$ is of order $2$, $4$, $2$, respectively.
\end{example}

\begin{remark} \label{rem:2and4}
Linear automorphisms of order $2$ and $4$ of basic Lie superalgebras of type A--D were classified in \cite[Tables 2,\,5]{Ser83}. It is a nontrivial question to relate the automorphisms $\theta(\I,\tau)$ here to those constructed {\em loc. cit.} Our list and Serganova's lists cannot be the same because the split odd rank one super Satake diagram (see Example~\ref{example:nonadmissible}(2)) is excluded here. For example, the Chevalley involution $\omega$ (of order $4$) is not on our list. On the other hand, distinct super Satake diagrams can provide involutions of $\g$ in the same conjugacy class; see Example~\ref{ex:superA} below (also see Remark~\ref{rem:yesno}). 
\end{remark}

\subsection{Supersymmetric pairs}
Fix a super admissible pair $(\I=\I_\circ \cup \I_\bu,\tau)$ and recall $\theta=\theta(\I,\tau)$ from \eqref{eq:theta}. We denote by $\mathfrak t$ the subalgebra of $\mathfrak g$ generated by 
\begin{align} \label{eq:t}
    \set{e_i,f_i\mid i\in \I_\bu}\cup \set{h\in \h \mid \theta(h)=h} \cup \set{f_i+\theta(f_i)\mid i\in \I_\circ}.
\end{align}
Clearly we have that $\t +\mathfrak b^+ =\g$. The corresponding generators of the universal enveloping algebra $\U(\t)$ can be modified
by some constant terms as follows.
\begin{lemma}
   Let $\kappa_i\in \C$, for $i\in \I_\circ$. The algebra $\U(\t)$ is generated by
    \begin{align*}
    \{e_i,f_i\mid i\in \I_\bu\}\cup \{h\in \h\mid \theta(h)=h \}\cup \{f_i+\theta(f_i)+\kappa_i\mid i\in \I_\circ\}.
\end{align*}
\end{lemma}

We will refer to $(\g,\t)$ as the {\em supersymmetric pair} associated with $(\I=\I_\circ \cup \I_\bu,\tau)$. When $\g$ is semisimple or Kac-Moody Lie algebra, $\t$ is a fixed point subalgebra by an involution, and the generators and relations of the subalgebra $\t$ were first given in \cite[Theorem 1.31]{Ber89}. Note that $\theta=\theta(\I,\tau)$ is of order $2$ if and only if $\mathfrak t$ is the $\theta$-fixed point subalgebra of $\g$. We will see, from the examples below, that (the even part of) $\t$ may not be reductive. This super phenomenon is reminiscent of the non-reductive Lie algebras associated to generalized Satake diagrams studied in \cite{RV20}.
Set $[a,b]=ab-(-1)^{p(a)p(b)}ba$
for homogeneous $a,b\in \U(\g)$ and extend $[\cdot,\cdot]$ bi-linearly.

%When $\g$ is a semisimple Lie algebra, a complete set of defining relations of $\U(\t)$ was obtained in \cite{Ber89}. 

\begin{example}
    (1) Recall from Example~\ref{example:admissible}(1) the following diagram
    {\rm
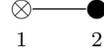
$$\hspace{.75in}\xy
(-5,0)*{\otimes};(5,0)*{\newmoon}**\dir{-};
(-5,-4)*{\scriptstyle 1};(5,-4)*{\scriptstyle 2};
\endxy$$}
where $\mathfrak g=\gl(1|2)$, $\I_\bu=\set{2}$ and $\tau=id$. By Example~\ref{example:order}, $\theta$ has order $2$ and $\t=\g^\theta$.
Then $\t$ is generated by $e_2,f_2,h_2,f_1+[e_2,e_1]$, and $\t\cong \mathfrak{osp}(1|2)$.

(2) Recall from Example~\ref{example:admissible}(2)  the following diagram 
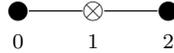
\rm{$$
 \xy
(-10,.75)*{};
(-10,0)*{\newmoon};(0,0)*{\otimes}**\dir{-};(10,0)*{\newmoon}**\dir{-};(-10,-4)*{\scriptstyle 0};
(0,-4)*{\scriptstyle 1};(10,-4)*{\scriptstyle 2};
\endxy
$$} 
\noindent where $\mathfrak g=\gl(2|2)$, $\I_\bu=\set{0,2}$ and $\tau=id$. By Example~\ref{example:order}, $\theta$ has order $4$, and hence $\t$ is not the fixed point subalgebra $\g^\theta$.  In this case $\t$ is generated by $e_0,e_2,f_0,f_2,f_1+[e_0,[e_2,e_1]]$ and certain Cartan elements. We observe that the even part $\mathfrak{t}_{\overline{0}} \cong \mathfrak{sl}(2) \oplus \mathfrak{sl}(2)$ and the odd part $\mathfrak{t}_{\overline{1}}$ has dimension $4$. There is no such simple Lie superalgebra.
\end{example}

\begin{example} [General super type A Satake diagrams with $\tau =\text{id}$]
 \label{ex:superA}
Let $a\in \N$ and $m,n_1, \ldots, n_a,k \in \N$. Set  
\[
N_0 =m +n_1 + \ldots + n_a +k, \qquad 
N_1=2(a+1), \qquad N =N_0 +N_1. 
\]
Let $V =\C^{m|2|n_1|2|n_2|2 | \ldots | n_a|2|k}$ be a vector superspace, which is isomorphic to $\C^{N_0|N_1}$. Denote by $I_n$ the identity matrix and denote 
$  S= \begin{pmatrix}
    0 & 1  \\
    -1 & 0
\end{pmatrix}.
$ We introduce an $N\times N$-matrix in the $(m|2|n_1|2|n_2|2 | \ldots | n_a|2|k)$-block diagonal form 
\[
J= J_{m,n_1, \ldots, n_a,k} ={\rm diag } (I_m, S, I_{n_1}, S, I_{n_2}, S, \ldots, I_{n_a}, S, I_k).
\]
and regarded it as an element in $\End (V)$. In particular, $J$ can be used to define a non-degenerate supersymmetric bilinear form on $V$.

Introduce the following shorthand notations of local diagrams 
{\rm
\[
{}_m\XBox :=\xy
(-5,0)*{\fullmoon};(0,0)*{\fullmoon}**\dir{-};(5,0)*{\cdots}**\dir{-};(10,0)*{\fullmoon}**\dir{-};(15,0)*{\otimes}**\dir{-};
(-5,-4)*{\scriptstyle m};(0,-4)*{\scriptstyle m-1};(10,-4)*{\scriptstyle 2};(15,-4)*{\scriptstyle 1};
\endxy\ ,
\quad
\XBox_k :=\xy
(-5,0)*{\otimes};(0,0)*{\fullmoon}**\dir{-};(5,0)*{\cdots}**\dir{-};(10,0)*{\fullmoon}**\dir{-};(15,0)*{\fullmoon}**\dir{-};
(-5,-4)*{\scriptstyle 1};(0,-4)*{\scriptstyle 2};(10,-4)*{\scriptstyle k-1};(15,-4)*{\scriptstyle k};
\endxy\ ,
\quad
\davidsstar_n :=\xy
(-5,0)*{\otimes};(0,0)*{\fullmoon}**\dir{-};(5,0)*{\cdots}**\dir{-};(10,0)*{\fullmoon}**\dir{-};(15,0)*{\otimes}**\dir{-};
(0,-4)*{\scriptstyle 1};(10,-4)*{\scriptstyle n-1};
\endxy
\]
}
for $m,k,n\geq 1$. For example, we have {\rm$\davidsstar_1= \xy
(0,0)*{\otimes};(5,0)*{\otimes}**\dir{-};
\endxy$} and {\rm$\XBox_1={}_1\XBox= \otimes$}. We also set {\rm$\XBox_0={}_0\XBox=\varnothing$} (which are to be removed) and {\rm$\davidsstar_0= \fullmoon$} (the even split rank one Satake diagram). 

Consider the following diagram
{\rm{
\begin{equation}
    \label{eq:splitA}
\xy
(0,0)*{{}_m\XBox};(8,0)*{\newmoon}**\dir{-};(16,0)*{\davidsstar_{n_1}}**\dir{-};(24,0)*{\newmoon}**\dir{-};(32,0)*{\davidsstar_{n_2}}**\dir{-};(40,0)*{\newmoon}**\dir{-};(47,0)*{\cdots}**\dir{-};(56,0)*{\davidsstar_{n_a}}**\dir{-};(64,0)*{\newmoon}**\dir{-};(72,0)*{\XBox_k}**\dir{-};
\endxy
\end{equation} }}
with the underlying Lie superalgebra $\g=\gl(m|2|n_1|2|n_2|2|\cdots|n_a|2|k) \cong \gl(N_0|N_1)$ and $\tau=id$. By Table \ref{table:evenrank1} (Type AII$_{3}$), Example~\ref{example:admissible}(1) and Proposition~\ref{prop:rank1}, we see that \eqref{eq:splitA} is super admissible, and every {\em genuinely super} type sA Satake diagram with $\tau =\text{id}$ must be of the form \eqref{eq:splitA}. By Proposition~\ref{prop:thetaorder}, the automorphism $\theta$ of $\g$ associated with \eqref{eq:splitA} has order $2$. 

Note that the equation 
\begin{align} \label{JX}
    X^{\rm{st}}J+JX=0
\end{align}
defines a subalgebra of $\gl(N_0|N_1)$ which can be identified with the ortho-symplectic Lie superalgebra $\mathfrak{osp}(N_0|N_1)$ (see \cite[\S 1.1.3]{CW12}), where $X^{\rm{st}}$ denotes the supertranspose \cite[(1.5)]{CW12}.
By checking on the generators of $\t$, one shows that every element $X$ in the fixed-point subalgebra $\t=\g^\theta$ satisfies
the equation \eqref{JX}. One then shows that $\t= \mathfrak{osp}(N_0|N_1)$ by a counting argument (either by a brute force computation or by the classical version of the basis Theorem \ref{thm:basis}). 

It follows that $\theta = -\Ad(J^{-1}) \circ \rm{st}$. In particular, various super Satake diagrams can give rise to isomorphic supersymmetric pairs (e.g. associated with different diagrams \eqref{eq:splitA} for fixed $N_0, N_1$). Compare Remark~\ref{rem:yesno} below on the quantum setting.
\end{example}

\begin{remark}
Similar to \eqref{eq:splitA}, one can draw precisely general super Satake diagrams of type sB, sC, sD (and also type sA with $\tau \neq \text{id}$), based on the corresponding odd real rank one super Satake diagrams in Table \ref{Tableoddrank1}. 
\end{remark}

%\textcolor{blue}
%    {\begin{remark}
%To  exclude the two diagrams in \eqref{exclude}, we may add two conditions
%\[
%$\text{ if } (i,j,k)\in \Iodd\times\I_\bu\times \I_\bu, \tau i=i, w_\bu(\alpha_i)=\alpha_j+\alpha_i+\alpha_k, \text{ then } -w_\bu(\alpha_j+\alpha_k)\neq \alpha_j+\alpha_k 
%\]
%and
%\[
%\text{ if $(i,j)\in \Iodd\times \I_\bu$, $\tau i=i$ and $w_\bu(\alpha_i)=\alpha_i+\alpha_j$, then $a_{ij}\neq \pm 3$}
%\]
%I think after weakening (3) and adding these two, our notion of super admissible pairs corresponds to $\t\cap \h=\h^\theta$. 
%\end{remark}    }

\section{Quantization of the automorphism $\theta(\I,\tau)$} 
\label{sec:quantum involution}
In this section we review basic constructions of quantum supergroups and construct a quantum analogue of the automorphism $\theta(I,\tau)$ from \eqref{eq:theta}.

\subsection{Quantum supergroups}

 Recall that $\g=\n^-\oplus \h \oplus \n^+$ is the triangular decomposition of a basic Lie superalgebra $\g$. Set 
 \[
 q_i:=q^{d_i},
 \]
 for $d_i$ see \eqref{di}. Denote the quantum integers and quantum binomial coefficients by
\[
[a]_i=\frac{q_i^a-q_i^{-a}}{q_i-q_i^{-1}},
\qquad
\qbinom{a}{k}_{i} =\frac{[a]_i [a-1]_i \ldots [a-k+1]_i}{[k]_i!}
\] 
for $a\in \Z, k \in \N,i\in \I$.
%It will also be convenient for us to introduce the following notation. We will say $i,j\in I$ are {\em connected} if they are joined by some edges in the corresponding Dynkin diagram and write $i\sim j$. Otherwise, we say $i,j\in I$ are {\em not connected} and write $i\nsim j$.

Recall $X=\Z\Pi$ and $Y=\Z\Pi^\vee$. A quantum supergroup $\breve{\U}_q(\g)$ (cf. \cite{Ya94}) is the unital associative algebra over $\C(q)$ with generators $ E_i,\  F_i,\  (i\in I),\ K_h\ (h\in Y)$ which satisfy the following relations, for $h,h'\in Y, i,k\in I$:
\begin{equation}
\begin{aligned}
\label{eq:Urelation}
    K_0=1,\quad K_{h}K_{h'} & =K_{h+h'},\\
      K_h E_i=q^{\alpha_i(h)}E_i K_h,  & \qquad
  K_h F_i=q^{-\alpha_i(h)}F_i K_h,\\
      E_i F_k-(-1)^{p(i)p(k)}F_k E_i &=\delta_{i,k}\frac{K_i-K_i^{-1}}{q_i- q_i^{-1}},
      \quad \text{ where } K_i :=K_{h_i}^{d_i},
    \end{aligned}
\end{equation}
%where $K_i :=K_{h_i}^{d_i}$,
together with the Serre-type relations among $E_i$'s (and respectively, among $F_i$'s) specified by the non-commutative Serre polynomials listed in Table~\ref{TableSerrePolyn}, see also \cite{CHW16}. 

The algebra $\breve{\U}_q(\g)$ is a Hopf superalgebra but not a Hopf algebra, and there is a simple way to modify this below if one prefers to work with Hopf algebras (cf. \cite{Ya94}). Define an algebra involution $\new$ of parity $0$ on $\breve{\U}_q(\g)$ as follows:
\begin{equation}
\label{eq:new}
    \new(K_h)=K_h,\quad \new(E_j)=(-1)^{p(j)}E_j\text{ and }
    \new(F_j)=(-1)^{p(j)}F_j,\quad \forall h\in Y,j\in \I.
\end{equation}

Let 
\[
\U=\breve{\U}_q(\g)\oplus \breve{\U}_q(\g)\new.
\] 
Then we extend the algebra structure on $\breve{\U}_q(\g)$ to $\U$ by declaring  
\begin{equation}
\label{eq:newrelation}
\new^2=1,\quad x\cdot \new= \new \cdot \new(x), \quad \forall x\in \breve{\U}_q(\g).
\end{equation}
One verifies that $\U$ is a Hopf algebra with the coproduct $\Delta$, counit $\epsilon$ and antipode $S$ given by\begin{equation}
\label{eq:Hopf}
    \begin{aligned}
    &\Delta(E_i)=E_i\otimes1 +\new^{p(i)}K_i\otimes E_i, &\epsilon(E_i)=0,\quad &S(E_i)=-\new^{p(i)}K_i^{-1}E_i.\\
    &\Delta(F_i)=F_i\otimes K_i^{-1}+\new^{p(i)}\otimes F_i, &\epsilon(F_i)=0,\quad&S(F_i)=-\new^{p(i)}F_iK_i.\\
    &\Delta(K_h)=K_h\otimes K_h, &\epsilon(K_h)=1,\quad &S(K_h)=K_{-h},\\
    &\Delta(\new)=\new\otimes \new,&\epsilon(\new)=1,\quad &S(\new)=\new.
    \end{aligned}
    \end{equation}

Let $\U^+,\U^-,\U^0$ denote the subalgebra of $\U$ generated by $\set{E_i\mid i\in \I},\ \set{F_i\mid i\in \I}$ and $ \set{\new,K_h\mid h\in Y}$ respectively. As in \cite{Lus93}, the multiplication map gives a triangular decomposition of $\U$: $\U\cong \U^+\otimes \U^0\otimes \U^-.$ We denote
 \[ 
 \U_\beta:=\set{u\in \U\mid K_iuK_i^{-1}=q^{(\alpha_i,\beta)}u,\ \forall i\in \I},
 \qquad \text{ for } \beta\in X.
 \]

 The Chevalley involution, denoted by $\om$, is an automophism of the algebra $\U$ such that
 \[
\om(E_i)=-F_i,\quad \om(F_i)=-(-1)^{p(i)}E_i,\quad \om (K_h)=K_{-h},\quad \om(\new)=\new.
 \]
The algebra $\U$ contains the following subalgebra 
\begin{align}
\label{eq:Uprime}
    \U':= \langle E_i, F_i, K_i^{\pm 1}, \new \mid i\in \I \rangle
\end{align}
which is actually a Hopf subalgebra by inspection. 
Let $\C(q)[X]$ be the group algebra of the root lattice $X$. There is an inclusion of algebras $\C(q)[X]\hookrightarrow {\U^0}$ such that $\alpha_i\mapsto K_i$, and hence we write
\[K_\beta:=\prod_{i\in \I}K_i^{n_i}
\qquad
\text{for } \beta=\sum_{i\in \I}n_i\alpha_i \in X.
\]

\subsection{Braid group symmetries}
Let $\sigma$ denote the algebra anti-automorphism of $\U$ such that 
\begin{equation}
\label{eq:sigma}
\begin{aligned}\sigma(E_j)=E_j,\quad\sigma(F_j)=F_j,\quad \sigma(K_{h_i})=(-1)^{\frac{p(i)}{d_i}} K_{-h_i},\quad \sigma(\new)=\new.
\end{aligned}
\end{equation}
Here and below we fix some roots of unity $(-1)^{\frac{1}{d_i}}$ for $i\in \Iodd$. It follows from $K_i =K_{h_i}^{d_i}$ that 
\[
\sigma(K_{i})=(-1)^{p(i)} K_i^{-1},
\]
and $\sigma$ restricts to an anti-involution on the subalgebra $\U'$ given in \eqref{eq:Uprime}. 

Recall that Lusztig \cite[\S 37.1.3]{Lus93} has defined automorphisms $T_{i,e}'$ and $T_{i,e}''$ of quantum groups. For $i\in \Ieven$ and $e\in \{\pm 1\}$,  $T_{i,e}'$ and $T_{i,e}''$ naturally extend to be algebra automorphisms of $\U$; cf. \cite{Ya99}. In particular, we have the following formulas for the automorphisms $T_{i,e}',\ T_{i,e}'':\U\to\U$: 
\begin{equation}
\label{eq:Ti'}
    \begin{aligned}
     &T_{i,e}' (E_i)=- K^{e}_{i}F_i, \quad T_{i,e}'(E_j)=  \sum_{r+s=|a_{ij}|} (-1)^r  q_i^{er} E^{(r)}_i E_j E^{(s)}_i \quad\mathrm{ for }~j\neq i,\\
        &T_{i,e}' (F_i)=- E_iK_{i}^{-e}, \quad T_{i,e}'(F_j)=  \sum_{r+s=|a_{ij}|} (-1)^r  q_i^{-er} F^{(s)}_i F_j F^{(r)}_i \quad\mathrm{ for }~j\neq i, \\
        & T_{i,e}'(K_h)=K_{r_i(h)},\quad T_{i,e}'(\new)=\new,
    \end{aligned}
\end{equation}
and
\begin{equation} \label{eq:Ti}
    \begin{aligned}
        &T_{i,-e}'' (E_i)=-F_i K^{-e}_{i}, \quad T_{i,-e}''(E_j)=  \sum_{r+s=|a_{ij}|} (-1)^r q_i^{er} E^{(s)}_i E_j E^{(r)}_i \quad\mathrm{ for }~j\neq i,\\
        &T_{i,-e}'' (F_i)=- K_{i}^{e}E_i, \quad T_{i,-e}''(F_j)=  \sum_{r+s=|a_{ij}|} (-1)^rq_i^{-er} F^{(r)}_i F_j F^{(s)}_i \quad\mathrm{ for }~j\neq i, \\
        & T_{i,-e}''(K_h)=K_{r_i(h)},\quad T_{i,-e}''(\new)=\new,
    \end{aligned}
\end{equation}
where $j\in \I, h\in Y$, $E^{(t)}_i=E_i^{t}/[t]_i!,\ F^{(t)}_i=F_i^{t}/[t]_i!$.

We often write $T_i:=T_{i,+1}''$ for simplicity. The following is well known in the quantum group setting \cite{Lus93}.

\begin{lemma}
\label{lemma:Ti-1}
    We have $T_i^{-1}=\sigma \circ T_i\circ \sigma^{-1} =T_{i,-1}':\U \to \U$ for any $i\in \I_\bu$.
\end{lemma}
\begin{proof}
    The formulas for $T_i^{-1}$ can be found in \cite[Lemma 6.6]{H10}. Then the lemma follows by checking directly on the generators of $\U$. 
\end{proof}
According to \cite[\S 6.3]{H10} (cf. \cite{Lus93}), the automorphisms $T_i$, for $i \in \Ieven$, satisfy the braid group relations. Consequently, for a given super admissible pair $(\I=\I_\circ \cup \I_\bu,\tau)$, we can define an algebra automorphism $T_w=T_{i_1}\cdots T_{i_k}$ for any reduced word $w=r_{i_1}\cdots r_{i_k}\in W_\bu$. In particular, one has $T_w(K_\mu)=K_{w(\mu)}$ for any $\mu\in Y$.

Denoted by $\U_{\bu}= \U^+_{\bu}\U^0_{\bu}\U^-_{\bu}$ the quantum group associated with $\I_\bu\subset \Ieven$, where  $\U^+_{\bu},\ \U^-_{\bu}$ and $\U^0_{\bu}$ denote the subalgebras generated by $\set{E_i\mid i\in \I_\bu},\ \set{F_i\mid i\in \I_\bu}$ and $\set{K_i^{\pm 1}\mid i\in \I_\bu}$, respectively; note that $\Ub^0$ does not contain $\new$. %We can view $\U_\bu$ as a subalgebra of $\Ui$ naturally. 

\subsection{Quantum analogue of $\theta(\I,\tau)$}

In this subsection we define a quantum analogue of $\theta(\I,\tau)$ given in  \eqref{eq:theta}, following \cite[\S 3.3]{Ko14}.

We first present some results by K\'eb\'e \cite{K99} and their straightforward super generalization in our setting. For any $j \in \I_\circ$, the subspace $\ad(\Ub)(E_j)$ is finite dimensional and contained in $\U^+$. This follows from the triangular decomposition of $\Ub$, the fact that $\ad(F_k)(E_j)=0$ for any $k\in \I_\bu$, $\ad(E_k)(E_j)\in \U^+$ and the quantum Serre relations of type (N-ISO) in Table~\ref{TableSerrePolyn}. Moreover, the $\ad(\U_\bu)$-module $\ad(\Ub)(E_j)$ is irreducible. Define $V_\bu^+$ to be the subalgebra generated by the elements of all
the finite dimensional subspace $\ad(\U_\bu)(E_j)$ for $j\in \I_\circ$. As proved in \cite{K99} we have a vector space isomorphism
\begin{equation}
\label{eq:VUb}
    \U^+\cong V_\bu^+\otimes \U_\bu^+.
\end{equation}

 The following lemma determines the action of $T_{w_\bu}$ on $E_j$ for $j\in \I_\circ$ (regardless of the parity of $j$). 
 
\begin{lemma}{\rm \cite[Lemma 3.5]{Ko14}}\label{lemma:hwt}
Let $j\in \I_\circ$.
\begin{enumerate}
    \item 
 The space $\ad(\Ub)(E_j)$ of $\U^+$ is a finite dimensional, irreducible submodule of $\U$ with highest weight vector $T_{w_\bu}(E_j)$ and lowest weight vector $E_j$.
 \item 
 The space $\ad(\Ub)(F_jK_j)$ of $S(\U^-)$ is a finite dimensional, irreducible submodule of $\U$ with highest weight vector $F_jK_j$ and lowest weight vector $T_{w_\bu}(F_jK_j)$.
\end{enumerate}
\end{lemma}

%Recall from Definition~\ref{def:superad} that $\tau\in \text{Aut}(\I,\I_\bu)$ and the action of $\tau$ on $\I_\bu$ coincides with the action of $-w_\bu$. The following well-known lemma determines the action of $T_{w_\bu}$ on $\U_{\bu}$. 

%\begin{lemma} 
%    For $i\in \I_\bu$, we have
%    \begin{align*}
%        &T_{w_\bu}(E_i)=-F_{\tau i}K_{\tau i},\quad T_{w_\bu}(F_i)=-K_{\tau i}^{-1}E_{\tau i},\quad  T_{w_\bu}(K_i)=-K_{\tau i}^{-1},\\
%        &T^{-1}_{w_\bu}(E_i)=-K_{\tau i}^{-1}F_{\tau i},\quad T_{w_\bu}^{-1}(F_i)=-E_{\tau i}K_{\tau i},\quad  T_{w_\bu}^{-1}(K_i)=-K_{\tau i}^{-1}.
%    \end{align*}
%\end{lemma}

Observe that if $i\neq \tau i$, then  we have  
\begin{equation}
    \label{eq:qiqti}
    q_i=\begin{dcases}
    q_{\tau i}, & \text{ if }  i=\tau i,\\
        q_{\tau i}, & \text{ if } i\neq\tau i, i\in \Ieven,\\
        q_{\tau i}^{-1}, & \text{ if }  i\neq\tau i, i\in \Iodd \; (\text{and thus }i\in \I_{\iso}).
    \end{dcases}
\end{equation}
Note that \eqref{eq:qiqti} will be used below and in Lemma~\ref{lemma:ineqtaui}.

The element $s(\I,\tau)$ in \eqref{eq:sadmissible} defines an algebra automorphism of $\U$ by setting 
\[
\Ad\big(s(\I,\tau)\big)(u)=s(\I,\tau)(\beta)u,\ \forall u\in \breve{\U}(\g)_\beta,\ \beta\in X,
\qquad
\Ad\big(s(\I,\tau)\big)(\new)=\new.
\]

%\textcolor{red}{For the involution $\tau$, the case when $\tau i=i$ is fine. When $\tau i\neq i$, we always have $d_i=(-1)^{p(i)}d_{\tau i}$. Then we need $K_{h_i}\mapsto (-1)^{\frac{p(i)}{d_i}}K_{h_{\tau i}}^{\frac{d_{\tau i}}{d_i}}=(-1)^{\frac{p(i)}{d_i}}K_{h_{\tau i}}^{(-1)^{p(i)}}$, this will satisfy the middle relation in \eqref{eq:Urelation}In fact, by Table 3, when $\tau i\neq i$ I think we always have $d_i,d_{\tau i}=\pm 1$.}

Note that $\tau\in \text{Aut}(A,\I_\bu)$ gives rises to an algebra isomorphism of $\U$, denoted again by $\tau$, such that 
\begin{align}
\label{tau:U}
\begin{split}
\tau(E_i) &=\begin{dcases}
E_{i},&\text{ if }\tau i=i\\
(-1)^{p(i)}E_{\tau i},&\text{ if }\tau i\neq i,
\end{dcases}
 \\
\tau: \U \longrightarrow \U, \qquad\qquad
\tau(F_i)& =F_{\tau i}, \qquad \tau(\new)=\new,
\\
\tau(K_{h_i}) & =
\begin{dcases}
   K_{h_{i}}, & \text{ if }\tau i=i\\
    K_{h_{\tau i}}^{(-1)^{p(i)}}, & \text{ if }\tau i\neq i.
\end{dcases}
\end{split}
\end{align}
In fact, when $\tau i\neq i$, we have $d_i=(-1)^{p(i)}d_{\tau i}$; see \eqref{eq:qiqti}. Hence it follows from $K_i =K_{h_i}^{d_i}$ that
\begin{align}
\label{eq:tauKi}
\tau(K_i)=K_{\tau i}, \text{ for all }i\in \I,
\end{align}
and so $\tau$ restricts to an involution on the subalgebra $\U'$ in \eqref{eq:Uprime}.  

Moreover, define an algebra automorphism $\dag$ of $\U$ by
\begin{equation}
    \label{eq:dag}
    \dag(E_i)=E_iK_i,\quad \dag(F_i)=K_i^{-1}F_i,\quad \dag(K_h)=K_h,\quad
    \dag(\new)=\new,\quad \forall i\in\I.
\end{equation}
The $q$-analogue of $\Ad(m_\bu)$ is given by 
\begin{equation}
\label{eq:Tbu}
    T_\bu=T_{w_\bu}\circ \dag.
\end{equation}

\begin{lemma} {\rm (cf. \cite[Lemma 4.1]{Ko14})}
\label{lemma:Tbu}
We have
\begin{enumerate}
    \item $T_\bu\circ \tau \circ \om |_{\Ub}=id$;
    \item $T_\bu\circ \tau=\tau\circ T_\bu:\U\to\U.$
\end{enumerate}
\end{lemma}

As a quantum analogue of $\theta(\I,\tau)$, the automorphism $\theta_q(\I,\tau)$ of $\U$ is defined by
\begin{equation}
\label{eq:thetaq}
    \theta_q(\I,\tau)=\Ad(s(\I,\tau))\circ T_\bu\circ \tau \circ \om.
\end{equation}

The automorphism $\theta_q(\I,\tau)$ is not involutive, but it retains crucial properties of $\theta(\I,\tau)$. In particular, it follows from Lemma~\ref{lemma:hwt} that for any $i\in \I_\circ$ there exist 
\begin{equation*}
Z_{i,\bu}^-=F_{i_1}\cdots F_{i_r}\in \U_\bu^-,\quad Z_{i,\bu}^+=E_{i_1}\cdots E_{i_r}\in \U_\bu^+,
\end{equation*}
for some $r\ge 1$, and scalars $a_i^\pm\in \C(q)^*$ such that
\begin{equation}
\label{eq:adZ+}
T_{w_\bu}^{-1}(F_iK_i)=a_i^-\ad(Z_{i,\bu}^-)(F_iK_i), \quad
T_{w_\bu}(E_i)=a_i^+ad(Z_{i,\bu}^+)(E_i).
\end{equation}
We often denote 
\[
\Theta=-w_\bu\circ \tau.
\]

\begin{proposition} {\rm (cf. \cite[Theorem 4.4]{Ko14})}
\label{prop:thetaq}
The following properties hold for the automorphism $\theta_q(\I,\tau)$:
\begin{enumerate}
    \item 
    $\theta_q(\I,\tau)|_{\Ub}$=id;
    \item 
    $\theta_q(\I,\tau)(K_\beta)=K_{\Theta(\beta)}$, for $\beta \in X$;
    \item 
    For any $i\in \I_\circ$, there exist $u_i,v_i\in \C(q)^*$ such that
    \begin{align*}
        \theta_q(\I,\tau)(E_i)&=-u_i\sigma(\ad(Z_{\tau i,\bu}^-)(F_{\tau i}K_{\tau i})),\\
        \theta_q(\I,\tau)(F_iK_i)&=-v_i\ad(Z_{i,\bu}^+)(F_{\tau i}K_{\tau i}).
    \end{align*}
\end{enumerate}
\end{proposition}

\begin{proof}
  The proof follows from the same argument as in \cite[Theorem 4.4]{Ko14}, with the help of Lemma~\ref{lemma:Ti-1} and Lemma~\ref{lemma:Tbu}.  
\end{proof}

\section{Quantum supersymmetric pairs}
\label{sec:QSP}

In this section we construct quantum supersymmetric pairs based on super Satake diagrams. Applying the projection technique from \cite{Let99, Let02} and \cite{Ko14}, we shall establish the quantum Iwasawa decomposition and a basis for $\Ui$ under the mild assumption that 
\begin{equation} \label{exclude}
    \text{\em the super Satake diagram is not type sFII in Table \ref{Tableoddrank1} and it does not contain}
\end{equation}
$$ \text{a subdiagram of the form }\quad
\xy
(-10,0)*{\newmoon};(0,0)*{\otimes}**\dir{-};(10,0)*{\newmoon}**\dir{-};(-10,-4)*{\scriptstyle i};
(0,-4)*{\scriptstyle j};(10,-4)*{\scriptstyle k};
\endxy .
$$
%i.e., type sAI${}_1$ with $m=n=1$ in Table \ref{Tableoddrank1}.

\subsection{$\imath$Quantum supergroups} 

In this subsection we give a definition of the quantum supersymmetric pair associated to a super Satake diagram $(\I=\I_\circ \cup \I_\bu,\tau)$.

Define $X^\io=\set{\alpha \in X\mid \Theta(\alpha)=\alpha}$ and $Y^\io=\{h\in Y\mid \Theta(h)=h \}$. As in \cite[Definition 5.1]{Ko14}, the quantum analogue of $f_i+\theta(f_i)+\kappa_i$ is defined to be 
\begin{equation}
\label{eq:Bi}
\begin{aligned}
    B_i&:= F_i-(-1)^{p(i)}\va_i\theta_q(F_iK_i)K_i^{-1}+\kappa_iK_i^{-1}\\
    &= F_i+\va_iT_{w_\bu}(E_{\tau i})K_i^{-1}+\kappa_iK_i^{-1} \qquad\qquad (\text{for }i\in \I_\circ),
    \end{aligned}
\end{equation}
for suitable $\va_i\in \C(q)^*$ and $\kappa_i\in \C(q)$.

\begin{definition}
\label{def:Ui}
    The algebra $\Ui$, with parameters $\va_i \in \C(q)^*, \kappa_i\in \C(q), \text{ for }i\in \I_\circ,$
    is the $\C(q)$-subalgebra of $\U$ generated by $\Uio$ together with the following elements:
    \begin{align*}
        B_i \; (i\in \I_\circ), \qquad
        E_i,\ F_i \; (i\in \I_\bu).
    \end{align*}
\end{definition}
The algebra $\Ui$ satisfies the relations (which follows from \eqref{eq:Urelation})
\begin{align*}
K_h  B_j &= q^{-\alpha_j(h)}  B_jK_h,\ \forall j\in I_\circ,  \\
K_h F_j &= q^{-\alpha_j(h)}  F_jK_h, \quad
K_h  E_j = q^{\alpha_j(h)}  E_jK_h,\quad \forall j\in I_\bu,  h \in Y^\io, \\
\new B_j &=(-1)^{p(j)}B_j \new,\quad \forall j\in I_\circ,
\end{align*}
and additional Serre type relations.

We extend the definition of $B_i$ by setting $B_i=F_i$ for $i\in I_\bu$. Let $\Uio$ denote the subalgebra generated by $\new$ and $K_\beta$ for all $ \beta\in Y^\io$. The following is a super generalization of  \cite[Proposition 5.2]{Ko14}.
\begin{proposition}
\label{prop:coideal}
    $\Ui$ is a right coideal subalgebra of $\U$.
\end{proposition}

\begin{proof}
    It suffices to show that 
    \begin{equation}
    \label{eq:deltabi}
        \Delta(B_i)\in \Ui\otimes \U
    \end{equation}
    for all $i\in \I_\circ$.
    It follows from Proposition~\ref{prop:thetaq} that 
    \begin{equation}
    \label{eq:Bi2}
        B_i=F_i-(-1)^{p(i)}\va_iv_i\ad(Z_{\tau i,\bu}^+)(E_{\tau i})K_i^{-1}+\kappa_iK_i^{-1}.
    \end{equation}
We use Sweedler notation to denote $\Delta(x) = x_{(1)} \otimes x_{(2)}$, for $x\in \U$, and define the left adjoint action of $\U$ on itself by
\begin{equation}
\label{eq:ad}    
\ad(x)(u)= x_{(1)}uS(x_{(2)}),\quad \text{ for all }x,u\in \U.
\end{equation}
 The  following standard formula holds:
\[
\Delta(\ad(x)(u))=x_{(1)}u_{(1)}S(u_{(3)})\otimes \ad(x_{(2)})(u_{(2)})
\]
for any $x,u\in \U$.

Applying this formula to $x=Z_{\tau i,\bu}^+$ and $u=E_{\tau i}$, we conclude that
    \begin{equation}
    \label{eq:deltaadZ}
        \Delta(\ad(Z_{\tau i,\bu}^+)(E_{\tau i}))-\ad(Z_{\tau i,\bu}^+)(E_{\tau i})\otimes 1\in \Ub^+\Ub^0\new^{p(\tau i)}K_{\tau i}\otimes \U
    \end{equation}
Formulas \eqref{eq:Bi2} and \eqref{eq:deltaadZ} together imply that
    \begin{equation}
    \label{eq:deltabi-}
        \Delta(B_i)-B_i\otimes K_i^{-1}\in \Ub^+\Uio\otimes \U.
    \end{equation}
    Hence \eqref{eq:deltabi} holds for all $i\in \I_\circ$.
\end{proof}
We shall refer to the algebra $\Ui$ as an $\imath$quantum supergroup and $(\U,\Ui)$ as a quantum supersymmetric pair associated with $(\I =\I_\circ \cup \I_\bu,\tau)$.

\subsection{Conditions on parameters}

In the remainder of this subsection, we will discuss the constraints on the parameters so that the algebra $\Ui$ is a flat deformation of $\U(\t)$; see \eqref{eq:t} for $\t$. Conditions on the parameters first come from the desired property
\begin{equation}
\label{eq:cap}
    \Ui\cap \U^0=\Uio.
\end{equation}
Define 
\begin{equation}
\label{eq:Ins}
    \I_{\ns}=\set{i\in \I_\circ\mid \tau i=i,\ \alpha_i(h_j)=0,\ \forall j\in \I_\bu}.
\end{equation}
We observe that $\I_{\ns}\cap \I_{\iso}=\varnothing$ because of Condition (4) in Definition~\ref{def:superad}. That is, for any $i\in \I_{\ns}$, $\alpha_i$ is either an even simple root or a non-isotropic odd simple root.

In order to guarantee the property \eqref{eq:cap}, the parameters $\kappa_i$ are required to satisfy the following conditions (cf. \cite{Ko14}):
\begin{equation}
\label{eq:kappai}
\begin{aligned}
    \kappa_i=0 \text{ unless } i\in \I_{\ns} \text{ and } \alpha_i(h_k)\in 2\Z \text{ for all } k\in \I_{\ns}\backslash\set{i}.
\end{aligned}
\end{equation}
This condition will always be imposed throughout the paper. 

The following is a straightforward super analogue of  \cite[Lemma 5.15]{Ko14}; cf. \cite[Lemma 4.3]{Sh22}.
\begin{lemma} 
\label{lemma:eibk}
Suppose that \eqref{eq:kappai} holds. For any $i\in I_\bu,\ k\in I$ we have
\begin{equation}
\label{eq:EjBk}
E_iB_k-(-1)^{p(i)p(k)}B_kE_i=\delta_{i,k}\frac{K_i-K_i^{-1}}{q_i-q_i^{-1}}.
\end{equation}
\end{lemma}

The next lemma gives a condition on the parameters $\va_i$.

%\begin{lemma}
    %Let $i\in \I_\circ$ such that $\tau i\neq i$ and $(\alpha_i,\Theta(\alpha_i))=0$. Then $(\alpha_i,\alpha_{\tau i})=0$ and $\Theta(\alpha_i)=-\alpha_{\tau i}$.
%\end{lemma}

\begin{lemma}
\label{lemma:ineqtaui}
    Suppose $i\neq \tau i\in \I_\circ$ such that $(\alpha_i,\alpha_{\tau i})=0$ and $\Theta(\alpha_i)=-\alpha_{\tau i}$. If $\va_i\neq \va_{\tau i}$, then $(K_iK_{\tau i})^{-1}\in \Uio$.
\end{lemma}
\begin{proof}
     By the assumption we have $\kappa_i=\kappa_{\tau i}=0$ and
    \begin{align*}
        [B_i,B_{\tau i}]=&[F_i+\va_iE_{\tau i}K_i^{-1},F_{\tau i}+\va_{\tau i}E_iK_{\tau i}^{-1}]\\
        =\,&\va_{\tau i}[F_i,E_i]K_{\tau i}^{-1}+\va_i[E_{\tau i},F_{\tau i}]K_i^{-1}\\
        =\,&-(-1)^{p(i)}\va_{\tau i}\frac{K_i-K_i^{-1}}{q_i-q_i^{-1}}K_{\tau i}^{-1}+\va_i\frac{K_{\tau i}-K_{\tau i}^{-1}}{q_{\tau i}-q_{\tau i}^{-1}}K_{ i}^{-1}\\
        \overset{\eqref{eq:qiqti}}{=} &(-1)^{p(i)}\left(-\va_{\tau i}\frac{K_i-K_i^{-1}}{q_i-q_i^{-1}}K_{\tau i}^{-1}+\va_i\frac{K_{\tau i}-K_{\tau i}^{-1}}{q_{ i}-q_{i}^{-1}}K_{ i}^{-1}\right).
    \end{align*}
    Thus if $\va_i\neq \va_{\tau i}$ we must have $(K_iK_{\tau i})^{-1}\in \Uio$.
\end{proof}

Therefore, we impose the following condition on $\va_i$ in order to guarantee \eqref{eq:cap}:
\begin{equation}
\label{eq:vai}
\va_i=\va_{\tau i} \qquad \text{ if $(\alpha_i,\alpha_{\tau i})=0$ and $\Theta(\alpha_i)=-\alpha_{\tau i}$}.    
\end{equation}

The proof of Proposition~\ref{prop:pi00p} below further requires the following (super) condition on the parameters:
\begin{equation}
    \label{eq:vai=vataui}
    \va_i=\va_{\tau i} \qquad \text{ if $p(i)=1$ and $(\alpha_i,\alpha_{\tau i})\in 2\Z$}.
\end{equation}
From now on, conditions \eqref{eq:kappai}, \eqref{eq:vai}, and \eqref{eq:vai=vataui} on parameters will always be imposed without explicit mention.

\subsection{Decompositions and projections}

As usual, the triangular decomposition of $\U$ induces an isomorphism of vector spaces
\begin{equation}
    \U^+\otimes \U^0\otimes S(\U^-)\cong \U.
\end{equation}
This leads to a direct sum decomposition 
\begin{equation}
\label{eq:Pprojection}
    \U=\bigoplus_{h \in Y}\U^+K_h S(\U^-)\oplus\U^+K_h\new S(\U^-).
\end{equation}
For any $h\in Y$, let 
\[
P_h:\U \longrightarrow \U^+K_hS(\U^-)\oplus\U^+K_h \new S(\U^-)
\]
denote the projection with respect to \eqref{eq:Pprojection}. We also use the symbol $P_\lambda$ for $\lambda\in X$ to denote the projection $P_\lambda:\U \to \U^+K_{\lambda} S(\U^-)\oplus\U^+K_\lambda \new S(\U^-)$ as above.
We have that 
\begin{equation}
\label{eq:deltaP}
    \Delta\circ P_h(x)=(id\otimes P_h)\Delta(x),\ \forall h\in Y,x\in \U.
\end{equation}
The identity \eqref{eq:deltaP} implies the following super analogue of \cite[Lemma 5.9]{Ko14} with the same proof:
\[
\Ui=\bigoplus_{h\in Y}P_{h}(\Ui).
\]

On the other hand, with respect to the decomposition 
\begin{equation}
\label{eq:piprojection}
    \U=\bigoplus_{\alpha,\beta\in X^+}\U_\alpha^+\U^0\U_{-\beta}^-,
\end{equation}
We have the following projections 
\begin{align}
    \label{eq:pi}
    \pi_{\alpha,\beta}:\U\longrightarrow \U_\alpha^+\U^0\U_{-\beta}^-.
\end{align}

\subsection{Quantum Serre relations of $\U$}
For any quantum group, the Serre relation can be presented as $S_{ij}(F_i,F_j) =S_{ij}(E_i,E_j)=0$ in terms of a non-commutative Serre polynomial
\[
S_{ij}(x,y)=\sum_{k=0}^{1-a_{ij}}(-1)^k\qbinom{1-a_{ij}}{k}_{q_i}x^{1-a_{ij}-k}yx^k.
\]

We recall the $q$-commutator on homogeneous elements in $\U$ by
\begin{equation}
\label{eq:adq}
    \ad_qu(v)=[u,v]_q=uv-(-1)^{p(u)p(v)}q^{(\alpha,\beta)}vu,\quad \text{ for all } u\in \U_\alpha,v\in \U_\beta.
\end{equation}

According to \cite{Ya94,Ya99}, various higher order quantum Serre relations show up for quantum supergroups associated with arbitrary Dynkin diagrams. 
For future reference, we adopt the notations from \cite[Proposition 2.7]{CHW16} and provide a list of all local Dynkin diagrams relevant to our super admissible pairs and the corresponding Serre relations. Each such Serre relation can be expressed in terms of a non-commutative Serre polynomial as in Table~\ref{TableSerrePolyn}. 
(Note that Dynkin subdiagrams of type (F3), (F4), (G1), and (D$\alpha$) in \cite[Prop.~ 2.7]{CHW16} are excluded since they do not appear as subdiagrams of super Satake diagrams.)

\begin{table}[htbp]
\renewcommand\arraystretch{2}
 \caption{Non-commutative Serre polynomials $\quad (x_i\in \U_{\alpha_i})$}{\label{TableSerrePolyn}}\resizebox{\linewidth}{!}{
\begin{tabular}{|c|c|c|}
\hline
Type & Local Dynkin diagrams & Non-commutative Serre polynomials  \\
\hline
ISO1 & $\xy 
(0,-6)*{\ };
(0,0)*{\otimes}; (0,-3)*{\scriptstyle i};
\endxy$& $x_i^2$\\
\hline
ISO2 & $\xy 
(0,-6)*{\ };
(-5,0)*{\otimes};(5,0)*{\otimes};(-5,-3)*{\scriptstyle i};(5,-3)*{\scriptstyle j};
\endxy$& $x_ix_j+x_jx_i$\\
\hline
N-ISO &$i\in \Ieven\cup \I_{\niso},\ j\neq i $& 
$\sum_{r+s=1+|a_{ij}|}(-1)^r  (-1)^{\binom{r}{2}p(i)+rp(i)p(j)}\qbinom{1+|a_{ij}|}{r}_i x_i^r x_j x_i^s$\\
\hline
   AB & $\xy
(-10,0)*{\odot};(0,0)*{\otimes}**\dir{-};(0,0)*{\otimes};(10,0)*{\odot}**\dir{-};
(-10,-3.5)*{\scriptstyle i};(0,-3.5)*{\scriptstyle j};(10,-3.5)*{\scriptstyle k};
\endxy$ &$[2]_{q_j}x_jx_kx_ix_j-[(-1)^{p(i)}x_jx_kx_jx_i+(-1)^{p(i)+p(i)p(k)}x_ix_jx_kx_j$\\
& $\xy (0,-6)*{\ };
(-10,0)*{\yy};(0,0)*{\otimes}**\dir{=};(0,0)*{\otimes};(10,0)*{\odot}**\dir{-};(-5,0)*{<};
(-10,-3.5)*{\scriptstyle i};(0,-3.5)*{\scriptstyle j};(10,-3.5)*{\scriptstyle k};
\endxy$& $+(-1)^{p(i)p(k)+p(k)}x_jx_ix_jx_k+(-1)^{p(k)}x_kx_jx_ix_j]$\\
\hline
CD1 & $\xy (0,-6)*{\ };
(-10,0)*{\fullmoon};(0,0)*{\otimes}**\dir{=};(0,0)*{\otimes};(10,0)*{\otimes}**\dir{-};(-5,0)*{>};
(-10,-3.5)*{\scriptstyle i};(0,-3.5)*{\scriptstyle j};(10,-3.5)*{\scriptstyle k};
\endxy$&$\ad_q x_j\circ \ad_q (\ad_q x_j(x_k))\circ \ad_q x_i\circ \ad_q x_j (x_k)$ \\
\hline
CD2 &$\xy (0,-6)*{\ };
(-10,0)*{\odot};(0,0)*{\fullmoon}**\dir{-};(0,0)*{\fullmoon};(10,0)*{\otimes}**\dir{-};(15,0)*{<};(10,0)*{\fullmoon};(20,0)*{\fullmoon}**\dir{=};
(-10,-3.5)*{\scriptstyle i};(0,-3.5)*{\scriptstyle j};(10,-3.5)*{\scriptstyle k};(20,-3.5)*{\scriptstyle l};
\endxy$&$\ad_q x_k\circ \ad_q x_j \circ \ad_q x_k\circ \ad_q x_l\circ \ad_q x_k\circ \ad_q x_j (x_i)$ \\ \hline
D & $\xy(0,-6)*{\ };(0,6)*{\ };
(0,0)*{\odot};(8,4)*{\otimes}**\dir{-};(0,0)*{\odot};(8,-4)*{\otimes}**\dir{-};(8,4)*{\otimes};(8,-4)*{\otimes}**\dir{=};
(-4,0)*{\scriptstyle i};(11,4)*{\scriptstyle j};(11,-4)*{\scriptstyle k};
\endxy$ & $\ad_q x_k\circ \ad_q x_j (x_i)-\ad_q x_j \circ \ad_q x_k(x_i)$ \\ \hline
F1 &$\xy(0,-6)*{\ };
{\ar@3{-}(-15,0)*{\fullmoon};(-5,0)*{\otimes}};(-10,0)*{>};
(-5,0)*{\otimes};(5,0)*{\fullmoon}**\dir{=};(0,0)*{<};(5,0)*{\fullmoon};(15,0)*{\fullmoon}**\dir{-};
(-15,-3.5)*{\scriptstyle 1};(-5,-3.5)*{\scriptstyle 2};(5,-3.5)*{\scriptstyle 3};(15,-3.5)*{\scriptstyle 4};
\endxy$ & \makecell{$\ad_q Y\circ \ad_q Y \circ \ad_q x_2 \circ \ad_q x_3(x_4),$\\$Y:=\ad_q(\ad_q x_1(x_2))\circ \ad_q x_3(x_2)$} \\ \hline
F2& $\xy(0,-6)*{\ };
{\ar@3{-}(-15,0)*{\fullmoon};(-5,0)*{\otimes}};(-10,0)*{>};
(-5,0)*{\otimes};(5,0)*{\fullmoon}**\dir{-};(10,0)*{<};(5,0)*{\fullmoon};(15,0)*{\fullmoon}**\dir{=};
(-15,-3.5)*{\scriptstyle 1};(-5,-3.5)*{\scriptstyle 2};(5,-3.5)*{\scriptstyle 3};(15,-3.5)*{\scriptstyle 4};
\endxy$ & \makecell{$\ad_q(\ad_q x_1(x_2))  \circ \ad_q (\ad_q x_3(x_2)) \circ \ad_q x_3(x_4)$\\$
-\ad_q(\ad_q x_3(x_2))\circ \ad_q (\ad_q x_1(x_2)) \circ \ad_q x_3(x_4)$} \\ \hline
%F3 & $\xy(0,-6)*{\ };
%(-10,0)*{\otimes};(0,0)*{\otimes}**\dir{=};(0,0)*{\otimes};(10,0)*{\fullmoon}**\dir{=};(5,0)*{<};
%(-10,-3.5)*{\scriptstyle 1};(0,-3.5)*{\scriptstyle 3};(10,-3.5)*{\scriptstyle 4};
%\endxy$& $\ad_q x_3\circ\ad_{q}x_1\circ \ad_{q} x_3 (x_4)$\\ \hline
%F4 & $ \xy (0,11)*{\ };
%{\ar@3{-}(-5,0)*{\otimes};(0,6)*{\otimes}};(-5,0)*{\otimes};(5,0)*{\otimes}**\dir{=};(0,6)*{\otimes};(5,0)*{\otimes}**\dir{-};
%(-8,0)*{\scriptstyle 1};(0,9)*{\scriptstyle 2};(8,0)*{\scriptstyle 3};
%\endxy$ & $[3]\ad_q x_i\circ \ad_q x_j (x_k)+[2]\ad_q x_j \circ \ad_q x_i(x_k)$
%\\\hline
%G1& $ \xy(0,-6)*{\ };
%{\ar@3{-}(0,0)*{\otimes};(10,0)*{\fullmoon}};(5,0)*{<};(-10,0)*{\otimes};(0,0)*{\otimes}**\dir{=};
%(-10,-3.5)*{\scriptstyle 1};(0,-3.5)*{\scriptstyle 2};(10,-3.5)*{\scriptstyle 3};
%\endxy$& 
%$\ad_q (\ad_q x_2(x_3))\circ \ad_q (\ad_q x_2(x_3)) \circ \ad_q (\ad_q x_2(x_3)) \circ \ad_q x_2(x_1)$ \\ \hline
G2 &  $ \xy (0,-6)*{\ };
{\ar@3{-}(0,0)*{\otimes};(10,0)*{\fullmoon}};(5,0)*{<};(-10,0)*{\LEFTcircle};(0,0)*{\otimes}**\dir{=};
(-10,-4)*{\scriptstyle 1};(0,-3.5)*{\scriptstyle 2};(10,-3.5)*{\scriptstyle 3};
\endxy$ & \makecell{ $\ad_q x_2\circ \ad_q x_3 \circ \ad_q x_3\circ \ad_q x_2(x_1)$ \\ $-
 	\ad_q x_3\circ \ad_q x_2 \circ \ad_q x_3\circ \ad_q x_2(x_1)$}\\ \hline
G3 &\makecell{\\$\xy 
{\ar@3{-}(-5,0)*{\otimes};(5,0)*{\otimes}};(-5,0)*{\otimes};(0,7)*{\fullmoon}**\dir{=};(0,7)*{\fullmoon};(5,0)*{\otimes}**\dir{-};
(-8,0)*{\scriptstyle 1};(0,10)*{\scriptstyle 2};(8,0)*{\scriptstyle 3};
\endxy$}  &   $\ad_q x_1\circ \ad_q x_2 (x_3)-[2]\ad_q x_2 \circ \ad_q x_1(x_3)$ \\ \hline
%D$\alpha$ &$\xy (0,13)*{\ };
%(-5,0)*{\otimes};(0,7)*{\otimes}**\dir{-};
%(-5,0)*{\otimes};(5,0)*{\otimes}**\dir{-};
%(0,7)*{\otimes};(5,0)*{\otimes}**\dir{-};
%(-8,0)*{\scriptstyle 1};(0,10)*{\scriptstyle 2};(8,0)*{\scriptstyle 3};
%(4,4)*{\scriptstyle \af};(0,-2)*{\scriptstyle -1-\af};
%\endxy$& $[\alpha+1]\ad_q x_1\circ \ad_q x_3 (x_2)+[\alpha]\ad_q x_3 \circ \ad_q x_1(x_2)$ \\ \hline
\end{tabular}}
\end{table}

%Now we define a subspace of $\Ui$ by \[\Ui_J=\bigoplus_{J\in \mathcal J}\C(q)B_J.

\subsection{The projection technique}
Let $p=p(x_{i_1},\ldots,x_{i_\ell})$ be a non-commutative polynomial in variables $x_{i_t}$ for $i_t\in \I$. To shorten notations we write
\begin{align*}
    &p(\underline{F})=p(F_{i_1},\ldots,F_{i_\ell}),\quad p(\underline{E})=p(E_{i_1},\ldots,E_{i_\ell}),\quad p(\underline{B})=p(B_{i_1},\ldots,B_{i_\ell}),\\
    &%p(\underline{E_\tau K^{-1}})=p(E_{\tau i_1}K_{i_1}^{-1},\ldots,E_{\tau i_\ell}K_{i_\ell}^{-1}),\quad 
    p\big(\underline{T_{w_\bu}(E_\tau)K^{-1}}\big) =p(T_{w_\bu}(E_{\tau i_1})K_{i_1}^{-1},\ldots, T_{w_\bu}(E_{\tau i_\ell})K_{i_\ell}^{-1}).
    %p(\underline{KE_\tau })=p(K_{i_1}E_{\tau i_1},\ldots,K_{i_\ell} E_{\tau i_\ell}),\\
    %&p(\underline{F_\tau K^{-1}})=p(F_{\tau i_1}K_{i_1}^{-1},\ldots,F_{\tau i_\ell}K_{i_\ell}^{-1}).
\end{align*}

\begin{lemma}
\label{lemma:p}
    The following are equivalent for a homogeneous polynomial $p$:

    (1) $p(\underline{F})=0$,

    (2) $p(\underline{E})=0$,

    (3) $p\big(\underline{T_{w_\bu}(E_\tau) K^{-1}} \big)=0$.

   % (4) $p(\underline{KE_\tau})=0$,

   % (5) $p(\underline{F_\tau K^{-1}})=0$.
\end{lemma}

\begin{proof}
    The equivalence of (1) and (2) follows from $p(\underline{F})=(-1)^{deg(p)}\om\circ p(\underline E)$. Moreover, upon applying an automorphism $T_{w_\bu}$, (2) is clearly equivalent to 
    \begin{align}
        \label{pTE}
    p\big(\underline{T_{w_\bu}(E_\tau)}\big) :=p(T_{w_\bu}(E_{\tau i_1}),\ldots,T_{w_\bu}(E_{\tau i_\ell}))=0.   
    \end{align}
 Next we observe that 
    $p\big(\underline{T_{w_\bu}(E_\tau) K^{-1}}\big)= q^\star p\big(\underline{T_{w_\bu}(E_\tau)}\big) K^{-1}_{i_1}\ldots K^{-1}_{i_\ell}$, for some explicit integer $\star$; this follows by checking directly that the resulting $q$-powers from commuting all $K_j^{-1}$ to the right of any monomial in $p$ is independent of the given monomial. 
%    \blue{Yes. Suppose $p$ has weight $\lambda$. We look at one 'monomial', say it is \[T_{w_\bu}(E_{\tau i_1})K_{i_1}^{-1}\cdots T_{w_\bu}(E_{\tau i_k})K_{i_k}^{-1}\]    Then $\lambda=\alpha_{i_1}+\cdots+\alpha_{i_k}$    Then the power of $q$ from moving all $K^{-1}$ to the right is 
%    \begin{align*}
%        &\left(-\alpha_{i_1},w_\bu(\lambda-\alpha_{\tau i_1})\right)+ \left(-\alpha_{i_2},w_\bu(\lambda-\alpha_{\tau i_1}-\alpha_{\tau i_2})\right)+\cdots+       \\
%        & \left(-\alpha_{i_{k}},w_\bu(\lambda-\alpha_{\tau i_1}-\cdots-\alpha_{\tau i_{k-1}})\right)
%    \end{align*}
%    Now if we switch, say for example, $i_1$ and $i_2$, we get the same power of $q$ since 
%    \[(\alpha_{i_1},w_\bu\alpha_{\tau i_2})=(\alpha_{i_2},w_\bu\alpha_{\tau i_1})\]
%    }
%
Thus \eqref{pTE} and (3) are equivalent. 
\end{proof}

Define a filtration $\Fi^*$ on $\Ui$ by the following degree function specified on generators:
\begin{equation}
    \label{eq:degree}
    \deg (B_i)=1,\ \forall i\in \I,\quad \deg(x)=0,\  \forall x\in \Ub^+\Uio.
\end{equation}

Our goal in this section is to establish a basis theorem for $\Ui$ (Theorem~\ref{thm:basis}) and the quantum Iwasawa decomposition (Theorem~\ref{thm:Iwa}). To that end, we need to show the following property \eqref{eq:Uirel} holds for any homogeneous non-commutative polynomial $p$ of degree $m$:
\begin{equation}
\label{eq:Uirel}
    p(\underline{F})=0 \text{ with } \deg p=m \, \Longrightarrow\, p(\underline{B})\in \Fi^{m-1}(\Ui).
\end{equation}
We shall develop the projection technique below  following \cite{Let02, Ko14}.

For any $J=(j_1,\ldots,j_r)\in I^r$, we denote $wt(J)=\sum_{i=1}^r\alpha_{j_i}$, $|J|=r$, and introduce the following shorthand notations:
\begin{equation}
E_J=E_{j_1}\cdots E_{j_r},\quad F_J=F_{j_1}\cdots F_{j_r},\quad B_J=B_{j_1}\cdots B_{j_r}.
\end{equation}
Suppose that the homogeneous polynomial $p$ has weight  $\lambda=\lambda_{1}\alpha_{i_1}+\cdots+\lambda_{\ell}\alpha_{i_\ell}$, where $\lambda_t\in \N$. Define $|\lambda|:=\lambda_1+\cdots +\lambda_\ell$. Given two weights $\lambda,\mu$, we say $\lambda\leq \mu$ if $\lambda_t\leq \mu_t$ for all $t$. 

\begin{lemma}
\label{lemma:pcriterion}
    Let $p$ be a non-commutative homogeneous polynomial of weight $\lambda$. If $P_{-\lambda}(p(\underline{B}))=0$, then $p$ satisfies the property \eqref{eq:Uirel}.
\end{lemma}
\begin{proof}
     Set $\Xi=p(\underline{B})$ and $Z=P_{-\lambda}(\Xi)$. It follows from \eqref{eq:deltabi-} and \eqref{eq:EjBk} that
   \begin{equation}
\label{eq:comultiXi}
    \Delta(\Xi)\in\Xi\otimes K_{-\lambda}+\sum_{\{J\mid wt(J)<\lambda\}}\U^{+}_\bu\U^{\imath0}B_J\otimes \U.
\end{equation}
Hence \eqref{eq:deltaP} implies that
   \begin{equation}
\label{eq:comultiZ}
    \Delta(Z)\in \Xi\otimes K_{-\lambda}+\sum_{\{J\mid wt(J)<\lambda\}}\U^{+}_\bu\U^{\imath0}B_J\otimes P_{-\lambda}(\U).
\end{equation}
If $Z=0$, then by applying $1\otimes \epsilon$ to \eqref{eq:comultiZ} we can conclude that $\Xi\in \Fi^{|\lambda|-1}$.
\end{proof}

%We consider the set $\Nrel$ of all degrees for which homogeneous relations in $\U^-$ lead to relations in $\Ui$, that is
%\begin{equation}
%\Nrel=\set{k\in \N\mid \text{ any polynomial $p$ of degree $m\leq k$ satisfies \eqref{eq:Uirel}}}.
%\end{equation}
%Thus to show $\eqref{eq:Uirel}$ holds for any homogeneous non-commutative polynomial $p$ it suffices to show $\Nrel=\N$. The following lemma provides main step in the proof of $\Nrel=\N$ below.

Lemma~\ref{lemma:pcriterion} reduces the verification of \eqref{eq:Uirel} for a polynomial $p$ with weight $\lambda$ to the identity $P_{-\lambda}(p(\underline{B})) = 0$. However, computing $P_{-\lambda}(p(\underline{B}))$ can often be challenging due to the complexity of the Serre relations, as illustrated in Table~\ref{TableSerrePolyn}. Proposition~\ref{prop:P-lambda=0} below offers a simple sufficient condition to guarantee the validity of $P_{-\lambda}(p(\underline{B})) = 0$.

The following technical lemma (see also \cite[Lemma 5.14]{Ko14}) is a key step in the proof of Proposition~\ref{prop:P-lambda=0}.
\begin{lemma}
\label{lemma:notinXi}
    Let $p$ be a non-commutative homogeneous polynomial of weight $\lambda$ such that $p(\underline{F})=0$ but $\pi_{\alpha,\beta}(p(\underline{B}))\neq 0$ for some $\alpha,\beta \in X^+$, then $\lambda-\alpha\notin X^\io$ and $\lambda-\beta \notin X^\io$.
\end{lemma}

\begin{proof}
    Suppose $\lambda=\lambda_{1}\alpha_{i_1}+\cdots+\lambda_{\ell}\alpha_{i_\ell}$ where $\lambda_t\in \N$ for $t=1,\ldots,\ell$. If $\set{i_1,\ldots,i_\ell}\subset \I_\bu$, there is nothing to prove since $p(\underline{B})=p(\underline{F})=0$. If $\set{i_1,\ldots,i_\ell}\cap \I_\circ\neq \varnothing$. Since $p(\underline{F})=0$, by \eqref{eq:Bi} we have $\beta\leq\lambda-\alpha_j$ for some $j\in \set{i_1,\ldots,i_\ell}\cap \I_\circ$. Therefore, $\alpha_j\leq \lambda-\beta$ and $\Theta(\alpha_j)\in -X^+$ imply $\lambda-\beta\notin X^\io$.
    
 On the other hand, applying Lemma~\ref{lemma:p}, we obtain $p(\underline{T_{w_\bu}(E_\tau)K^{-1}}) = 0$. Consequently, we have $\alpha \leq -\Theta(\lambda - \alpha_{j'})$ for some $j' \in {i_1, \ldots, i_\ell} \cap \I_\circ$. This implies $-\Theta(\lambda) - \alpha \notin X^\io$. Considering that $\lambda + \Theta(\lambda) \in X^\io$, we can conclude that $\lambda - \alpha \notin X^\io$.
\end{proof}

The following proposition is an adaption of \cite[Proposition 5.16]{Ko14} to the super case.
\begin{proposition}
\label{prop:P-lambda=0}
    Let $p$ be a non-commutative homogeneous polynomial of weight $\lambda$ such that $\pi_{0,0}\circ  P_{-\lambda}(p(\underline{B}))\in \Uio$, then we have $P_{-\lambda}(p(\underline{B}))=0$.
\end{proposition}

\begin{proof}
      Set $\Xi=p(\underline{B})$ and $Z=P_{-\lambda}(\Xi)$. Assume now that $Z\neq 0$. Choose $\alpha$ maximal such that $\pi_{\alpha,\beta}(Z)\neq 0$ for some $\beta \in X^+$. In this case by \eqref{eq:Hopf} we have
      \begin{equation}
    \label{eq:P-lambda}
          0\neq (id\otimes \pi_{\alpha,0})\Delta(Z)\in S(U^-) K_{-\lambda+\alpha}\otimes  U_\alpha^+K_{-\lambda}\oplus S(U^-) \new K_{-\lambda+\alpha}\otimes  U_\alpha^+K_{-\lambda}.
      \end{equation}
      Now, if $\alpha \neq 0$, the relations \eqref{eq:comultiZ} and \eqref{eq:P-lambda} imply that $K_{-\lambda+\alpha}\in \Ui$, which is in contradiction to Lemma~\ref{lemma:notinXi}.
\end{proof}

\begin{example}
\label{example:sAI}
Consider the following diagram from Example~\ref{example:nonadmissible}(2):
\[\xy 
(0,-6)*{\ };
(0,0)*{\otimes}; (0,-3)*{\scriptstyle i};
\endxy\]
where $\I_\circ=\{i\}$ and $\I_\bu=\varnothing$. Note that this is by definition {\bf not} a super Satake diagram because of Condition (4) in Definition \ref{def:superad}. 
In this case we have $B_i=F_i+\va_iE_iK_i^{-1}$ and the only Serre relation is {\rm (ISO1)} in Table~\ref{TableSerrePolyn}. Hence $p(x_i)=x_i^2$. We see that $ p(B_i)=B_i^2=\va_i\frac{1-K_i^{-2}}{q_i-q_i^{-1}}\notin \Uio$. Thus $\pi_{0,0}\circ P_{-2\alpha_i}(p(\underline{B}))\notin \Uio$ and $P_{-2\alpha_i}(p(\underline{B}))\neq 0$. This is the main reason we exclude such a diagram from being a super Satake diagram.
\end{example}

\begin{example}
    Let us consider the super admissible pair in Example~\ref{example:admissible}(1) associated to the super Satake diagram   
    {\rm 
$$\hspace{.75in}\xy
(-5,0)*{\otimes};(5,0)*{\newmoon}**\dir{-};
(-5,-4)*{\scriptstyle 1};(5,-4)*{\scriptstyle 2};
\endxy$$}
There are two polynomials 
\begin{align*}
    p_1(x_1)=x_1^2,\quad p_2(x_1,x_2)=x_2^2x_1-[2]_2x_2x_1x_2+x_1x_2^2
\end{align*}
expressing the Serre relations {\rm (ISO1)} and {\rm (N-ISO)} separately. By definition \eqref{eq:Bi} we have $B_1=F_1+\va_1T_{2}(E_1)K_1^{-1}$ and $B_2=F_2$. By a direct computation we obtain that
\begin{align*}
p_1(B_1)=\,&B_1^2=q\va_1F_1K_1^{-1}T_2(E_1)+\va_1K_1^{-1}T_{2}(E_1)F_1\\
=\,&q\va_1K_1^{-1}\left[E_2(F_1E_1+E_1F_1)-q(F_1E_1+E_1F_1)E_2\right]
=-q\va_1E_2. \\
p_2(B_1,B_2)=\,&p_2(B_1,F_2)=p_2(F_1,F_2)+\va_1 p_2(T_2(E_1)K_1^{-1},F_2)\\
=\,&0+\va_1 p_2(T_2(E_1)K_1^{-1},T_2(E_2)K_2^{-1})=0.
\end{align*}
Therefore, we see that $P_{-2\alpha_1}(p_1(\underline{B}))=P_{-2\alpha_2-\alpha_1}(p_2(\underline{B}))=0$. Hence by Lemma~\ref{lemma:pcriterion} $p_1$ and $p_2$ satisfy \eqref{eq:Uirel}.
\end{example}

%\subsection{A key technical step}

Denote
\begin{align}
\label{SU}
    \mathcal S(\U)=\{\text{all non-commutative Serre polynomials in Table~\ref{TableSerrePolyn}} \}.
\end{align}
Note that $p$ is homogeneous and $p(\underline{F})=0$, for $p \in \mathcal S(\U)$. The following key technical step toward the quantum Iwasawa decomposition is a super analogue of \cite[Lemma 5.11]{Ko14}. Recall the projection $\pi_{0,0}$ from \eqref{eq:pi}. 
\begin{proposition}
\label{prop:pi00p}
    Let $(\I=\I_\circ \cup \I_\bu,\tau)$ be a super admissible pair. Suppose $p\in \mathcal S(\U)$ has weight $\lambda$, then we have 
    \[
    \pi_{0,0}\circ P_{-\lambda}(p(\underline{B}))\in \U^{\imath 0}.
    \]
\end{proposition}

\begin{proof}
%It follows by definition of $\pi_{0,0}$ that $\pi_{0,0}\circ P_{-\lambda}(p(\underline{B}))\in \U^{0}$, and note that $\U^{\imath 0} =\U^0 \cap \Ui$. It suffices to show that $\pi_{0,0}\circ P_{-\lambda}(p(\underline{B}))\in \Ui$.

We consult Table~\ref{TableSerrePolyn} for the corresponding local Dynkin diagrams and notations. According to Definition~\ref{def:superad}, all odd simple roots belong to $\I_\circ$. Within a given super admissible pair $(\I=\I_\circ \cup \I_\bu,\tau)$, multiple local Dynkin diagrams are encompassed. Each local diagram in Table~\ref{TableSerrePolyn} gives rise to a non-commutative Serre polynomial, and we proceed to prove the proposition through a case-by-case examination. % Since Dynkin diagrams of type (F3), (F4), (G1), and (D$\alpha$) in Table~\ref{TableSerrePolyn} do not appear as local diagrams of a super admissible pair, we do not need to consider those cases. 
Dynkin diagram of type (F2) underlying the super Satake diagram sFII will not be considered by our assumption at the beginning of Section~\ref{sec:QSP}.
%, (F2) and one case of (AB) are not considered either.

We proceed with the proof below case-by-case following the types in Table~\ref{TableSerrePolyn}.

(ISO1) In this case we have $p(x_i)=x_i^2$. If $\tau i=i$, by Condition (4) in Definition~\ref{def:superad} we see that $\alpha_i$ must be connected to some $\alpha_j$ for $j\in \I_\bu$. Thus for weight reason we see that $\pi_{0,0}\circ P_{-\lambda}(p(\underline{B}))=0$.

If $\tau i\neq i$, since $\tau$ preserves parities, we have $p(\tau i)=1$ and hence $E_{\tau i}^2=0$. Thus we have
\[B_i^2=(F_iT_{w_\bu}(E_{\tau i})+T_{w_\bu}(E_{\tau i})F_i)K_i^{-1}\overset{\eqref{eq:Urelation}}{=}0.\]

(ISO2) In this case we see that $p(B_i,B_j)\neq 0$ if and only if $\Theta(\alpha_i)=-\alpha_{j}$ and $a_{ij}=0$. According to Lemma~\ref{lemma:ineqtaui} the condition $\va_i=\va_{j}$ guarantees that $p(B_i,B_j)\in \Uio$.

(N-ISO) For weight reason and Condition (4) in Definition~\ref{def:superad}, $\pi_{0,0}\circ P_{-\lambda}(p(\underline{B}))=0$ unless $j\in \Ieven$. The case when $j\in \Ieven$ follows from \cite[Lemma 5.11]{Ko14}.

(AB) For weight reason $\pi_{0,0}\circ P_{-
    \lambda}(p(\underline{B}))=0$ unless $\tau j= j,\ i,k\in \I_\bu$ and $a_{j\ell}=0$ for any $\ell\in \I_\bu\backslash\set{i,k}$. In this case the super Satake diagram contains a subdiagram of the form {\rm
$\xy
(-10,0)*{\newmoon};(0,0)*{\otimes}**\dir{-};(10,0)*{\newmoon}**\dir{-};(-10,-3)*{\scriptstyle i};
(0,-3)*{\scriptstyle j};(10,-3)*{\scriptstyle k};
\endxy$}, which is excluded by our assumption.

(CD1)--(CD2) For weight reason we must have $\pi_{0,0}\circ P_{-\lambda}(p(\underline{B}))=0$.

(D) For weight reason, $\pi_{0,0}\circ P_{-
    \lambda}(p(\underline{B}))=0$ unless  $ i\in\I_\bu$ and $j=\tau k$. 
   
    Suppose that $i\in \I_\bu$ and $j=\tau k$. Thus, the only possible non-zero term in $\pi_{0,0}\circ P_{-\lambda}(\ad_q B_k\circ \ad_q B_j (F_i))$ is the monomial $K_i^{-1}K_j^{-1}K_k^{-1}$. By \eqref{eq:vai=vataui} we have $\va_j=\va_k$ in this case. Thus if $rK_i^{- 1}K_j^{-1}K_k^{-1}$ is a summand in $\pi_{0,0}\circ P_{-\lambda}(\ad_q B_k\circ \ad_q B_j (F_i))$ with $r\neq 0$, then it must be a summand in $\pi_{0,0}\circ P_{-\lambda}(\ad_q B_j\circ \ad_q B_k (F_i))$ as well. Hence we conclude that $\pi_{0,0}\circ P_{-
    \lambda}(p(\underline{B}))=0$.

(F1) For weight reason, we always have $\pi_{0,0}\circ P_{-
    \lambda}(p(\underline{B}))=0$ in this case since the degree of $x_2$ is odd.

\begin{comment}
(F2) If $4\in \I_\circ$ or $1\in \I_\circ$, then for weight reason, we always have $\pi_{0,0}\circ P_{-
    \lambda}(p(\underline{B}))=0$.

    Assume that $1,4\in \I_\bu$. We see that $3$ cannot belong to $\I_\circ$ because of condition (4) in Definition~\ref{def:superad}. Thus we have $3\in \I_\bu$ as well.
    
    \blue{This case is not proved. The difficulty: 1. For weight reason $\pi_{0,0}\circ P_{-
    \lambda}(p(\underline{B}))$ can be non-zero, thus in order to check it lies in $\Ui$ or not there is no other way but direct computation. Also I am not quite sure what are the simple roots for (F2). Usually these non-distinguished ones are obtained from the distinguished ones by applying odd reflections}
\end{comment}

(G2) We observe that $1,2\in \I_\circ$, hence for weight reason, we always have $\pi_{0,0}\circ P_{-
    \lambda}(p(\underline{B}))=0$ in this case.

(G3) We observe that $2$ must lie in $\I_\bu$ and $\tau 1=1,\ \tau 3=3$. Hence, for weight reasons, we always have $\pi_{0,0}\circ P_{-
    \lambda}(p(\underline{B}))=0$ in this case.
\end{proof}

The main upshot of developing the projection technique is the following.

\begin{proposition}
\label{prop:p}
    Given a super admissible pair $(I=\I_\circ \cup \I_\bu, \tau)$, %if $p\in \mathcal S(\U)$, then
    any homogeneous non-commutative polynomial $p$ satisfies the property \eqref{eq:Uirel}, that is, $p(\underline{F})=0 \text{ with } \deg p=m \Longrightarrow p(\underline{B})\in \Fi^{m-1}(\Ui)$.
\end{proposition}

\begin{proof}
    Since $\U^-$ is isomorphic to the free algebra generated by $F_i, i \in \I$ quotienting the ideal generated by the Serre relations (cf. \cite{Ya94}), it suffices to examine whether the non-commutative polynomials arising in Table~\ref{TableSerrePolyn} satisfy \eqref{eq:Uirel}. Now it follows from the combination of Proposition~\ref{prop:pi00p}, Proposition~\ref{prop:P-lambda=0} and Lemma~\ref{lemma:pcriterion}. 
\end{proof}

\subsection{Quantum Iwasawa decomposition and bases for $\Ui$}
In this subsection we follow \cite{Let02} and \cite{Ko14} to construct a super generalization of the quantum Iwasawa decomposition with respect to $(\U,\Ui)$.  We note that, for $\mathfrak g$ purely even, Proposition~\ref{prop:Ubasis} traces back to \cite[Proposition 6.1]{Ko14}. Theorem~\ref{thm:basis} is constructed in \cite[Proposition 6.3]{Ko14} and Theorem~\ref{thm:Iwa} can be found in \cite[Proposition 6.13]{Ko14}.

Let $\mathcal{J}$ be a fixed subset of $\cup_{s\in \Z_{\geqslant 0}}I^s$ such that $\{F_J\mid J\in \mathcal{J}\}$ forms a basis for $\U^{-}$. We define an increasing filtration $\mathcal{F}^*$ of $\U^{-}$ by letting  
\[
\mathcal F^t(\U^{-})= \text{span} \{F_J\mid J\in \I^s,\ s\leqslant t\},\quad \text{ for } t\in \Z_{\geqslant 0}.
\]
Since the quantum Serre relations for $\U$ are always homogeneous, the set $\{F_J\mid J\in \mathcal{J},\ |J|\leqslant t\}$ serves as a basis for $\mathcal{F}^t(\U^{-})$.

\begin{proposition}
\label{prop:Ubasis}
The set $\{B_J\mid J\in \mathcal J\}$ is a basis of the left $\U^{+}\U^{0}$-module $\U$.
\end{proposition}

\begin{proof}
First for any $J\in \mathcal J$ such that $|J|=t$ we have $F_J-B_J\in \U^{+}\U^{0}\mathcal F^{t-1}(\U^{-})$. Thus by induction on $t$ we conclude that each $F_J$ is contained in the left $\U^{+}\U^{0}$-module generated by $\{B_J\mid J\in \mathcal J\}$.

It remains to show that $\{B_J\mid J\in \mathcal J\}$ is linearly independent. Assume that there exists a non-empty finite subset $\mathcal J'\subset \mathcal J$ such that $\sum_{J\in \mathcal J'}a_JB_J=0$. Let $t_0=max\{|J|\mid J\in \mathcal J'\}$. In view of the definition of $B_j$, we have
\[\sum_{J\in \mathcal J',|J|=t_0}a_JF_J=0.\]
The linear independence of $\{F_J\mid J\in \mathcal J\}$ implies that $a_J=0$ for all $J\in \mathcal J',|J|=t_0$. Then through induction we conclude the desired result. 
\end{proof}

\begin{theorem}
\label{thm:basis}
Retain the assumption \eqref{exclude}. 
Then the set $\{B_J \mid J \in \mathcal{J}\}$ forms a basis for the left $\U^{+}_{\bu}\Uio$-module $\Ui$.
\end{theorem}

\begin{proof}
First of all, since $\{B_J\mid J\in \mathcal J\}$ is linearly independent over $\U^{+}\U^{0}$, it is also independent over $\U^{+}_{\bu}\Uio$.

Secondly, let $L\in \I^t$. One can apply the Serre relations in Table~\ref{TableSerrePolyn} repeatedly to write
\[F_L=\sum_{J\in\mathcal J, |J|=t}a_JF_J\]
for some $a_J\in \C(q)$. Hence, using Proposition~\ref{prop:p} and Lemma~\ref{lemma:eibk} one sees that
\[B_L-\sum_{J\in\mathcal J,|J|=t}a_JB_J\in\sum_{s<t}\sum_{J\in I^s}\U^{+}_{\bu}\Uio B_J.\]
Thus through induction on $t$ one obtains $B_L\in \sum_{J\in \mathcal J}\U_\bu^+\Uio B_J$. Therefore, $\{B_J\mid J\in \mathcal J\}$ forms a basis of $\U^{+}_{\bu}\Uio$-module $\Ui$.
\end{proof}

Define a subspace of $\Ui$ by
\begin{equation}
    \Ui_{\mathcal J}:=\bigoplus_{J\in \mathcal J}\C(q)B_J.
\end{equation}
Then Theorem~\ref{thm:basis} can be reformulated as that the multiplication map
\[
\U_\bu^+\otimes \Uio \otimes \Ui_{\mathcal J}\longrightarrow \Ui
\]
is an isomorphism of vector spaces.

Fix a subset $\I_\tau\subset \I_\circ$ consists of exactly one element of each $\tau$-orbit within $\I_\circ$. Let $\Uio_\tau$ denote the subalgebra generated by $\set{K_i^{\pm 1}\mid i\in \I_\tau}$. Then we have the following algebra isomorphism
\begin{equation}
\label{eq:UtauUio}
    \Uio_\tau\otimes \Uio \cong \U^0. 
\end{equation}

Now we are ready to establish a super generalization of the quantum Iwasawa decomposition.

\begin{theorem} [Quantum Iwasawa decomposition]
\label{thm:Iwa}
    Retain the assumption \eqref{exclude}. Then the  multiplication map gives an isomorphism of vector spaces
    \[V_\bu^+\otimes \Uio_\tau \otimes \Ui \cong \U.\]
\end{theorem}

\begin{proof}
    Combining the isomorphism \eqref{eq:VUb} with \eqref{eq:UtauUio} gives us an isomorphism of vector spaces
\begin{align}  \label{eq:iso2}
\U^+\U^0\cong V_\bu^+\otimes \Uio_\tau\otimes \Ub^+\otimes \Uio.
\end{align}
By applying Proposition~\ref{prop:Ubasis}, the isomorphism \eqref{eq:iso2} above, and then Theorem~\ref{thm:basis}, we obtain that 
\[
\U \cong \U^+\U^0 \otimes \Ui_{\mathcal J}
\cong V_\bu^+\otimes \Uio_\tau\otimes \Ub^+ \Uio \otimes \Uio
\cong V_\bu^+\otimes \Uio_\tau\otimes \Ui.
\]
The theorem is proved.
\end{proof}

\section{Braid group actions on $\Ui$}
\label{sec:braid}

In this section, we show that the braid group symmetries $T_i$ on $\U$ from \eqref{eq:Ti'}--\eqref{eq:Ti}, for $i\in \I_\bu$,  preserve the subalgebra $\Ui$. We also show that a crucial anti-involution $\wp$ on $\U$ preserves the subalgebra $\Ui$.

\begin{lemma}{\rm \cite[Lemma 4.1]{BW18b}}
\label{lemma:TiTw}
    We have \[
        T_iT_{w_\bu}=T_{w_\bu}T_{\tau i},\quad \forall i\in \I_\bu.
    \]
    In particular $T_{w_\bu}^2$ commutes with $T_i$ for any $i\in \I_\bu$.
\end{lemma}

\begin{theorem}
\label{thm:braidonUi}
    For any $i\in \I_\bu$ and $e\in \pm 1$, the braid group operators $T_{i,e}'$ and $T_{i,e}''$ restrict to algebra isomorphisms of $\Ui$. More explicitly, we have, for $j\neq i$,
    \begin{align*}
    T_{i,e}'(B_j)=\sum_{r+s=|a_{ij}|}(-1)^rq_i^{-er}B_i^{(s)}B_jB_i^{(r)},
    \end{align*}
    \begin{align*}
        T_{i,-e}''(B_j)=\sum_{r+s=|a_{ij}|}(-1)^rq_i^{-er}B_i^{(r)}B_jB_i^{(s)}.
    \end{align*}
\end{theorem}
\begin{proof}
We only prove for $T_i$ as the other cases are similar. Since $T_i$ preserves $\Ub$ and $\Uio$, it suffices to show that $T_i(B_j)\in \Ui$ for $j\in \I_\circ$. Recall $B_j$ from \eqref{eq:Bi}. If $\kappa_j\neq 0$, by \eqref{eq:kappai} we have $\alpha_j(h_i)=0$. Hence $T_i(B_j)=B_j$. The case where $i\in \I_\bu,\  j\in \I_\circ$ such that $\alpha_j(h_i)\neq 0$ follows from exactly the same computations as in \cite[Theorem~ 4.2]{BW18b}.
\end{proof}

There exists an anti-involution $\wp$ on $\U$ (cf. \cite{BKK00}) such that  
\begin{equation}
\label{eq:wp}
    \wp(E_i)= q_i^{-1}F_iK_i,\quad \wp(F_i)=q_iK_i^{-1}E_i,\quad \wp(K_i)=K_i,\quad \wp(\new)=\new.
\end{equation}

\begin{lemma}
\label{lemma:wp1}
    For $i\in \Ieven$ and $j\in \I$, we have
    \begin{align*}
        \wp  (T''_{i,e} (E_{j})) &= (-q_i)^{e\la i, j'\ra} T'_{i,-e} (\wp (E_j)),
\\
\wp  (T'_{i,e} (E_{j})) &= (-q_i)^{-e\la i, j'\ra} T''_{i,-e} (\wp (E_j)).
    \end{align*}
\end{lemma}
\begin{proof}
    We only prove for the first equality. When $i\neq j$ and $a_{ij}=0$, there is nothing to prove. Since $i\in \Ieven$, we have $p(i)=0$. The case when $j\in \Ieven$ follows from a rerun of the proof of \cite[Lemma 4.4]{BW18b}.

    Thus it remains to prove the first equality for $j\in \Iodd$ and $i$ is connected with $j$. Indeed, we have
\begin{equation}
\label{eq:wpT}
     \begin{aligned}
         \wp(T_{i,e}''(E_j))
         &= \wp \Big(\sum_{r+s=|a_{ij}|}(-1)^rq_i^{-er}E_i^{(s)}E_jE_i^{(r)}\Big)\\
         &= \sum_{r+s=|a_{ij}|}(-1)^rq_i^{-er}(q_i^{-1}F_iK_i)^{(r)}(q_j^{-1}F_jK_j)(q_i^{-1}F_iK_i)^{(s)}.
     \end{aligned}
     \end{equation}
    Since $p(i) = 0$, we observe that $(\alpha_i, \alpha_j)$ and $(\alpha_i, \alpha_i)$ have opposite signs. Therefore we have $|a_{ij}|=-2\frac{(\alpha_i,\alpha_j)}{(\alpha_i,\alpha_i)}$. Combine this fact with the identity $(F_iK_i)^s=q_i^{ (s-s^2)}F_i^sK_i^s$
     we see that \eqref{eq:wpT} is equivalent to
     \begin{align*}
         \wp(T_{i,e}''(E_j))=&\sum_{r+s=|a_{ij}|}(-1)^rq_i^{-er}q_j^{-1}F_i^{(r)}F_jF_i^{(s)}T_i(K_j)\\
         =&\sum_{r+s=|a_{ij}|}(-1)^sq_j^{-1}(-q_i)^{-e|a_{ij}|}q_i^{es}F_i^{(r)}F_jF_i^{(s)}T_i(K_j)\\
         =\,&q_j^{-1}(-q_i)^{e\la i,j'\ra}\sum_{r+s=|a_{ij}|}q_i^{es}F_i^{(r)}F_jF_i^{(s)}T_i(K_j).
     \end{align*}
On the other hand, we have\[T_{i,-e}'(\wp(E_j))=T_{i,-e}'(q_j^{-1}F_jK_j)=q_j^{-1}\sum_{r+s=|a_{ij}|}q_i^{es}F_i^{(r)}F_jF_i^{(s)}T_i(K_j).\]     
This completes the proof.
\end{proof}

\begin{corollary}
\label{cor:wp}
For $j\in \I_\circ$ and $e\in \pm 1$, we have
\begin{align*}&\wp(T_{w_\bu,e}''(E_j))=(-1)^{\alpha_j(2\rho_\bu^\vee)}q^{e(\alpha_j,2\rho_\bu)}T_{w_\bu,-e}'(\wp(E_j)),\\
&\wp(T_{w_\bu,e}'(E_j))=(-1)^{\alpha_j(2\rho_\bu^\vee)}q^{-e(\alpha_j,2\rho_\bu)}T_{w_\bu,-e}''(\wp(E_j)).
\end{align*}
\end{corollary}
\begin{proof}
    It follows from a rerun of the proof of \cite[Corollary 4.5]{BW18b} with the help of Lemma~\ref{lemma:wp1}.
\end{proof}

\begin{proposition}
\label{prop:wp}
    Assume that 
\begin{align}\label{eq:wpcondition1}
&\va_i=q_i^{-1} \quad \text{ if }\kappa_i\neq
0, \\
\label{eq:wpcondition2}
    &\va_i=(-1)^{\alpha_i(2\rho_\bu^\vee)}q^{-(\alpha_i,w_\bu\alpha_{\tau i}+2\rho_\bu)}q_{\tau i}q_i^{-1}\overline{\va_{\tau i}},\quad \forall i\in \I_\circ,\\
    \label{eq:wpcondition3}
&\va_i^{-1}=\overline{\va_i},\quad \forall i\in \I_\circ.
\end{align}
    then the anti-involution $\wp$ on $\U$ restricts to an anti-involution on $\Ui$ such that
    \begin{align}
    \label{eq:wpblack}
        \wp(E_{i}) &= q^{-1}_i F_{i} K_{i}, \quad \wp(F_{i}) = q^{-1}_i E_{i} K_i^{-1}, \quad \wp(K_h) = K_h, \quad \wp(\new)=\new,  \\ 
\wp (B_j ) &=  q_j^{-1} \va_{\tau j}^{-1} T_{w_\bu}^{-1} (B_{\tau j} ) \cdot K_{w_\bu \tau j}K_j^{-1}.
\label{eq:wpwhite}
    \end{align}
for any $h\in Y^\io, i \in \I_\bu$, and any $j \in \I_\circ.$
\end{proposition}

%\blue{\eqref{eq:wpcondition2} is different from what we obtained in \cite{Sh22}. It is mainly because the difference in how we define $\wp$ in \eqref{eq:wp}. In fact, for $i$ odd, we can define $\wp(E_i)=t_iF_iK_i,\ \wp(F_i)=t_i^{-1}E_iK_i^{-1}$ for any $t_i$ and it will still be an algebra isomorphism. If we further assume $t_i=t_{\tau i}$, then the term $q_{\tau i}q_i^{-1}$ in \eqref{eq:wpcondition2} will become $1$. I think that version of this condition is what we essentially need, from the duality point of view.}

\begin{proof}
Equation \eqref{eq:wpblack} follows from the definition of $\wp$ in \eqref{eq:wp}. To prove \eqref{eq:wpwhite}, we separate into two cases. The case when $j\in  \Ieven$ follows as in \cite[Proposition 4.6]{BW18b}. Thus it suffices to prove \eqref{eq:wpwhite} for $j\in \Iodd$.

 Applying \eqref{eq:wp}, \eqref{eq:wpcondition2}, \eqref{eq:wpcondition3},  and Corollary~\ref{cor:wp} we have
\begin{align*}
\wp(B_j)=\,&q_j^{-1}E_jK_j^{-1}+\va_j K_j^{-1}\wp(T_{w_\bu}(E_{\tau j}))+\kappa_j K_j^{-1}\\
=\,&q_j^{-1}\left(E_j+q_j\va_j K_j^{-1}\wp(T_{w_\bu}(E_{
\tau j}))K_j+q_j\kappa_j\right)K_j^{-1}\\
=\,&q_j^{-1}\left(E_j+ \va_{\tau j}^{-1}T_{w_\bu}^{-1}(F_{\tau j})K_{w_\bu \tau j}+q_j\kappa_j\right)K_j^{-1}.
\end{align*}
On the other hand, we have
\begin{align*}
    T_{w_\bu}^{-1}(B_{\tau j})=\,&T_{w_\bu}^{-1}(F_{\tau j}+\va_{\tau j}T_{w_\bu}(E_j)K_{\tau j}^{-1}+\kappa_{\tau j}K_{\tau j}^{-1})\\
    =\,&(\va_{\tau j}E_j+T_{w_\bu}^{-1}(F_{\tau j})K_{w_\bu \tau j}+\kappa_{\tau_j})K_{w_\bu \tau j}^{-1}.
\end{align*}
If $\kappa_j\neq 0$, then we have $\va_j=q_j^{-1}$ by \eqref{eq:wpcondition1}. Thus, we see that \eqref{eq:wpwhite} holds for $j\in \Iodd$.
\end{proof}

\section{Quasi $K$-matrix}
\label{sec:quasiK}
In this section, we retain the mild assumption \eqref{exclude} on super Satake diagrams. We shall follow \cite{WZ22} to formulate the quasi $K$-matrix $\up=\sum_{\mu\in X^+_{\overline{0}}}\up_\mu$ with $\up_0=1$ and $\up_\mu\in \U^+$ such that 
\begin{align}
\label{eq:biup}
\begin{split}
B_i \up &=\up \sigma \tau (B_{\tau i}),\qquad \forall i\in \I_\circ,
\\
x \up &= \up x,\qquad \forall x\in \U^{\imath 0} \U_\bu.
\end{split}
\end{align}
Here $X^+_{\overline 0} =\big\{\beta \in X \mid \beta \in \oplus_{i\in \I} \N \alpha_i, p(\beta)=0 \big\}.$
We shall impose in this section the restriction $\kappa_i=0$ for all $i\in \I_\circ$ for the sake of simplicity. To prove the existence and uniqueness of $\up$, we shall generalize the earlier approach developed in \cite{BW18a, BK19} (also cf. \cite{Sh22} for additional super signs). 

\subsection{Super skew derivatives}

Recall an anti-involution $\sigma$ on $\U$ defined in \eqref{eq:sigma}. Let $\U^{\geq}$ (resp. $\U^{\leq}$) denote the Hopf subalgebra of $\U$ generated by $\U^0$ and $\U^+$ (resp. $\U^-$). According to \cite[\S 2.4]{Ya94}, We abuse the notation $\la \cdot,\cdot \ra$ to denote a non-degenerated bilinear pairing 
$\la \cdot,\cdot \ra$ on $\U^{\leq}\times \U^{\geq}$ such that for all $x,x'\in \U^{\geq}$, $y,y'\in \U^{\leq}$, $\mu,\nu \in X$ and $a,b\in \set{0,1}$, we have
\begin{equation}
\label{eq:bilinearform}
\begin{aligned}
    &\la y,xx'\ra=\la \Delta(y),x'\otimes x\ra,\quad &\la yy',x \ra=\la y\otimes y',\Delta(x)\ra, \\
    &\la K_\mu\new^a,K_\nu \new^b\ra=(-1)^{ab}q^{-(\mu,\nu)},\quad &\la F_j,E_k\ra=\delta_{j,k}\frac{1}{q_j-q_j^{-1}}, \\
    & \la K_\mu \new^a,E_j\ra =0,\quad &\la F_j,K_\mu \new^a\ra=0
\end{aligned}
\end{equation}

Moreover, the super skew derivations (cf. \cite[\S 1.5]{CHW13}) $\lskew$ and on $\U^{+}$ satisfy $\rskew(E_j)=\delta_{i,j},\ \lskew(E_j)=\delta_{i,j},$ and
\begin{equation}
    \label{eq:skew}
    \begin{aligned}
    &\lskew(xy)=(-1)^{p(y)p(i)}\lskew(x)y+q^{(\alpha_i,\mu)}x\lskew(y),\\
    &\rskew(xy)=(-1)^{p(y)p(i)}q^{(\alpha_i,\upsilon)}\rskew(x)y+x\rskew(y)
    \end{aligned}
\end{equation}
for all $x\in \U^{+}_\mu,\ y\in \U^{+}_\upsilon$.

The following lemmas are straightforward super generalizations of well-known results for quantum groups.
\begin{lemma}
For all $x\in \U^{+}$, $y\in \U^{-}$ and $i\in I$ one has
\begin{equation}
\label{eq:Fiy}
    \la F_iy,x\ra= (-1)^{p(x)p(i)}\la F_i,E_i\ra\la y,\lskew(x)\ra,\quad \la yF_i,x\ra=\la F_i,E_i\ra \la y,\rskew(x)\ra.
\end{equation}
\end{lemma}
\begin{proof}
It follows from a rerun of the proof of \cite[Lemma 6.2]{Sh22}.
\end{proof}

\begin{lemma}
\label{lem:commuF}
For all $x\in \U^{+}$, we have
\begin{equation*}
    [x,F_j]=xF_j-(-1)^{p(x)p(j)}F_jx=\frac{1}{q_j-q_j^{-1}}\big(r_j(x)K_j-K_j^{-1}{}_jr(x)\big).
\end{equation*}
\end{lemma}

\begin{proof}
    It follows from a rerun of the proof of \cite[Lemma 6.3]{Sh22}.
\end{proof}

\begin{lemma}
\label{lemma:rijr}
For any $i,j\in I$, we have
\[r_i\circ {}_jr(x)=(-1)^{p(i)p(j)}{}_jr \circ r_i(x),\quad \forall x\in \U^+.\]
\end{lemma}

\begin{proof}
If $u=E_k$ or $1$, then we certainly have $r_i\circ {}_jr(u)=(-1)^{p(i)p(j)}{}_jr \circ r_i(u)$. Thus it is enough to show that $r_i\circ {}_jr(xy)={}_jr \circ r_i(xy)$ for any $x\in \U^{+}_\mu,\ y\in \U^{+}_\upsilon$.

We have
\begin{align*}
    r_i\circ {}_jr(xy)
    &= r_i((-1)^{p(y)p(j)}{}_jr(x)y+q^{(\alpha_j,\mu)}x{}_jr(y))\\
    &= (-1)^{p(y)p(j)}[(-1)^{p(y)p(i)}q^{(\alpha_i,\upsilon)}r_i\circ{}_jr(x)y+{}_jr(x)r_i(y)]\\
    &\qquad+q^{(\alpha_j,\mu)}[(-1)^{p({}_jr(y))p(i)}q^{(\alpha_i,\upsilon-\alpha_j)}r_i(x){}_jr(y)+xr_i\circ {}_jr(y)], 
\end{align*}
\begin{align*}
    {}_jr\circ r_i(xy)
    &={}_jr((-1)^{p(y)p(i)}q^{(\alpha_i,\upsilon)}\rskew(x)y+x\rskew(y))\\\
    &= (-1)^{p(y)p(i)}q^{(\alpha_i,\upsilon)}[(-1)^{p(y)p(j)}{}_jr\circ r_i(x)y+q^{(\alpha_j,\mu-\alpha_i)}r_i(x){}_jr(y)]\\
    &\qquad+[(-1)^{p({}_ir(y))p(j)}{}_jr(x)r_i(y)+q^{(\alpha_j,\mu)}x{}_jr\circ r_i(y)].
\end{align*}
Now since $p({}_kr(y))=p(y)\pm p(k)$ for any $k\in I$, we have
$%\[
r_i\circ {}_jr=(-1)^{p(i)p(j)}{}_jr \circ r_i.
$%\]
\end{proof}

\subsection{A recursive construction}

By convention, it is understood that $\va_i=0$ for all $i\in \I_\bu$. The first step toward the construction of $\up$ is to translate \eqref{eq:biup} into a recursive formula of $\up_\mu$ as in \cite[Proposition 6.1]{BK19} (also cf. \cite{BW18a}). Let $\rho_\bu$ denote the half sum of positive roots of the Levi subalgebra associated with $I_\bu \subset I$.
\begin{proposition}
\label{prop:biup}
Let $\up=\sum_{\mu\in X^+_{\overline 0}}\up_\mu$ with $\up_\mu\in \U^{+}_\mu$ be an element  in the completion of $\U$, then the following are equivalent:
\begin{enumerate}
\item 
 The element $\up$ satisfies \eqref{eq:biup}.
\item
 For all $i\in \I$, we have
\begin{equation}
\label{eq:upB2}
     B_i \up=\up( F_i+(-1)^{\delta_{i,\tau i}p(i)+\alpha_i(2\rho_\bu^\vee)}{q}^{(\alpha_i,2\rho_\bu+w_\bu\alpha_{\tau i})}\va_{\tau i}  \overline{T_{w_\bu}( E_{\tau i})} K_i).
\end{equation}
\item 
The element $\up$ satisfy the following relations:
\begin{equation}\label{eq:derivativeup}
    \begin{aligned}
    &\rskew(\up_\mu)=-(q_i-q_i^{-1})\up_{\mu-\alpha_i-w_\bu (\alpha_{\tau i})}(-1)^{\delta_{i,\tau i}p(i)+\alpha_i(2\rho_\bu^\vee)}{q}^{(\alpha_i,2\rho_\bu+w_\bu\alpha_{\tau i})}\va_{\tau i}  \overline{T_{w_\bu}( E_{\tau i})},\\
    &\lskew(\up_\mu)=-(q_i-q_i^{-1})q^{(\alpha_i,w_\bu\alpha_{\tau i})}\va_i T_{w_\bu}( E_{\tau i})\up_{\mu-\alpha_i-w_\bu\alpha_{\tau i}}.
    \end{aligned}
\end{equation}
\end{enumerate}
Moreover, if either of (1)--(3) holds then we have 
\begin{equation}
\label{eq:upnonzero}
    \up_\mu=0 \text{ unless }w_\bu\tau(\mu)=\mu.
\end{equation} 
\end{proposition}

\begin{proof}
Recall the anti-involution $\sigma$ from \eqref{eq:sigma} and the automorphism $\tau$ from \eqref{tau:U} on $\U$. By applying $\sigma, \tau$ and using Lemma~\ref{lemma:Ti-1}, we have 
\[
\sigma \tau (B_{\tau i})=F_i+(-1)^{\delta_{i,\tau i}p(i)}\va_{\tau i}K_iT_{w_\bu}^{-1}(E_{\tau i}).
\]
Thus (1) implies (2) by applying \cite[Lemma 4.17]{BW18b}. The equivalence of (2) and (3) follows from Lemma \ref{lem:commuF}. The fact that Condition (3) implies  \eqref{eq:upnonzero} and Condition~ (1) follows from the same proof as in \cite[Proposition 6.1]{BK19}. 
\end{proof}

\begin{remark}
Comparing \eqref{eq:derivativeup} with \cite[(6.10)]{Sh22}, the key distinction lies in the additional sign $(-1)^{\delta_{i,\tau i}p(i)}$, which never occurred in the special setting of \cite{Sh22}. This extra sign arises due to $(-1)^{p(i)}\frac{q_{\tau i}-q_{\tau i}^{-1}}{q_i-q_{i}^{-1}}=(-1)^{\delta_{i,\tau i}p(i)}$, cf. \eqref{eq:qiqti}. This sign does not affect subsequent arguments involving a rerun of \cite[\S 6]{Sh22}.
\end{remark}

Moreover, it follows as in \cite[Lemma 6.5]{Sh22} that $\up_\mu=0$ unless $p(\mu)=0$. The system of equations \eqref{eq:derivativeup} for all $i\in \I$ provides an equivalent condition for the existence of $\up$, and our objective is to solve it recursively. 

To do so, we introduce several notations as follows. For any non-commutative homogeneous polynomial $p$, we write
\[ p(\underline{F}
)=f_1p^{(1)}(\underline{F}
)+\cdots+f_tp^{(t)}(\underline{F}
)\]
as a sum of monomials, where $f_k\in \C(q)^*$. For each monomial $p^{(k)}(\underline{F}
)=F_{k_1}\cdots F_{k_\ell}$, we define 
\[ \text{tail}(p^{(k)}):=k_\ell\in \I  \text{ and } {\rm{Head}} (p^{(k)})=F_{k_1}\cdots F_{k_{\ell-1}}.\]

Recall from \eqref{SU} the set $\mathcal{S}(\U)$ of non-commutative homogeneous Serre polynomials in Table~\ref{TableSerrePolyn}. Now we assume that $\up_{\mu'}$ is defined for all $\mu'<\mu$. The following proposition is a super analogue of \cite[Proposition 6.3]{BK19} which provides a concrete way to establish $\up_\mu$ satisfying \eqref{eq:derivativeup}.
\begin{proposition}
\label{prop:recursive}
 Let $\mu\in X^+_{\overline 0}$ with $ht(\mu)\geqslant 2$ and fix $A_i,\ {}_iA\in \U^{+}_{\mu-\alpha_i}$ for all $i\in I$. The following are equivalent. 
     \begin{enumerate}
         \item 
     There exist an element $\Xi\in \U_\mu^+$ such that $$\rskew(\Xi)=A_i,\quad \lskew(\Xi)={}_iA,\quad \forall i\in \I.$$
     \item
     The following properties hold for $A_i$ and ${}_iA$:
     \begin{enumerate}
     \item For all $i,j\in I$, we have 
     \begin{equation*}
r_i({}_jA)=(-1)^{p(i)p(j)}{}_ir(A_j).
\end{equation*}
\item For any $p\in \mathcal S(\U)$ with $p(\underline{F}
)=f_1p^{(1)}(\underline{F}
)+\cdots+f_tp^{(t)}(\underline{F}
)$, we have
\begin{equation*}
    \sum_{1\leq k\leq t}f_k\la Head(p^{(k)}), A_{tail(p^{(k)})} \ra=0.
\end{equation*}
     \end{enumerate}
\end{enumerate}
\end{proposition}

\begin{proof}
    $(1) \Rightarrow (2)$. (2a) follows from Lemma~\ref{lemma:rijr}, while (2b) follows from \eqref{eq:Fiy} and the fact that $\la p(\underline{F}),\Xi \ra=0$. 

    $(2)\Rightarrow (1)$ follows from a rerun of the proof of \cite[Proposition 6.3]{BK19}.
\end{proof}

Keeping \eqref{eq:derivativeup} in mind, we define
\begin{equation}\label{eq:AiiA}
    \begin{aligned}
    &A_i:=-(q_i-q_i^{-1})\up_{\mu-\alpha_i-w_\bu (\alpha_{\tau i})}(-1)^{\delta_{i,\tau i}p(i)+\alpha_i(2\rho_\bu^\vee)}{q}^{(\alpha_i,2\rho_\bu+w_\bu\alpha_{\tau i})}\va_{\tau i}  \overline{T_{w_\bu}( E_{\tau i})},\\
    &{}_iA:=-(q_i-q_i^{-1})q^{(\alpha_i,w_\bu\alpha_{\tau i})}\va_i T_{w_\bu}( E_{\tau i})\up_{\mu-\alpha_i-w_\bu\alpha_{\tau i}}.
    \end{aligned}
\end{equation}
To construct $\up_\mu$, we just need to verify that \eqref{eq:AiiA} satisfies Condition (2) in Proposition~\ref{prop:recursive} for all $i\in \I$.

\begin{lemma}
\label{lemma:2a}
    For all $i,j\in \I$, we have \[r_i({}_jA)=(-1)^{p(i)p(j)}{}_ir(A_j).\]
\end{lemma}

\begin{proof}
    The lemma follows from a rerun of the proof of \cite[Lemma 6.13]{Sh22}.
\end{proof}

Lemma~\ref{lemma:2a} confirms that $A_i$ and ${}_iA$ indeed satisfy Condition (2a) in Proposition~\ref{prop:recursive}. To verify Condition (2b), we perform a case-by-case check, similar to the approach taken in Proposition~\ref{prop:pi00p}.

\begin{theorem}
\label{thm:2b}
     For any $p\in \mathcal S(\U)$ with $p(\underline{F}
)=f_1p^{(1)}(\underline{F}
)+\cdots+f_tp^{(t)}(\underline{F}
)$,
\begin{equation}
\label{eq:2b'}
    \sum_{1\leq k\leq t}f_k\la Head(p^{(k)}), A_{tail(p^{(k)})} \ra=0.
\end{equation}
\end{theorem}

\begin{proof}
    We note that the proof of the statement in the case where $\I_\circ$ comprises solely of even simple roots has already been established in \cite[Lemma 6.8]{BK19}.

In our case, we refer to Table~\ref{TableSerrePolyn} for the relevant local Dynkin diagrams and corresponding notations. Each super admissible pair $(\I=\I_\circ \cup \I_\bu,\tau)$ encompasses multiple local Dynkin diagrams. 
We shall establish the identity \eqref{eq:2b'} case-by-case for each Serre polynomial $p\in \mathcal S(\U)$ associated with aach local diagram listed in Table~\ref{TableSerrePolyn}.

(ISO1) In this case, we have $p(x_i)=x_i^2$. Thus, $Head(p)=x_i$ and $Tail(p)=i\in \Iodd\cap \I_{\circ}$. Due to weight considerations, we observe that $\langle F_i,A_i\rangle$ must be $0$ unless $\mu=2\alpha_i$. However, if $\mu=2\alpha_i$, then we have: \[\mu-\alpha_i-w_\bu(\alpha_{\tau i})=\alpha_i-w_\bu(\alpha_{\tau i}).\] 
If $\tau i=i$, we find that $\alpha_i-w_\bu(\alpha_{\tau i})\notin X^+$ by considering Condition (4) in Definition~\ref{def:superad}. Hence, $\up_{\mu-\alpha_i-w_\bu(\alpha_{\tau i})}=0$, leading to $A_i=0$ and $\la F_i,A_i \ra=0$. If $\tau i\neq i$, then for the same reason we also have $\la F_i,A_i \ra=0$.

(ISO2) In this case, we have $p=p^{(1)}+p^{(2)}$ where $p^{(1)}=x_ix_j$ and $p^{(2)}=x_jx_i$. Therefore, we have
\[Head(p^{(1)})=F_i,\quad Head(p^{(2)})=F_j,\quad tail(p^{(1)})=j,\quad tail(p^{(2)})=i.\]
Due to weight considerations, we see that $\la F_i,A_j \ra=\la F_j,A_i \ra=0$ unless $\mu= \alpha_i+\alpha_j$. Assume that  $\mu= \alpha_i+\alpha_j$. We see that $w_\bu\tau (\mu)=\mu$ if and only if $i=\tau j$ and $\Theta(\alpha_i)=-\alpha_j$. According to \eqref{eq:vai} we have $\va_i=\va_{j}$. By \eqref{eq:AiiA} we have
$A_i=-(q_i-q_i^{-1}) q^{(\alpha_i,\alpha_{j})}E_{j}$, and $ A_j=-(q_j-q_j^{-1}) q^{(\alpha_i,\alpha_{j})}E_{i}.$ Thus by \eqref{eq:qiqti} we have
\begin{align*}
    \la F_i,A_{j}\ra+\la F_j,A_i\ra=0.
\end{align*}

(N-ISO) In this case, we may assume $\mu=(1+|a_{ij}|)\alpha_i+\alpha_j$ otherwise all terms vanish. Moreover, we observe that $a_{ij}<0$ when $i\in \Ieven$. If $j\in \Ieven$, the proof follows from \cite[Lemma 6.8]{BK19}. Hence we can assume $j\in \Iodd$ and consequently $j\in \I_\circ$. If $\up_\mu\neq 0$, we must have $w_\bu \circ \tau (\mu)=\mu$.

%Then we observe that \[\mu-\alpha_j-w_\bu \alpha_{\tau j}=(1+|a_{ij}|)\alpha_i-w_\bu\alpha_{\tau j}=(1+|a_{ij}|)\alpha_i-w_\bu\alpha_{\tau j}\notin X^+.\] Thus $A_j=0$. 

If $i\in \I_\bu$, we see that  $w_\bu\tau(\mu)=\mu$ together with $w_\bu\alpha_{\tau i}=-\alpha_i$ implies that
\[w_\bu\alpha_{\tau j}=w_\bu\circ \tau (\mu-(1+|a_{ij}|)\alpha_i)=\mu+(1+|a_{ij}|)\alpha_i=\alpha_j+2(1+|a_{ij}|)\alpha_i.\]
Hence $\tau j=j$, and $|a_{ij}|=2(1+|a_{ij}|)$, which is impossible.

If $i\in I_\circ$, then we have
\begin{equation}
\label{eq:NISO1}
(1+|a_{ij}|)(\alpha_i-\alpha_{\tau i})+(\alpha_j-\alpha_{\tau j})=(w_\bu-\text{id})\circ \tau (\mu).
\end{equation}
Since $(w_\bu-\text{id})(\alpha_k)\in \sum_{\ell \in \I_\bu}\mathbb{N} \alpha_\ell$ for all $k\in \I_\circ$, we can see that \eqref{eq:NISO1} holds only if one of the following condition holds:
\begin{enumerate}
\item $\tau i=i$ and $\tau j=j$,
\item $\tau i=j$ and $a_{ij}=0$.
\end{enumerate}
However, the case where $\tau i=j$ and $a_{ij}=0$ is impossible due to parity in Definition~\ref{def:superad}, and the case where $\tau i=i$ and $\tau j=j$ is impossible due to Condition (4) in Definition~\ref{def:superad}.

(AB) In this case, we may assume $\mu=\alpha_i+2\alpha_j+\alpha_k$ otherwise all terms vanish. We see that
\begin{equation}
    \label{eq:AB1}
(\alpha_i-\alpha_{\tau i})+(\alpha_k-\alpha_{\tau k})+2(\alpha_j-\alpha_{\tau j})=(w_\bu-id)\circ \tau(\mu)\in\sum_{t \in \I_\bu}\mathbb{N} \alpha_t.\end{equation}
Thus we must have $j=\tau j$. Due to Condition (4) in Definition~\ref{def:superad}, at least one of $i,k$ belongs to $\I_\bu$. Thus we also see from \eqref{eq:AB1} that $i=\tau i$ and $k=\tau k$.

If $i\in \I_\circ$ and $k\in \I_\bu$, then $w_\bu\tau(\mu)=\mu$ implies $\Theta(\alpha_i)=-\alpha_i$ and $i\in \Ieven$. Moreover, we have
\[ 
\mu-\alpha_i-w_\bu\alpha_{\tau i}=\mu-2\alpha_i\notin X^+,
\]
and hence $A_i=0$. Since $k\in \I_\bu$, by \eqref{eq:AiiA} we see that $A_k=0$. Moreover, we see that $\up_{\alpha_i}=0$ and hence  by the recursion \eqref{eq:AiiA} we have $A_j=0$ as well.

If $i\in \I_\bu$ and $k\in \I_\circ$, then the theorem follows similarly as in the previous case.

If $i,k\in \I_\bu$, since $\mu=\alpha_i+2\alpha_j+\alpha_k$, the super Satake diagram contains a subdiagram of the form 
$\xy
(-10,0)*{\newmoon};(0,0)*{\otimes}**\dir{-};(10,0)*{\newmoon}**\dir{-};(-10,-4)*{\scriptstyle i};
(0,-4)*{\scriptstyle j};(10,-4)*{\scriptstyle k};
\endxy$, which is excluded by our assumption in \eqref{exclude}. 

(CD1) In this case we may assume $\mu=3\alpha_j+2\alpha_k+\alpha_i$. Due to Condition (4) in Definition~\ref{def:superad}, we must have $i\in \I_\bu$ and $k$ is connected to some $t\in \I_\bu$ as well. Therefore, for weight reason we see that $w_\bu \tau \mu\neq \mu$. Hence $\up_\mu=0$.

(CD2) In this case we may assume $\mu=\alpha_i+2\alpha_j+3\alpha_k+\alpha_\ell$. Again for weight reason we have $w_\bu \tau \mu\neq \mu$. Hence $\up_\mu=0$.

(D) In this case we may assume $\mu=\alpha_k+\alpha_j+\alpha_i$ and $\tau i=i$. There are three possible cases:

If $i\in \I_\circ,\ \tau j=k$, we see that $\mu-\alpha_i-w_\bu \alpha_{\tau i}\notin X^+$, hence $A_i=0$. On the other hand, we see that $\up_{\alpha_i}=0$ and hence  by the recursion \eqref{eq:AiiA} we have $A_j=A_k=0$.

If $i\in \I_\bu,\ \tau j=j,\ \tau k=k$, we have $A_j=A_k=0$ by weight considerations.

If $i\in \I_\bu,\ \tau j=k$, we compute that 
\[ A_j=\va_k (q^2-1)\overline{T_i(E_k)},\quad A_k=\va_j (q^2-1)\overline{T_i(E_j)} \]
and 
\[p=x_kx_jx_i-x_jx_kx_i+[2](x_jx_ix_k-x_kx_ix_j)-x_ix_jx_k+x_ix_kx_j.\]
In order to prove \eqref{eq:2b'} in this case, it suffices to show that \begin{align*}
    \la F_jF_i,A_k\ra=\la F_kF_i,A_j \ra,\quad \la  F_iF_j,A_k\ra=\la F_iF_k,A_j \ra.
\end{align*}
We show that $\la F_jF_i,A_k\ra=\la F_kF_i,A_j \ra$ via the following computation and the other equality can be verified similarly.
\begin{align*}
    \la F_jF_i,A_k \ra=\,&\va_j(q^2-1)\la F_jF_i, E_iE_j-qE_jE_i \ra
    =\va_j(q^2-1)[\la F_jF_i, E_iE_j\ra -q\la F_jF_i, E_jE_i\ra]\\
    \overset{\eqref{eq:Fiy}}{=}&\va_j(q^2-1)(q^{-1}-q). \\
     \la F_kF_i,A_j \ra=\,&\va_k(q^2-1)\la F_kF_i, E_iE_k-qE_kE_i \ra
    =\va_k(q^2-1)[\la F_kF_i, E_iE_k\ra -q\la F_kF_i, E_kE_i\ra]\\
    \overset{\eqref{eq:Fiy}}{=}&\va_j(q^2-1)(q^{-1}-q). 
\end{align*}
By \eqref{eq:vai=vataui} we have $\va_j=\va_k$, and hence $\la F_jF_i,A_k\ra=\la F_kF_i,A_j \ra$.

(F1) In this case we may assume $\mu=\alpha_4+3\alpha_3+5\alpha_2+2\alpha_1$. Hence for weight reason we always have $w_\bu \tau (\mu)\neq \mu$. Thus $\up_\mu=0$.

(G2) In this case we may assume $\mu=\alpha_1+2\alpha_2+2\alpha_3$ and $3\in \I_\bu$. We compute that $w_\bu \tau (\mu)\neq \mu$. Thus $\up_\mu=0$.

(G3) In this case we may assume $\mu=\alpha_1+\alpha_2+\alpha_3$ and $2\in \I_\bu$. We compute that $w_\bu \tau (\mu)\neq \mu$. Thus $\up_\mu=0$.

The theorem is proved.
\end{proof}

We now establish a super generalization of \cite[Theorem~3.16]{WZ22} (compare \cite[Theorem 6.10]{BK19}).

\begin{theorem}
\label{thm:ibar}
Retain the assumption \eqref{exclude}.
There exists a unique element 
\[
\up =\sum_{\mu \in X^+_{\overline{0}}} \up_\mu
\qquad (\text{with } \up_0=1, \up_\mu \in \U^+_\mu), 
\]
such that 
\begin{align*}
B_{\tau i}\up &=\up \sigma \tau (B_{i}),  \qquad \text{ for } i\in \I_\circ,
\\
x \up &=\up x,  \qquad\qquad \text{ for }x \in \U^{\imath 0} \U_\bu.
\end{align*}
Moreover, $\up_\mu=0$ unless $\mu \in X^+_{\overline{0}}$ and $w_\bu \tau (\mu) =\mu$.
\end{theorem}

\begin{proof}
    Follows by combining Proposition \ref{prop:biup}, Proposition~\ref{prop:recursive}, Lemma~\ref{lemma:2a} and Theorem~\ref{thm:2b}.
\end{proof}

\begin{example}
 \label{ex:quasi-split}
Consider the following Satake diagram 
$$\xy
(-10,0)*{\otimes};
(0,0)*{\otimes}**\dir{-};(-10,-4)*{\scriptstyle -1};(0,-4)*{\scriptstyle 1};
(-10,2)\ar@/_/@{<-->} (0,2);
\endxy$$

We have $(\alpha_{1},\alpha_{-1})=1$, $d_{-1}=1=-d_1$ and
\begin{align*}
    B_{1}=F_{1}+\va_{1} E_{-1}K_{1}^{-1},\quad
    B_{-1}=F_{-1}+\va_{-1} E_{1}K_{-1}^{-1}.
\end{align*}
In this case, following the recursive constructions we get
\begin{equation*}
    \up=\Big(\sum_{k\geqslant 0}\frac{\va_{1}^k}{\{k\}!}(E_{1}E_{-1}+qE_{-1}E_{1})^k\Big)\Big(\sum_{k\geqslant 0}\frac{\va_{-1}^k}{\{k\}!}(E_{-1}E_{1}+qE_{1}E_{-1})^k \Big).
\end{equation*}
Compare with the formulas in the even case in \cite[Lemma~ 3.1]{DK19} and \cite[(4.1)]{BW18a}.
\end{example}

\begin{comment}
    \subsection{G2}
Consider the following diagram
$$ \xy (0,-6)*{\ };
{\ar@3{-}(0,0)*{\otimes};(10,0)*{\newmoon}};(5,0)*{<};(-10,0)*{\LEFTcircle};(0,0)*{\otimes}**\dir{=};
(-10,-4)*{\scriptstyle 1};(0,-3.5)*{\scriptstyle 2};(10,-3.5)*{\scriptstyle 3};
\endxy$$
where $1,3\in \I_\bu$ and $2\in \I_\circ$. The simple roots are given  by
\[ \alpha_1=\epsilon,\quad \alpha_2=\epsilon-\frac{1}{2}(\delta_1+\delta_2),\quad \alpha_3=\delta_2 \]
where $(\epsilon,\epsilon)=-2,\  (\delta_1,\delta_1)=2$ and $(\delta_2,\delta_2)=6$.

The higher degree Serre relation is given by $p=\ad_q x_2\circ \ad_q x_3 \circ \ad_q x_3\circ \ad_q x_2(x_1)-
 	\ad_q x_3\circ \ad_q x_2 \circ \ad_q x_3\circ \ad_q x_2(x_1)$ of weight $\lambda=\alpha_1+2\alpha_2+2\alpha_3$. Since $T_1T_3(E_2)$ is of weight $2\alpha_1+\alpha_2+\alpha_3$, we see that for weight reason we have \[\pi_{0,0}\circ P_{-\lambda}(p(F_1,B_2,F_3))=0.\]
  Coideal property follows similarly as in the sBII case while the quasi $K$-matrix also exists for weight reason.
\end{comment}

\section{$\imath$Schur duality}
\label{sec:duality}
In this section we establish an $\imath$Schur duality between a distinguished class of $\imath$quantum supergroups and the $q$-Brauer algebras. Such a duality generalizes the $\imath$Schur dualities of type AI and AII in \cite{CS22}.

\subsection{The $\imath$quantum supergroups of type AI-II}

We consider the following super Satake diagram $I$:
\begin{equation}
\label{eq:satake}
{\rm 
\hspace{.75in}\xy
(-20,0)*{\fullmoon};(-10,0)*{\cdots}**\dir{-};(0,0)*{\fullmoon}**\dir{-};(10,0)*{\otimes}**\dir{-};(20,0)*{\newmoon}**\dir{-};(30,0)*{\fullmoon}**\dir{-};(40,0)*{\newmoon}**\dir{-};(50,0)*{\cdots}**\dir{-};(60,0)*{\fullmoon}**\dir{-};(70,0)*{\newmoon}**\dir{-};
(-20,-4)*{\scriptstyle 1};(0,-4)*{\scriptstyle m-1};(10,-4)*{\scriptstyle m};(20,-4)*{\scriptstyle m+1};(30,-4)*{\scriptstyle m+2};(40,-4)*{\scriptstyle m+3};(70,-4)*{\scriptstyle m+2n-1};
\endxy
}\end{equation}
where $\Iodd=\{m\}$,  $I_\bu=\{m+2a-1\mid 1\leq a\leq n\}$ and $I_\circ=I \backslash I_\bu$. In the case $n=0$, we obtain a Satake diagram of type AI; when $m=0$, we obtain a Satake diagram of type AII. In the case $m=n=1$, we obtain a super Satake diagram as in Example~\ref{example:admissible}(1). Note that \eqref{eq:satake} is also a special case of \eqref{eq:splitA}.

\begin{lemma}
The pair $(\I=\I_\circ \cup \I_\bu,\tau=id)$ forms a super admissible pair.
\end{lemma}

\begin{proof}
    The conditions in Definition~\ref{def:superad} are verified by a direct computation.
\end{proof}

Moreover, for the super Satake diagram \eqref{eq:satake} we have
\begin{equation*}
    d_j=\begin{cases}
        1,& \text{ if } 1\leq j\leq m,\\
        -1,& \text{ if } m+1\leq j\leq m+2n-1.
    \end{cases}
\end{equation*}

Let $\U$ denote the quantum supergroup for $\g =\gl(m|2n)$ associated to the Dynkin diagram underlying \eqref{eq:satake} (by ignoring the black color of the nodes). Following Definition~\ref{def:Ui}, we define $\Ui$ to be the $\C(q)$-subalgebra of $\U$ generated by
\begin{align*}
 B_j= %\left\{  \begin{aligned}
 F_j+\va_jT_{w_\bu}( E_{j})  K_j^{-1},\ \ \text{for } j\in I_\circ.
% \\ &F_i,\ \ &\text{if } i\in I_\bu.
%\end{aligned}\right.
\end{align*}
together with  $K_h \ (h\in Y^\io),\  E_j, F_j \ (j\in I_\bu)$, $\new$. %In the case $m=0$, our $\Ui$ specializes to the $\imath$quantum group of type AII. In the case $n=0$, our $\Ui$ specializes to the $\imath$quantum group of type AI.

We identify $\g =\gl(m|2n)$ with the superspace of matrices parameterized by the set
\begin{equation}
\label{eq:Imn}
    I(m|2n)=\{\overline 1,\ldots,\overline {m},\underline 1,\ldots,\underline {2n}\}
\end{equation}
in total order
$$\overline 1<\cdots<\overline {m} %<0
<\underline 1<\cdots< \underline {2n}.$$ 
Let $\mathfrak h$ be the Cartan subalgebra consisting of diagonal matrices. Denote by $\{\epsilon_{a}\}_{a\in I(m|2n)}$ the basis of $\mathfrak h^*$ dual to the set of standard matrices $\{E_{a,a}\}_{a\in I(m|2n)}$. The fundamental root system $\Pi=\set{\alpha_i\mid i\in \I}$ associated to \eqref{eq:satake} is given by
\begin{equation*}
\begin{aligned}
    &\alpha_i=\epsilon_{\overline{i}}-\epsilon_{\overline{i+1}}, & 1\leq i\leq m-1,\\
    &\alpha_{m}=\epsilon_{\overline{m}}-\epsilon_{\underline{1}},&\\
    &\alpha_i=\epsilon_{\underline{i-m}}-\epsilon_{\underline{i-m+1}}, & m+1\leq i\leq m+2n-1.
\end{aligned}
\end{equation*}
where \begin{equation*}
    (\epsilon_{a},\epsilon_{a'})=
    \begin{cases}
    1 & \text{ if }  a=a'<0, \\
    -1 & \text{ if } a=a'>0, \\
    0 & \text{ else.}
    \end{cases}
\end{equation*}

Denote the natural representation of $\U$ by $\V=\Veven \oplus \Vodd$, where $\dim \Veven=m,\ \dim\Vodd=2n$, with a natural (ordered) basis $\{v_{a}\mid a\in I(m|2n)\}$. Let $|\cdot|$ denote the parity function on $\V$ where $|v_a|=0$ for all $a<0$ and $|v_a|=1$ for $a>0$.
The $\C(q)$-vector superspace $\V$ can be identified with the natural representation of $\U$, where 
\begin{align}
  \label{eq:natural}
\begin{aligned}
    &E_i v_{a}=\begin{cases}
        v_{\overline{i}}& \text{ if }1\leq i<m,\ a=\overline{i+1}, \\
        v_{\overline{m}}& \text{ if } i=m,\ a=\underline{1}, \\
        v_{\underline{i}}& \text{ if }m+1\leq i\leq m+2n-1,\ a=\underline{i+1},\\
        0&\text{else}.
        \end{cases}\\
    &F_i v_{a}=\begin{cases}
        v_{\overline{i+1}}& \text{ if }1\leq i<m,\ a=\overline{i}, \\
        v_{\underline{\mathfrak 1}}& \text{ if } i=m,\ a=\overline{m}, \\
        v_{\underline{i+1}}& \text{ if }m+1\leq i\leq m+2n-1,\ a=\underline{i},\\
        0&\text{else}.
        \end{cases}\\
        &K_iv_a=q^{(\alpha_i,\epsilon_a)}v_a.
\end{aligned}
\end{align} 
\subsection{The $q$-Brauer algebras}

The Iwahori-Hecke algebra $\Hy_{\mathfrak S_d}$ is a $\C(q)$-algebra generated by $H_1,\cdots,H_{d-1}$, subject to the following relations:
\begin{align*}
&(Q1)\ (H_i-q)(H_i+q^{-1})=0,\ \ \ \  \\
&(Q2)\ H_iH_{i+1}H_i=H_{i+1}H_iH_{i+1}, \\
&(Q3)\ H_iH_j=H_jH_i,\ \ \ \ \text{for }|i-j|>1.
\end{align*}

We recall the definition of a $q$-Brauer algebra. 

\begin{definition}{\rm (cf. \cite[Definition 3.1]{We12}) }\label{def:qB1}
Fix $d\in \Z_{\geqslant 2}$, and $z \in \C(q)^*$. The $q$-Brauer algebra $\B_n(q,z)$ is a $\C(q)$-algebra with generators $H_1,\ldots, H_{d-1}$ and $\Qy$ subject to the relations (Q1)--(Q3) above and the following relations:
\begin{equation*}
\begin{aligned}
&(Q4)\ \Qy^2=\frac{z-z^{-1}}{q-q^{-1}}\Qy,\quad(Q5)\ H_{1}\Qy=\Qy H_{1}=q\Qy,\quad(Q6)\ \Qy H_{2}\Qy=z\Qy,\\
&(Q7)\ H_i\Qy=\Qy H_i\ \ \text{for }i>2,\\
&(Q8)\ H_{2}H_{3}H_{1}^{-1}H_{2}^{-1}\Qy_{2}=\Qy_{2}=\Qy_{2}H_{2}H_{3}H_{1}^{-1}H_{2}^{-1},\\
&\qquad\text{ where } \Qy_{2}=\Qy(H_{2}H_{3}H_{1}^{-1}H_{2}^{-1})\Qy.
\end{aligned}
\end{equation*}
\end{definition}

It is known (\cite[Theorem 3.8]{We12}) that the dimension of the $q$-Brauer algebra $\B_d(q,z)$ is $(2d-1)!!$. Moreover, $\B_d(q,z)$ contains the Iwahori-Hecke algebra $\Hy_{\mathfrak S_d}$ as a subalgebra. Denote
\begin{align}  \label{club}
    \clubsuit=\set{i\in \Z\mid 4-2d\leq i\leq d-2}\backslash \set{i\in \Z\mid 4-2d<i<3-d,\ 2\nmid i}.
\end{align}
A criterion for the semisimplicity of the $q$-Brauer algebra is given as follows.

\begin{proposition}\cite[Theorem A]{RSS24}
\label{prop:semisimple}
For $d\geq 2$, the $q$-Brauer algebra $\B_d(q,z)$ is (split) semisimple if and only if $z^2\neq q^{2a}$ for $a\in \clubsuit$ in \eqref{club}.
\end{proposition}

\subsection{Action of the $q$-Brauer algebra}

It is well known that $\V^{\otimes d}$ can be endowed with a right $\Hy_{\mathfrak S_d}$-module structure via the following action (cf. \cite{Mi06,Sh22}): 
\begin{equation}
\label{eq:HA}
\begin{aligned}
&v_{a_1}\otimes\ldots \otimes v_{a_n}\cdot H_j\\
=&\left\{\begin{aligned}
&q^{1-2|v_{a_j}|}v_{a_1}\otimes\ldots \otimes v_{a_n}&\hbox{ if } a_j=a_{j+1}, \\
&(-1)^{|v_{a_j}||v_{a_{j+1}}|}\ldots\otimes v_{a_{j+1}}\otimes v_{a_j}\otimes \ldots\ \ &\hbox{ if } a_{j}>a_{j+1}, \\
&(-1)^{|v_{a_j}||v_{a_{j+1}}|}\ldots \otimes v_{a_{j+1}}\otimes v_{a_j}\otimes \ldots+(q-q^{-1})v_{a_1}\otimes\ldots \otimes v_{a_n}\ \ &\hbox{ if } a_{j}<a_{j+1}.
\end{aligned}
\right.
\end{aligned}
\end{equation}

For $i=2,\ldots,m$ and $j=2,\ldots,n$, we set $$\tau_i:=\prod_{j=1}^{i-1}(-\va_j),\quad \zeta_j:=\prod_{k=1}^{j}(\va_{m+2k-2}),\quad \tau_1=1.$$

A $\C(q)$-linear operator $\Xi$ on $\V\otimes \V$ is defined by
\begin{equation}
    \label{eq:operatorXi}
    \begin{aligned}
    \Xi(v_{\overline 1}\otimes v_{\overline 1}) &=q^{-2n}\Big(\sum_{i=1}^{m}\tau_i^{-1} q^{m-2i+1}v_{\overline{i}}\otimes v_{\overline{i}}\Big)
           \\
    &\quad -\tau_{m}^{-1}q^{-m-2n}\Big(\sum_{j=1}^{n} \zeta_j^{-1}q^{j+1}(v_{\underline{2j-1}}\otimes v_{\underline{2j}}-q^{-1}v_{\underline{2j}}\otimes v_{\underline{2j-1}})\Big),
    \\
    \Xi(v_{\overline{i}}\otimes v_{\overline{i}}) &=\tau_i\Xi(v_{\overline 1}\otimes v_{\overline 1}),\ \text{ for } 2\leq i\leq m,
    \\
    \Xi(v_{\underline{1}}\otimes v_{\underline{2}}) &=\tau_{m}\va_m\Xi(v_{\overline 1}\otimes v_{\overline 1}), \qquad\quad
           \Xi(v_{\underline{2}}\otimes v_{\underline{1}})=(-q^{-1})\Xi(v_{\underline{1}}\otimes v_{\underline{2}}), 
    \\
    \Xi(v_{\underline{2j-1}}\otimes v_{\underline{2j}}) &=\zeta_j\va_m^{-1} q^{3j-2} \Xi(v_{\underline{1}}\otimes v_{\underline{2}}),\ \text{ for all } 2\leq j\leq n,
    \\
    \Xi(v_{\underline{2j}}\otimes v_{\underline{2j-1}}) &=(-q^{-1})\Xi(v_{\underline{2j-1}}\otimes v_{\underline{2j}}),\ \text{ for all } 2\leq j\leq n,
    \\
    \Xi(v_a\otimes v_b) &=0,\ \text{ if } (a,b)\notin \{(\overline{i},\overline{i}),(\underline{2j-1},\underline{2j}),(\underline{2j},\underline{2j-1})\mid 1\leq i\leq m,1\leq j\leq n\}.
        \end{aligned}
    \end{equation}

\begin{proposition}
\label{prop:braueract}
For $d\ge 2$, $\V^{\otimes d}$ is a right $\B_d(q,q^{m-2n})$-module by letting $\Hy_{\mathfrak S_d}$ act as in \eqref{eq:HA} and $\Qy$ act as $\Xi\otimes 1^{\otimes d-2}$.
\end{proposition}

\begin{proof}
By \cite{Mi06}, the action of $H_i$ satisfies relations (Q1)--(Q3) in Definition \ref{def:qB1}. The verification of the relation (Q4)--(Q7) is very similar to the proof of \cite[Propositions 4.4 and 5.3]{CS22}. 
We prove the relation (Q4) for both statements as an example. Noting that the action of $\Qy$ depends solely on the first two tensor factors, hence it suffices to show that $$v_{\overline{1}}\otimes v_{\overline{1}}\cdot\Qy^2=\frac{q^{m-2n}-q^{-m+2n}}{q-q^{-1}}v_{\overline{1}}\otimes v_{\overline{1}}\cdot\Qy.$$

Indeed we have
\begin{align*}
&v_{\overline{1}}\otimes v_{\overline{1}}\cdot\Qy^2\\
=\,&q^{-2n}\Big(\sum_{i=1}^{m}\tau_i^{-1} q^{m-2i+1}v_{\overline{i}}\otimes v_{\overline{i}}\cdot \Qy\Big)\\
           &\qquad-\tau_{m}^{-1}q^{-m-2n}\Big(\sum_{j=1}^{n} \zeta_j^{-1}q^{j+1}(v_{\underline{2j-1}}\otimes v_{\underline{2j}}-q^{-1}v_{\underline{2j}}\otimes v_{\underline{2j-1}})\cdot \Qy\Big)\\
    =\,&\Bigg[q^{-2n}\Big(\sum_{i=1}^{m}q^{m-2i+1}\Big)-q^{-m-2n}\Big(\sum_{j=1}^{n}q^{4j-1}(1+q^{-2})\Big)\Bigg]v_{\overline{1}}\otimes v_{\overline{1}}\cdot\Qy\\
    =\,&\left(q^{-2n}\frac{q^{m}-q^{-m}}{q-q^{-1}}-q^{-m}\frac{q^{2n}-q^{-2n}}{q-q^{-1}}\right)v_{\overline{1}}\otimes v_{\overline{1}}\cdot\Qy\\
    =\,&\frac{q^{m-2n}-q^{2n-m}}{q-q^{-1}}v_{\overline{1}}\otimes v_{\overline{1}}\cdot\Qy.
\end{align*}
This concludes the proof of the relation (Q4).
\end{proof}

\subsection{The $\Ui$-action}

By a direct calculation using the $\U$-action on $\V$ in \eqref{eq:natural} we obtain that, for $i\in \Ieven$, 
\begin{align}
\label{eq:Biact}
B_i\cdot v_a =\begin{cases}
v_{\overline{i+1}}& \hbox {if } 1\leq i\leq m-1, a=\overline{i}, \\
q\va_i v_{\overline{i}} & \hbox {if } 1\leq i\leq m-1, a=\overline{i+1},\\
v_{\underline{1}}& \hbox {if } i=m, a=\overline{m},\\
-q\va_{m}v_{\overline{m}}& \hbox {if } i=m, a=\underline{ 2},\\
v_{\underline {2k+1}}& \hbox {if } i=m+2k,a=\underline{2k},\\
-q\va_{m+2k} v_{\underline{2k-1}} & \hbox {if } i=m+2k,a=\underline{2k+2},\\
0& \hbox {else}.
\end{cases}
\end{align}

Via the comultiplication $\Delta$, we view $\V^{\otimes d}$ as a $\U$-module and a $\Ui$-module.

\begin{lemma}
\label{lem:Bineqmcommute}    
The action of $B_i$ for $i\neq m$ commutes with the action of $\B_d(q,q^{m-2n})$ on $\V^{\otimes d}$.
\end{lemma}
\begin{proof}
    Suppose $i\neq m$. By \cite{Mi06}, we know that the actions of $\U$ commute with the action of $H_j$, for $1\le j \le d-1$. Hence the action of $B_i$ also commutes with the action of $H_j$, for $1\le j \le d-1$. Thus to prove the lemma we only need to show that the action of $B_i$ commutes with the action of $e$. Since $e$ acts on $\V^{\otimes d}$ via $\Xi\otimes 1^{\otimes d-2}$, it suffices to assume that $d=2$ and check that $B_i$ commutes with $\Xi$ on $\V\otimes \V$.

    By \eqref{eq:operatorXi}, the action of $\Xi$ on $\V\otimes \V$ depends heavily on $\Xi(v_{\overline{1}}\otimes v_{\overline{1}})$. We shall show that $B_i$ commutes with $\Xi$ when acting on $v_{\overline{1}}\otimes v_{\overline{1}}$. The checking of the rest actions are either similar or straightforward. We divide the proof into the following two cases:

    (1) $1\le i\le m-1$: In this case we have
    \begin{align}
        \label{eq:comultiBAI}
        \Delta(B_i)=B_i\otimes K_i^{-1}+ \new \otimes F_i
    \end{align}
    By \eqref{eq:Biact} and \eqref{eq:comultiBAI}, we see that
    \[ [B_i\cdot (v_{\overline{1}}\otimes v_{\overline{1}})]\cdot e=\Xi \cdot B_i(v_{\overline{1}}\otimes v_{\overline{1}})=0. \]
    Thus we need to show $B_i\cdot \Xi(v_{\overline{1}}\otimes v_{\overline{1}})=0$ as well. To that end, We write $\Xi(v_{\overline{1}}\otimes v_{\overline{1}})=q^{-2n}L_1-\tau_m^{-1}q^{-m-2n}L_2$ where $$L_1=\sum_{i=1}^{m}\tau_i^{-1} q^{m-2i+1}v_{\overline{i}}\otimes v_{\overline{i}},\ L_2=\sum_{j=1}^{n} \zeta_j^{-1}q^{j+1}(v_{\underline{2j-1}}\otimes v_{\underline{2j}}-q^{-1}v_{\underline{2j}}\otimes v_{\underline{2j-1}}).$$
    We see that the sum $L_1$ is proportional to the action defined in \cite[Proposition 4.4]{CS22}. Hence by \cite[Proposition Theorem~4.6]{CS22}, we see that $B_i(L_1)=0$. By \eqref{eq:Biact} and \eqref{eq:comultiBAI} we have $B_i(L_2)=0$ as well. Hence $B_i\cdot \Xi(v_{\overline{1}}\otimes v_{\overline{1}})=0$.

    (2) $m<i<m+2n-1$: In this case we may suppose $i=m+2l$ for some $1\le l\le n-1$ by \eqref{eq:satake}. Then we calculate directly that (cf. \cite[Example 7.9]{Ko14})
    \begin{equation}
    \label{eq:comultiBAII}
\begin{aligned}
\Delta(B_{i})|_{\V^{\otimes 2
}}=\,&B_{i}\otimes K_{i}^{-1}+\new \otimes F_{i}+\va_{i}(q^2-1)E_{i+1} K_{i-1}\otimes E_{i-1}E_{i}K_i^{-1}\\
&-\va_{i}(q^3-q)E_{i-1}K_{i+1}\otimes E_{i}E_{i+1}K_i^{-1}+q\va_{i} \new K_{i-1}K_{i+1}\otimes E_{i-1}E_{i}E_{i+1}.
\end{aligned}
\end{equation}
 By \eqref{eq:Biact} and \eqref{eq:comultiBAII}, we see that
    \[ [B_i\cdot (v_{\overline{1}}\otimes v_{\overline{1}})]\cdot e=\Xi \cdot B_i(v_{\overline{1}}\otimes v_{\overline{1}})=0. \]
    Thus we need to show $B_i\cdot \Xi(v_{\overline{1}}\otimes v_{\overline{1}})=0$ as well. To that end, We write $\Xi(v_{\overline{1}}\otimes v_{\overline{1}})=q^{-2n}L_1-\tau_m^{-1}q^{2n-m}L_2$ as in case (1).

By \eqref{eq:Biact} and \eqref{eq:comultiBAII}, we observe that $B_i(L_1)=0$. On the other hand, we observe that $L_2$ is proportional to the action defined in \cite[Propositon 5.3]{CS22}
after switching $q$ and $q^{-1}$ (which is due to \eqref{eq:satake}). Hence by \cite[Theorem 5.4]{CS22} we have  we have $B_i(L_2)=0$ as well. Thus $B_i\cdot \Xi(v_{\overline{1}}\otimes v_{\overline{1}})=0$ as we desired.
\end{proof}

\begin{theorem}
\label{thm:duality}
Let $m,n \ge 1$. 
\begin{enumerate}
    \item 
The left action of $\Ui$ on $\V^{\otimes d}$ commutes with the right action of $\B_d(q,q^{m-2n})$ defined in Proposition~\ref{prop:braueract}:
\[
\Ui \stackrel{\Psi}{\curvearrowright} \V^{\otimes d} \stackrel{\Phi}{\curvearrowleft} \B_d(q,q^{m-2n}).
\]
\item 
The following double centralizer property holds if $\pm (m-2n) \notin \clubsuit$ \eqref{club}:
\begin{align*}
\Psi(\Ui) &=\End_{\B_d(q,q^{m-2n})} (\V^{\otimes d}),\\
\Phi(\B_d(q,q^{m-2n}))&=\End_{\Ui} (\V^{\otimes d}).
\end{align*}
\end{enumerate}
\end{theorem}

\begin{proof}
    (1) By \cite{Mi06}, we know that the actions of $\U$ commute with the action of $H_i$, for $1\le i \le d-1$. Moreover, by \cite{CS22}, we know that the actions of $B_i$ for $i\neq m$ and $\U_\bu$ commutes with the action of $\B_d(q,q^{m-2n})$.
    
    Therefore, to prove (1)--(2) it suffices to show that the action of $B_{m}$ commutes with the action of $\Qy$. Since $\Qy$ acts on $\V^{\otimes d}$ by $\Xi\otimes 1^{\otimes d-2}$, it suffices to verify the commuting actions on the first two tensor factors.

    By definition we have 
    $ B_{m}=F_{m}+q\va_{m}K_{m}^{-1}T_{m+1}(E_{m})$. We  first observe that 
    \begin{equation}
    \label{eq:FXi}
        \begin{aligned}
            F_{m}[\Xi(v_{\overline 1}\otimes v_{\overline 1})]
            &=(F_{m}\otimes K_{m}^{-1}+\new\otimes F_{m})[\Xi(v_{\overline{1}}\otimes v_{\overline 1})]\\
            &=\tau_{m}^{-1}q^{-2n-m+1}(q^{-1}v_{\underline{1}}\otimes v_{\overline{m}}+v_{\overline{m}}\otimes v_{\underline{1}}).
        \end{aligned}
    \end{equation}
    Next we compute that
    \begin{equation}
        \label{eq:EXi}
        \begin{aligned}
            E_{m}E_{m+1}[\Xi(v_{\overline 1}\otimes v_{\overline 1})]&=0,
            \\
            E_{m+1}E_{m}[\Xi(v_{\underline 1}\otimes v_{\underline 1})]&=-\tau_{m}^{-1}\va_m^{-1} q^{2-m-2n}E_{m+1}E_{m}(v_{\underline 1}\otimes v_{\underline 2}-q^{-1}v_{\underline 2}\otimes v_{\underline 1})\\
            &= -\tau_{m}^{-1}\va_m^{-1} q^{2-m-2n}(q^{-1}v_{\underline{1}}\otimes v_{\overline{m}}+v_{\overline{m}}\otimes v_{\underline{1}}).
        \end{aligned}
    \end{equation}
    Combine \eqref{eq:FXi} and \eqref{eq:EXi} we see that
    \begin{align*}
        B_{m}[\Xi(v_{\overline{1}}\otimes v_{\overline{1}})]
        &= \tau_{m}^{-1}(q^{-2n-m+1}-q^{-m-2n+1})(q^{-1}v_{\underline{1}}\otimes v_{\overline{m}}+v_{\overline{m}}\otimes v_{\underline{1}})=0,\\
        \Xi[B_{m}\cdot(v_{\overline{1}}\otimes v_{\overline{1}})]&=0.
    \end{align*}
    Therefore we see that
$B_{m}$ and $\Qy$ commute on the basis vector $v_{\overline{1}}\otimes v_{\overline{1}}$. The verification on other basis vectors are similar.

(2) It is well known that the double centralizer property is synonymous with the multiplicity-free decomposition of $\V^{\otimes d}$ as an ($\Ui, \B_d(q,q^{m-2n})$)-bimodule if the action of $\B_d(q,q^{m-2n})$ is semisimple (cf. \cite[\S3.1.3]{CW12}). Proposition~\ref{prop:semisimple} affirms that the $q$-Brauer algebra $\B_d(q,q^{m-n})$ is semisimple given our assumption. Consequently, proving the double centralizer property is reduced, through a deformation argument, to the case where $q=1$. In the limit as $q$ tends to $1$ and $\va_i$ takes on $-1$ for $1\leq i\leq m-1$, $\U^{\io}$ transforms into the enveloping algebra of the orthosymplectic Lie algebra $osp(m |2n)$, while $\V$ becomes its natural representation. The multiplicity-free decomposition of $\V^{\otimes d}$  has already been established in \cite{ES16}. This concludes the proof.
\end{proof}

\begin{remark}
    The duality presented in Theorem~\ref{thm:duality} unifies the two $\imath$Schur dualities of types AI and AII \cite{CS22}; also see \cite{Mol03} for an earlier variant of such $\imath$Schur duality where a different version of the $q$-Brauer algebra was used. This duality can be further extended to arbitrary super Satake diagrams of the form \eqref{eq:splitA}.
\end{remark}

\begin{remark}
\label{rem:yesno}
    The $\Ui$ above is a new $q$-deformation of the $\text{osp}$ Lie superalgebras, and it is actually one of many different $q$-deformations of the $\mathfrak{osp}$ Lie superalgebras associated to different super type A Satake diagrams \eqref{eq:splitA} in Example~\ref{ex:superA}. The quantum supergroups associated to the same $\g$ with different super Dynkin diagrams are isomorphic as algebras (due to Yamane \cite{Ya94, Ya99}), but such isomorphisms do not preserve the Hopf algebra structures. Hence the quantum supersymmetric pairs associated with different super Satake diagrams are (expected to be) non-isomorphic as Yamane's algebra isomorphisms on quantum supergroups do not respect the coideal subalgebras. 
    
    It is an interesting and nontrivial question to ask if there are algebra isomorphisms among the $\imath$quantum supergroups (compare the classical setting in Example~\ref{ex:superA}). 
\end{remark}

\appendix
\section{Quantum supersymmetric pairs when $\I_\bu \cap \I_\niso \neq \emptyset$}
\label{sec:non-isotropic}

It is required that $\I_\bu\subset \Ieven$ in Definition~\ref{def:superad} of super admissible pairs $(\I=\I_\circ \cup \I_\bu,
\tau)$. In this appendix, we relax the requirement to allow $\I_\bu$ to contain non-isotropic odd simple roots. Non-isotropic odd simple roots appear in Dynkin diagrams associated with type B and type G Lie superalgebras. 

Recall in the definition of $\Ui$, we utilize braid group operators associated with $I_\bu$ in \eqref{eq:Bi}. The braid group operators associated with non-isotropic odd simple roots in general are constructed in \cite[Lemma 7.4.1]{Ya99}. 

Recall $\LEFTcircle$ denotes a non-isotropic odd simple root. In type B, there are four new super Satake diagrams of real rank one, depicted in \eqref{eq:sBI1}--\eqref{eq:sBII-1} below, for $n\ge 1$:
\begin{equation}
    \label{eq:sBI1}
    \xy 
(-10,0)*{\otimes};
(0,0)*{\newmoon}**\dir{-};
(10,0)*{\cdots}**\dir{-};
(20,0)*{\newmoon}**\dir{-};(30,0)*{\LEFTcircle}**\dir{=};
(25,0)*{>};
(-10,-4)*{\scriptstyle 0};(0,-4)*{\scriptstyle 1};(20,-4)*{\scriptstyle n-1};(30,-4)*{\scriptstyle n};
(45,0)*{(\text{sBIII}_{1})}
\endxy
\end{equation}
\begin{equation}
    \label{eq:sBI-1}
     \xy 
(-10,0)*{\fullmoon};
(0,0)*{\newmoon}**\dir{-};
(10,0)*{\cdots}**\dir{-};
(20,0)*{\newmoon}**\dir{-};(30,0)*{\LEFTcircle}**\dir{=};
(25,0)*{>};
(-10,-4)*{\scriptstyle 0};(0,-4)*{\scriptstyle 1};(20,-4)*{\scriptstyle n-1};(30,-4)*{\scriptstyle n};
(45,0)*{(\text{sBIII}_{-1})}
\endxy
\end{equation}
where $I_\bu =\{1,2,\ldots,n\}$ and $\I_\circ =\{0\}$ for both \eqref{eq:sBI1}--\eqref{eq:sBI-1}; and
\begin{equation}
    \label{eq:sBII1}
    \xy 
(-20,0)*{\newmoon};
(-10,0)*{\otimes}**\dir{-};
(0,0)*{\newmoon}**\dir{-};
(10,0)*{\cdots}**\dir{-};
(20,0)*{\newmoon}**\dir{-};(30,0)*{\LEFTcircle}**\dir{=};
(25,0)*{>};
(-10,-4)*{\scriptstyle 0};(0,-4)*{\scriptstyle 1};(20,-4)*{\scriptstyle n-1};(30,-4)*{\scriptstyle n};(-20,-4)*{\scriptstyle -1};
(45,0)*{(\text{sBIV}_{1})}
\endxy
\end{equation}
\begin{equation}
    \label{eq:sBII-1}
    \xy 
(-20,0)*{\newmoon};
(-10,0)*{\fullmoon}**\dir{-};
(0,0)*{\newmoon}**\dir{-};
(10,0)*{\cdots}**\dir{-};
(20,0)*{\newmoon}**\dir{-};(30,0)*{\LEFTcircle}**\dir{=};
(25,0)*{>};
(-10,-4)*{\scriptstyle 0};(0,-4)*{\scriptstyle 1};(20,-4)*{\scriptstyle n-1};(30,-4)*{\scriptstyle n};(-20,-4)*{\scriptstyle -1};
(45,0)*{(\text{sBIV}_{-1})}
\endxy
\end{equation}
where  $I_\bu =\{-1,1,\ldots,n\}$ and $I_\circ =\{0\}$  for both \eqref{eq:sBII1}--\eqref{eq:sBII-1}.

We shall refer to the super Satake diagrams \eqref{eq:sBI1}, \eqref{eq:sBI-1}, \eqref{eq:sBII1} and \eqref{eq:sBII-1} as type sBIII${}_1$, sBIII${}_{-1}$, sBIV${}_{1}$ and sBIV${}_{-1}$, respectively.

\subsection{Super type sB}
\label{subsec:BI}
In this subsection we zoom into sBIII${}_\pi$ with $\pi=\pm1$ to study the behavior of the non-isotropic roots.

Let $\U_\pi$ denote the corresponding quantum supergroup associated with sBIII${}_\pi$. In both cases, we define the $\imath$quantum supergroup $\Ui_\pi$ to be the $\C(q)$-subalgebra generated by $\U_{\pi,\bu}$ and $B_{n-1}$ where 
\[ B_{0}:=F_{0}+\va_{0}T_1\cdots T_{n-1}T_nT_{n-1}\cdots T_1(E_{0})K_{0}^{-1} \]

According to  \cite[Lemma 7.4.1]{Ya99}, we have 
\begin{equation}
\label{eq:Tn}
\begin{aligned}
    &T_n(E_{n-1})=x_{n-1}(E_n^2E_{n-1}+\pi(q^\pi-1)E_nE_{n-1}E_n-q^\pi E_{n-1}E_n^{2}),\\
    &T_n(F_{n-1})=y_{n-1}(F_n^2F_{n-1}+\pi(q^\pi-1)F_nF_{n-1}F_n-q^\pi F_{n-1}F_n^{2}).
    \end{aligned}
\end{equation}
where $x_{n-1}y_{n-1}=q^\pi$.

\begin{lemma}
\label{lemma:BI}
The following identities hold in $\U_\pi$:

(1) $T_n(E_{n-1})=x_{n-1}\ad(E^2_n)(E_{n-1})$. 

(2) $T_1\cdots T_{n-1}T_nT_{n-1}\cdots T_1(E_{0})=x_{n-1}\ad(E_1\cdots E_{n-1}E^2_nE_{n-1}\cdots E_1)(E_{0})$. 
\end{lemma}

\begin{proof}
We only prove for $\pi=1$ as the other case is similar.

    (1) By definition we have \[  \Delta(E_n^2)=\Delta(E_n)^2=E_n^2\otimes 1+E_n\new K_n\otimes E_n+\new K_nE_n\otimes E_n+K_n^2\otimes E_n^2. \] Hence
    \begin{align*}
        \ad(E_n^2)(E_{n-1})=\,&E_n^2E_{n-1}+E_n\new K_nE_{n-1}(-\new K_n^{-1}E_n)\\
        &\quad+\new K_nE_n E_{n-1}(-\new K_n^{-1}E_n)+K_n^2E_{n-1}(-\new K_n^{-1}E_n)^2\\
        =\,&E_n^2E_{n-1}+(q-1)E_nE_{n-1}E_n-qE_{n-1}E_n^2=\frac{1}{x_{n-1}}
        T_n(E_{n-1}).
    \end{align*}

    (2) follows from (1) and \eqref{eq:adZ+}.
\end{proof}

\begin{proposition}
    $\Ui_\pi$ is a right coideal subalgebra of $\U_\pi$.
\end{proposition}
\begin{proof}
    It follows from a rerun of the proof of Proposition~\ref{prop:coideal} and Lemma~\ref{lemma:BI}.
\end{proof}

\begin{comment}

\begin{align*}
    &T_n(E_{n-1})=q[E_3,[E_3,E_2]_q]_q=q (E_3^2E_2+(q-1)E_3E_2E_3-qE_2E_3^2)=q\ad(E_3^2)(E_2)\\
    &T_3(F_2)=[F_3,[F_3,F_2]_q]_q=F_3^2F_2+(q-1)F_3F_2F_3-qF_2F_3^2\\
    &T_3(K_2)=K_2K_3^2.
\end{align*}
Therefore, we see that $T_1T_3(E_2)=q\ad(E_1E_3^2)(E_2)$.

Also we have 
\[ B_2=F_2+\va_2 T_1T_3(E_2)K_2^{-1}. \]
Hence by the same proof as in Proposition~\ref{prop:coideal}, the subalgebra generated by $B_2$ and $U_\bu$ is still a coideal subalgebra of $\U$.
\end{comment}

Recall $x_{n-1},y_{n-1}$ from \eqref{eq:Tn}. In the context of the quantum supersymmetric pair associated with \eqref{eq:sBI1} and \eqref{eq:sBI-1}, the role played by $T_{w_\bu}$ in the usual quantum (super)symmetric pair is now taken up by the operator $\mathcal T_{\rm III}:=T_1\cdots T_{n-1}T_nT_{n-1}\cdots T_1$. Therefore we have $B_{0}=F_{0}+\va_{0}{\mathcal T}_{\rm III}(E_0)K_{0}^{-1}$.

\begin{proposition}
\label{prop:BI}
    When $x_{n-1}=\pm 1$, the braid group operators $T_n$ restrict to an algebra isomorphism of $\Ui_\pi$.
\end{proposition}
\begin{proof}
When $x_{n-1}=\pm 1$, we have $y_{n-1}=\pm q^\pi$.
Note that $T_n(E_n)=-F_nK_n$ and $T_n(F_n)=K_n^{-1}E_n$, see \cite[Lemma 7.4.1]{Ya99}.
    It suffices to show that $T_n(B_{0})\in \Ui_\pi$. Thus we only need to consider the case when $n=1$. In that case we have
    \begin{align*}
        T_1(B_0)=\,& T_1(F_0)+\va_{0}T_1\mathcal T_{\rm III}(E_{0})T_1(K_0^{-1})\\
        =\,&T_1(F_0)+\va_0\mathcal T_{\rm III}T_1(E_0)T_1(K_0^{-1})
        \\
        =\,&y_{0}(F_1^2F_0+\pi(q^\pi-1)F_1F_0F_1-q^\pi F_0F_1^{2})+\\
        &\qquad+x_{0}\va_0\mathcal T_{\rm III}(E_1^2E_0+\pi(q^\pi-1)E_1E_0E_1-q^\pi E_0E_1^{2})K_1^{-2}K_0^{-1}\\
        =\,&y_{0}(F_1^2F_0+\pi(q^\pi-1)F_1F_0F_1-q^\pi F_0F_1^{2})+\\
        &\qquad+x_{0}\va_0(T_1(E_1)^2\mathcal T_{\rm III}(E_0)+\pi(q^\pi-1)T_1(E_1)\mathcal T_{\rm III}(E_0)T_1(E_1)\\
        &\qquad-q^\pi\mathcal T_{\rm III}(E_0)T_1(E_1)^{2})K_1^{-2}K_0^{-1}\\
        =\,&y_{0}(F_1^2F_0+\pi(q^\pi-1)F_1F_0F_1-q^\pi F_0F_1^{2})+\\
        &\qquad+q^{\pi}x_{0}\va_0(F_1^2\mathcal T_{\rm III}(E_0)K_0^{-1}+\pi(q^\pi-1)F_1\mathcal T_{\rm III}(E_0)K_0^{-1}F_1\\
        &\qquad-q^\pi\mathcal T_{\rm III}(E_0)K_0^{-1}F_1^{2})\\
        =\,&y_{0}(B_1^2B_0+\pi(q^\pi-1)B_1B_0B_1-q^\pi B_0B_1^{2})\in \Ui_\pi.
    \end{align*}
    This concludes the proof.
\end{proof}

\begin{remark}
Consider the other two real rank one super Satake diagram of type sBIV${}_\pi$, i.e. \eqref{eq:sBII1} and \eqref{eq:sBII-1}. Abusing the notation a little bit, we let $\U_\pi$ denote the corresponding quantum supergroup associated with sBIV${}_\pi$. In both cases, we define the $\imath$quantum supergroup $\Ui_\pi$ to be the $\C(q)$-subalgebra generated by $\U_{\pi,\bu}$ and $B_{0}:=F_{0}+\va_{0}\mathcal{T}_{\rm IV}(E_{0})K_{0}^{-1}$ where 
\[ \mathcal T_{\rm IV}:=T_{-1} \cdot T_1\cdots T_{n-1}T_nT_{n-1}\cdots T_1.\]

As in the case sBIII$_\pi$, the operator $\mathcal{T}_{\rm IV}$ plays the role of $T_{w_\bu}$ in the usual quantum (super)symmetric pairs. Recall the braid group operator $T_n$ from \eqref{eq:Tn}. By a similar approach as in Lemma~\ref{lemma:BI} and Proposition~\ref{prop:BI} one can show that $\U_\pi^\io$ is a right coideal subalgebra of $\U_\pi$.
\end{remark}

\subsection{Quantum Iwasawa decomposition}
\label{subsec:QIDBIBII}
We continue to use the notations from Section~\ref{sec:QSP}. Now we consider the super admissible pair $(I_B=I_\circ\cup I_\bu,id)$ where $I_B$ is a Dynkin diagram associated with a type B Lie superalgebra and $I_B$ contains a non-isotropic odd simple root which lies in $I_\bu$. In this case, the possible real rank one subdiagrams of $(I_B=I_\circ\cup I_\bu,id)$ can only be of type sAI, sAII, sBIII${}_\pi$ and sBIV${}_\pi$, for $\pi =\pm 1$. 

By the constructions in \S~\ref{subsec:BI}, we can now define a quantum supersymmetric pair $(\U_B,\Ui)$ following Definition~\ref{def:Ui} (where we substitute some parts of $T_{w_\bu}$ by $\mathcal{T}_{\rm III}$ and $\mathcal{T}_{\rm IV}$ when dealing with subdiagrams of type sBIII${}_\pi$ and sBIV${}_\pi$ respectively).

We want to establish the quantum Iwasawa decomposition with respect to $(\U_B,\Ui)$, the key step is to prove the following variant of Prposition~\ref{prop:pi00p} for $\Ui$.
\begin{proposition}
    Given a type B super admissible pair $(I_B=\I_\circ \cup \I_\bu,\tau)$, we have 
    \[ \pi_{0,0}\circ P_{-\lambda}(p(\underline{B}))\in \Uio\]
    for any  $p\in \mathcal S(\U_B)$ of weight $\lambda$.
\end{proposition}

\begin{proof}
Recall that  the possible real rank one subdiagrams of $(I_B=I_\circ\cup I_\bu,id)$ can only be of type sAI, sAII, sBIII${}_\pi$ and sBIV${}_\pi$. According to Table~\ref{TableSerrePolyn}, the potential local Serre relations of $\U_B$ are (ISO1), (N-ISO), and (AB). For (ISO1) and (N-ISO) of all cases, the statement follows from the same proof as in Proposition~\ref{prop:pi00p}.

     As for (AB) of $\U_B$, we have $p=[2]x_2x_3x_1x_2-(x_1x_2x_3x_2+x_2x_3x_2x_1-x_3x_2x_1x_2-x_2x_1x_2x_3)$. By weight reason the statement is automatic true unless $I_B$ contains a subdiagram of the form $\xy
(-10,0)*{\newmoon};(0,0)*{\otimes}**\dir{-};(10,0)*{\newmoon}**\dir{-};(-10,-4)*{\scriptstyle i};
(0,-4)*{\scriptstyle j};(10,-4)*{\scriptstyle k};
\endxy$, which is excluded by our assumption in \eqref{exclude}. 
This concludes the proof.
\end{proof}
Consequently, there exists a quantum Iwasawa decomposition for $(\U_B,\Ui)$ in the sense of Theorem~\ref{thm:Iwa}.

\subsection{Quasi $K$-matrix}
Keep the assumption of $(I_B,\tau)$ in \S~\ref{subsec:QIDBIBII}. 
The quasi $K$-matrix associated with $(\U_B,\Ui)$ can be defined in the same way as in \S~\ref{sec:quasiK} via the intrinsic relation \eqref{eq:biup}. The key step toward the construction is to prove the following variant of Theorem~\ref{thm:2b} for $(\U_B,\Ui)$.
\begin{proposition}
      For any $p\in \mathcal S(\U_B)$ with $p(\underline{F}
)=f_1p^{(1)}(\underline{F}
)+\cdots+f_tp^{(t)}(\underline{F}
)$, we have
\begin{equation*}
    \sum_{1\leq k\leq t}f_k\la Head(p^{(k)}), A_{tail(p^{(k)})} \ra=0.
\end{equation*}
\end{proposition}

\begin{proof}
   Once again, as indicated in Table~\ref{TableSerrePolyn}, the potential local Serre relations consist of (ISO1), (N-ISO), and (AB). To substantiate the assertion, it is adequate to provide a proof for (AB), as the reasoning for the other two cases aligns with that in Theorem~\ref{thm:2b}.

   Regarding (AB) of $\U_B$, we can assume $\mu=2\alpha_{n-1}+\alpha_{n-2}+\alpha_n$; otherwise, the statement is automatically true. However, through a direct computation, it becomes evident that $w_\bu\circ \tau (\mu)\neq \mu$, thus implying $\up_\mu=0$.
\end{proof}
Consequently, there exists a unique element $\up=\sum_{\mu}\up_\mu$, with $\up_0=1$ and $\up_\mu\in (\U_B)^{+}_\mu$, such that the equality \eqref{eq:biup} holds for all $i\in I$.

\subsection{Type $G_3$}

    Another case where the non-isotropic simple root shows up is in a type G Lie superalgebra with Dynkin diagram of the form $$ 
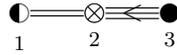
\xy (0,-6)*{\ };
{\ar@3{-}(0,0)*{\otimes};(10,0)*{\newmoon}};(5,0)*{<};(-10,0)*{\LEFTcircle};(0,0)*{\otimes}**\dir{=};
(-10,-4)*{\scriptstyle 1};(0,-3.5)*{\scriptstyle 2};(10,-3.5)*{\scriptstyle 3};
\endxy$$
where $\I_\bu =\{1,3\}$ and $\I_\circ =\{2\}$. The simple roots are given  by
\[ 
\alpha_1=\epsilon,\quad \alpha_2=\epsilon-\frac{1}{2}(\delta_1+\delta_2),\quad \alpha_3=\delta_2 
\]
where $(\epsilon,\epsilon)=-2,\  (\delta_1,\delta_1)=2$ and $(\delta_2,\delta_2)=6$. Similarly to the type B case, one can establish the existence of the coideal subalgebra property, the quantum Iwasawa decomposition, and the quasi $K$-matrix in exactly the same way.

	\vspace{3mm}
\noindent{\bf Funding and Competing Interests.}

We are grateful to the referees for careful reading and corrections (especially with regard to Table 3). YS was partially supported by the Jefferson Dissertation Fellowship at University of Virginia in 2023-24. WW is partially supported by DMS--2401351. The results of this paper have been presented at Osaka Metropolitan University in February 2024. 

The authors have no competing interests to declare that are relevant to the content of this article.
The authors declare that the data supporting the findings of this study are available within the paper.

\end{document}